\documentclass[1p]{elsarticle}
\usepackage{amsfonts}

\usepackage{amsmath,amssymb,comment}
\usepackage{epsfig}
\usepackage{graphicx}
\usepackage{MnSymbol}

\usepackage{url}

\DeclareMathAlphabet{\mathpzc}{OT1}{pzc}{m}{it}
\catcode`@=11
\@addtoreset{equation}{section}
\catcode`@=12

\newtheorem{theorem}{Theorem}[section]

\newtheorem{claim}[theorem]{Claim}

\newtheorem{corollary}[theorem]{Corollary}

\newtheorem{definition}[theorem]{Definition}

\newtheorem{lemma}[theorem]{Lemma}

\newtheorem{proposition}[theorem]{Proposition}
\newtheorem{remark}[theorem]{Remark}

\newenvironment{proof}[1][Proof]{\textbf{#1.} }{\ $\Box$ \\}

\def\r{\varrho}
\newcommand{\Vr}{V_{\r}}
\newcommand{\ro}{\r_0}
\newcommand{\Vro}{V_{\ro}}

\newcommand\Hog[1]{\mathcal{H}^{#1}}

\def\a{\alpha}
\def\T{T}
\newcommand\TT[1]{\T^{#1}}

\def\A{A}
\def\B{B}
\def\ka{a}
\def\kb{b}
\def\kd{d}
\def\Ap{\A_{p}}
\def\Bp{\B_{p}}
\def\Bq{\B_{q}}
\def\ap{\ka_{p}}
\def\bp{\kb_{p}}
\def\cp{c_{p}}
\def\ddp{d_{p}}

\def\FT{P}
\newcommand\FR[1]{G_{#1}}

\def\CC{\mathcal{C}}
\def\OO{\mathcal{O}}
\def\oo{\mathpzc{o}}

\def\re{\text{\rm Re}\,}
\def\im{\text{\rm Im}\,}

\newcommand\Kl[1]{K^{#1}}
\newcommand\Klt[1]{\tilde{ K}^{#1}}

\def\k{\ell}
\def\lo{k_0}
\def\ri{\k_0}
\def\ril{\k_1}
\def\CIn{\ka_V}
\def\Id{\text{\rm Id}\,}

\def\gdf{r_*}
\def\gd{r_0}

\def\rvl{r_{\k}^{>}}
\def\ss{\nu}
\def\ll{\sigma}
\def\rfm{r^{\leq}}
\def\rv{r^{>}}

\def\NN{\mathbb{N}}
\def\RR{\mathbb{R}}
\def\ZZ{\mathbb{Z}}
\def\C{\mathbb{C}}

\def\FN{R^*}

\def\Ra{\mathcal{R}}

\newcommand\X[2]{\mathcal{X}^{#1}_{#2}}
\renewcommand\L[1]{\mathcal{L}^{#1}}
\renewcommand\S[1]{\mathcal{S}^{#1}}
\newcommand\FF{\mathcal{F}}
\newcommand\BB{\mathcal{B}}

\def\TFBFM{Theorem~2.2 in~\cite{BFM2015b}}
\def\TFBFMe{Theorem~2.2 in~\cite{BFM2015b} }
\def\TFBFMa{Theorem~2.7 in~\cite{BFM2015b}}
\def\TFBFMae{Theorem~2.7 in~\cite{BFM2015b} }
\def\TFBFMf{Theorem~2.8 in~\cite{BFM2015b}}

\def\rl{s}
\def\rx{r}

\def\DS{\Sigma_{\rl,\rx}}
\def\DSf{\Sigma_{\rl^{\leq},\rfm}}

\def\Dl{D_{\lambda}}
\def\Dx{D_{x}}

\def\ii{\text{i}}
\def\re{\text{\rm Re}\,}
\def\im{\text{\rm Im}\,}
\def\dd{\,d}

\newcommand\XD[3]{\mathcal{Y}^{#1,#2}_{#3}}

\newcommand\Klp[2]{K^{#1}_{#2}}
\newcommand\KMil[1]{K^{\leq #1}}
\newcommand\Kalg[1]{K^{(#1)}}

\newcommand\RMil[1]{R^{\leq #1}}

\newcommand\Rl[1]{ R^{#1}}

\newcommand\Yl[1]{Y^{#1}}
\newcommand\YMil[1]{Y^{\leq #1}}

\newcommand\El[2]{E^{#1}_{#2}}
\newcommand\Ef[1]{E_{#1}}
\newcommand\Em[1]{E^{>#1}}
\newcommand\Ew[2]{\widehat E^{>#1}_{#2}}

\newcommand\ME[1]{\overline{#1}}
\newcommand\ZE[1]{\widetilde{#1}}

\def\fpa{\varphi}
\def\fqa{\psi}
\def\Mx{M_{p}}
\def\My{M_{q}}

\def\df{\ell_*}

\newcommand\range{\text{range}}

\def\Dl{D_{\lambda}}
\def\Dx{D_{x}}
\def\Dz{D_{z}}

\begin{document}

\title{Invariant manifolds of parabolic fixed points (I).
Existence and dependence on parameters}

\author{Inmaculada~Baldom\'a}
\ead{immaculada.baldoma@upc.edu}
\address{Departament de Matem\`{a}tiques,\\
         Universitat Polit\`{e}cnica de Catalunya,\\
         Diagonal 647, 08028 Barcelona, Spain}

\author{Ernest~Fontich}
\ead{fontich@ub.edu}
\address{Departament de Matem\`{a}tiques i Inform\`atica,\\
        Universitat de Barcelona, \\ Gran Via 585,
        08007 Barcelona, Spain }

\author{Pau~Mart\'{\i}n}
\ead{p.martin@upc.edu}
\address{Departament de Matem\`{a}tiques,\\
         Universitat Polit\`{e}cnica de Catalunya,\\
         Ed.~C3, Jordi Girona 1--3, 08034 Barcelona, Spain}

\begin{abstract}
    In this paper we study the existence and regularity of stable manifolds
    associated to fixed points of parabolic type in the differentiable and analytic cases,
    using the parametrization method.

    The parametrization method relies on a suitable approximate solution of
    a functional equation. In the case of parabolic points, if the manifolds have dimension two
    or higher, in general this approximation cannot be
    obtained in the ring of polynomials but as a sum of homogeneous functions and it is given in~\cite{BFM2015b}.
    Assuming a sufficiently good approximation is found, here we provide an ``a posteriori'' result which gives a
    true invariant manifold close to the approximated one. In the
    differentiable case, in some cases, there is a loss of
    regularity.

    We also consider the case of parabolic periodic orbits of
    periodic vector fields and the dependence of the manifolds on
    parameters. Examples are provided.

    We apply our method to prove that in several situations, namely, related to the
    parabolic infinity in the elliptic spatial three body problem, these invariant
    manifolds exist and do have polynomial approximations.
\end{abstract}

\maketitle

\tableofcontents


\section{Introduction}

Parabolic fixed points of maps (or parabolic periodic orbits, in the
case of flows) appear in general as bifurcation points but they  are
also present for all values of the parameters in important problems.
For instance, the ``parabolic infinity'' in several instances of the
three body problem. See
\cite{Moser01,SimoL80,Robinson84,Robinson15,Xia92,GuardiaMS14,DelshamsKRS14}.

The purpose of this work is, given a map with a parabolic fixed
point, that is, a point  where the map is tangent to the identity,
to provide conditions under which the parabolic point possesses a
stable invariant set (which in general will not contain a
neighborhood of the fixed point) which can be parametrized as a
regular invariant manifold. This is the first part of a two papers
work, being~\cite{BFM2015b} the second. In the second one, we study
the existence of approximate solutions of the invariance equation
that the parabolic invariant manifold should satisfy. Here we are
concerned with the existence of the actual manifold.

The existence of invariant manifolds of parabolic fixed points and
their  regularity has been considered in
\cite{McGehee73,Easton84,Robinson84}, when the dynamical system is
analytic and the stable manifold set is one dimensional. Invariant
manifolds of parabolic fixed points with nilpotent linear part were
studied in~\cite{Cas92,Cas97,Fon99}. In~\cite{MeissL16} the authors
use the manifolds of a parabolic point as pieces of the boundary of
regions with regular and ergodic behavior respectively for a
specially chosen family of two dimensional symplectic maps. The case
of stable manifolds of higher dimension, but still in the analytic
category, was considered in~\cite{BF2004}. All these works share the
use of the graph transform method to obtain the parabolic invariant
manifold.

The problem of parabolic fixed points in the context of holomorphic
maps has also been studied in a completely different approach
by~\cite{Hakim98,Ecalle85}. See also the survey~\cite{Abate15}.

When the map is not analytic, but $\CC^k$, $1$-dimensional stable
manifolds of parabolic point have been studied in~\cite{BFdLM2007}.
In this work, unlike the previously cited ones, is used the
parametrization method
\cite{CabreFL03a,CabreFL03b,CabreFL05,HdL2006,HdL2007}. See
also~\cite{HaroetAl}.

The procedure here is as follows. Let $F:U\subset \RR^n\times \RR^m
\to \RR^n\times \RR^m$ be a map and assume $(0,0)\in U$ is parabolic
point, i.e.,  $F(0,0) = (0,0)$ and  $DF(0,0) = \Id$. Assume
furthermore certain conditions on the first non-vanishing nonlinear
terms which imply some ``weak contraction'' in the
$(x,0)$-directions and some ``weak expansion'' in the
$(0,y)$-directions, to be specified later. Even if our conditions
are more general and in fact do not always imply ``weak expansion"
in the $(0,y)$-directions, for the sake of simplicity of this
introduction, let us assume that there is this expansion. Then one
looks for an invariant stable manifold~$W^s$ of the origin as an
immersion $K:V\subset \RR^n\to \RR^n\times \RR^m$, which we call
parametrization of the manifold, with $K(0) = (0,0)$, $DK(0) =
(\Id,0)^{\top}$, $\range (K) = W^s$ and satisfying the
\emph{invariance equation}
\begin{equation}
\label{eq:inv_equation} F \circ K = K \circ R,
\end{equation}
where $R: V\to V$ is a reparametrization of the dynamics of~$F$
on~$W^s$. In general, $V$ is a domain which contains $0$ on its
boundary. The procedure to find such~$K$ and~$R$ has two steps.
First, find functions $K^{\le}$ and $R$ solving approximately the
invariance equation, that is, satisfying
\begin{equation}
\label{eq:inv_approx} F \circ K^{\le}(x) - K^{\le} \circ R (x) =
\OO (\|x\|^{\ell}),
\end{equation}
for some $\ell$ large enough, depending on the degree of the first
non-vanishing nonlinear terms of~$F$ at $(0,0)$. Once these
functions are obtained, the invariance equation can be rewritten as
a fixed point equation for a perturbation of~$K^{\le}$ and solved in
appropriate Banach spaces.

Of course, if the invariance equation does have solutions $K$ and
$R$, they will not be unique, since for any diffeomorphism $T:V \to
V$, the functions $K \circ T$ and $T^{-1} \circ R \circ T$ also
satisfy the same equation. The same claim holds for the approximate
invariant equation~\eqref{eq:inv_approx} if, for instance, $T(x) = x
+ o(\|x\|)$. The parametrization method aims to obtain the
``simplest'' parametrization (or the parametrization that provides
the ``simplest'' $R$).

There are two important reasons to use the parametrization method to
obtain the invariant manifolds of a parabolic fixed point. The first
one, from the theoretical point of view, is that is better suited to
deal with cases of finite differentiability than the graph transform
method since the operators involved are more regular. The second one
is related to the computation of the approximate solutions of the
invariance equation. From a computational point of view, it provides
a way to explicitly obtain such approximations. And reciprocally, if
one is able to produce functions $K^{\le}$ and $R$ that are
approximate solutions of the invariance equation, then there exists
a true solution close to the given approximation. This is a type of
\emph{a posteriori} argument (see
\cite{GonzalezLJV05,HdL2006,HdL2007,GonzalezHL14,FLS09a,FLS09b,
FLS15}).

The parametrization method is used in~\cite{CabreFL03a,CabreFL03b}
to find nonresonant manifolds of fixed points of maps in Banach
spaces. In such setting, the approximations $K^{\le}$ and $R$ can be
taken as polynomials. The degrees of~$K^{\le}$ and~$R$ depend on the
spectrum of~$DF(0,0)$. The homogeneous terms of these polynomials
are found recursively. The homogenous term of degree~$j$ must
satisfy a linear equation which depends on the terms of degree~$i$,
for $1\le i \le j-1$. In solving these equations, $K^{\le}$ and~$R$
play different roles and are not unique, even in the class of
polynomials. A possible criterium to determine them is to look for
the ``simplest'' polynomial $R$, in the sense that the majority of
its coefficients vanish. This fact only depends on the spectrum
of~$DF(0,0)$.

In the case when the origin is parabolic and $n=1$, in
\cite{BFdLM2007} it is shown that it is also possible to find
polynomials~$K^{\le}$ and~$R$ which are approximate solutions of the
invariance equation. Again, these polynomials are not determined
uniquely, but there is a choice in which~$R$ is the ``simplest''.
Its degree only depends on the degree of the first non-vanishing
term of the contracting part. A related result was obtained
in~\cite{BH08} where the Gevrey character of the manifolds is
established for analytic maps.

The situation changes drastically when one considers invariant
manifolds of parabolic points of dimension two or more. Although
these cases were successfully dealt in the analytical
context~\cite{BF2004} by means of the graph transform method, a
simple  computation shows that generically there are no polynomial
approximate solutions of the invariance equation. In the spirit of
the parametrization method, if it is not possible to find
approximate solutions, the fixed point part of the argument cannot
be carried on. We remark that this fact implies that, generically,
the invariant manifolds obtained in~\cite{BF2004}, which are
analytic outside the origin, do not have a polynomial approximation.

In the present paper, we deal with the actual existence of the
invariant manifold and we study its regularity and dependence on
parameters, assuming that a suitable approximate solution of the
invariance equation is known. In the companion
paper~\cite{BFM2015b}, we derive a method to find such
approximations and their regularity. However, since, in general,
these approximations are not polynomials but sums of homogeneous
functions of increasing degree, we reproduce in
Section~\ref{sectionalgorithm} the algorithm derived
in~\cite{BFM2015b} to obtain them. It should be remarked that, in
general, these homogeneous functions need not be rational functions.
We also remark that the conditions under which these approximations
can be found allow several characteristic directions in the domain
under consideration (see~\cite{Hakim98,Abate15}).

When considering parabolic points, one has to look at the first
non-vanishing homogenous terms of the Taylor expansion of the map at
the parabolic point. One looks for ``contracting'' and
 ``expanding'' directions (in certain subsets) in the dynamics generated by the polynomial
map obtained by truncating the Taylor expansion of the map at the
parabolic point at the lowest non vanishing order in each component.
We will assume that the degree of all the ``contracting'' directions
is $N$, the degree of all the ``expanding'' directions is $M$,
without assuming that $N=M$. The fact that $N\neq M$ has
consequences both at a formal level, when solving the approximate
invariance equations, and at the analytical level, when considering
the fixed point equation that provides the manifold. In particular,
the behavior and regularity at the origin of the formal
approximation and the invariant manifold depend on the relation
between $N$ and $M$.

We remark that, as it is often the case, the hypotheses to carry out
the fixed point procedure are milder than the ones required for
solving the approximate invariance equation. The reason is that to
solve the fixed point equation it is enough to start with an
approximate solution having an error of prescribed high enough order
depending of the first non-vanishing nonlinear terms. Of course,
some care is required to deal with the regularity of the involved
objects.

We include in our study the dependence on parameters of the
invariant manifolds, which is rather cumbersome but useful for the
applications. In particular, it allows to derive the analogous
statement for flows from the one for maps. This is performed
separately for the actual manifold, in the present paper, and for
the approximate solutions of the invariance equation, in the
companion paper. The dependence on parameters of the invariant
manifolds in the case that they are one dimensional and the map is
analytic is already done in~\cite{Guardia15}, where it was used to
find regular foliations of the invariant manifolds of some parabolic
cylinders.

As a side application of our method, we prove that, in several
instances  of the three body problem, namely in perturbations of the
restricted spatial elliptic three body problem, the ``parabolic
infinity'' is foliated with parabolic fixed points with stable
manifolds of dimension two \emph{that have polynomial approximation
at the origin}.  This fact is rather surprising, since to be able to
solve the approximate invariance equations in the ring of
polynomials, one obtains a larger number of obstructions than
coefficients at each order. Then, the fixed point machinery works at
any order and as a result one obtains the invariant manifolds of the
``parabolic infinity'' and their expansion at the origin.

The structure of the paper is as follows. In
Section~\ref{sectionhypothesis} we  present the setting and
hypotheses as well as two theorems of existence of invariant
manifolds for maps. In Section~\ref{sec:parameters}, we present the
result concerning the regularity with respect to parameters and in
Section~\ref{sec:flow} we deal with the results for flows. In
Section~\ref{sectionalgorithm} we describe the algorithm
from~\cite{BFM2015b} developed to compute the approximate solutions
of the invariance equation. In Section~\ref{sec:esrtbp} we apply the
algorithm to the elliptic spatial restricted three body problem to
obtain the invariant manifolds of the `` parabolic infinity''. In
Section~\ref{sectionexamples} we provide two examples that show that
our hypotheses are indeed necessary and that the loss of
differentiability can take place. We remark the differences between
one-dimensional and multidimensional parabolic manifolds. The rest
of the paper is devoted to the actual proofs of the results.

\section{Main results}

This section is devoted to present all the results of this work
related to the existence and regularity of parametrizations of
invariant sets. There are three settings we consider: the map case
in Section~\ref{sectionhypothesis}, the dependence on parameters in
the map case in Section~\ref{sec:parameters} and the periodic flow
case in Section~\ref{sec:flow}.

\subsection{The map case}\label{sectionhypothesis}
The first result is a posteriori type theorem which assures the
existence of an invariant manifold close to a sufficiently good
approximate solution of the invariance
equation~\eqref{eq:inv_equation}. Then we provide sufficient
conditions to ensure the existence of an invariant manifold by means
of the results in~\cite{BFM2015b} about approximate solutions of the
invariance equation, that is, solutions of~\eqref{eq:inv_approx}.

\subsubsection{Set up}
Let $U\subset \RR^{n}\times \RR^{m}$ be an open set such that $0\in
U$. We consider $\CC^r$ maps  $ F:U \to \RR^{n+m}$, with $r$ to be specified later,
of the form
\begin{equation}\label{defF}
 F(x,y) = \left ( \begin{array}{c} x+  p(x,y) +  f(x,y) \\ y+ q(x,y)+ g(x,y)\end{array}\right ),\qquad x\in \RR^n, \,y\in \RR^m,
\end{equation}
where $ p$ and $ q$ are homogeneous polynomials of degrees $N\geq 2$
and $M\geq 2$ respectively,  $f(x,y)= \OO(\Vert
(x,y)\Vert^{N+1})$ and $g(x,y)=\OO(\Vert (x,y)\Vert^{M+1})$. With these conditions, the origin is a parabolic fixed
point of $F$.

We introduce the constants
\begin{equation}\label{defeta}
L=\min \{N,M\},\qquad \eta = 1+N-L.
\end{equation}

We denote the projection onto a variable as a subscript, i.e. $X_x$,
and by $B_{\r}$ the open ball centered at the origin of radius
$\r>0$.

Given $V\subset \RR^n$ such that $0\in \partial V$ and $\r>0$, we
introduce the set
\begin{equation*}
\Vr = V \cap B_{\rho}.
\end{equation*}
We will consider sets~$V$ \emph{star-shaped with respect to~$0$},
\emph{i.e.}, $0\in \partial V$ and, for all $x \in V$ and $\lambda
\in (0,1)$, $\lambda x \in V$.

We define the stable set of $F$ over~$V$ associated to the origin
$0$ as:
\begin{equation*}
W^{{\rm s}}_V = \{(x,y) \in U : F_x^k(x,y)\in V,\; k\geq0 ,\;
F^{k}(x,y)\to 0 \;\text{as}\; k\to \infty\}
\end{equation*}
and its local version, when we restrict $V$ to the set $\Vr$:
\begin{equation}
\label{defstablemanifoldlocal} W^{{\rm s}}_{V,\r} = \{(x,y) \in U :
F_x^k(x,y)\in \Vr,\; k\geq0 ,\; F^{k}(x,y)\to 0 \;\text{as}\; k\to
\infty\}.
\end{equation}

Let $V\subset \RR^n$ be an open star-shaped with respect to~$0$ set.
Take $\r>0$, some norms in $\RR^n$ and $\RR^m$ and consider the
following constants:
\begin{equation}\label{defconstants}
\begin{aligned}
&\ap = - \sup_{ x\in \Vr} \frac{\Vert  x+  p( x,0) \Vert -\Vert
x\Vert }{\Vert  x\Vert^{N}},
&\qquad &\bp  =\sup_{ x\in \Vr} \frac{\Vert  p( x,0)\Vert}{\Vert  x\Vert^{N}}, \\
&\Ap= - \sup_{ x\in \Vr} \frac{\Vert \Id + D_x p( x,0) \Vert
-1}{\Vert  x\Vert^{N-1}},
&\qquad &\Bp = \sup_{ x\in \Vr} \frac{\Vert \Id -D_x p( x,0) \Vert -1}{\Vert  x\Vert^{N-1}}, \\
&
\Bq = -\sup_{ x\in \Vr} \frac{\Vert \Id -D_y q( x,0) \Vert-1}{\Vert  x\Vert^{M-1}}, &&\\
&\cp  =
\begin{cases}
\;\;\ap,  & \text{if}\;\; \Bq \leq 0,\\
\;\;\bp, &  \text{otherwise,}
\end{cases}
& \qquad &\ddp  =
\begin{cases}
\;\;\ap,  & \text{if}\;\; \Ap\leq 0,\\
\;\;\bp, &  \text{otherwise},
\end{cases}
\end{aligned}
\end{equation}
where the norms of linear maps are the corresponding operator norms.
We emphasize that all the previous constants depend on $\r$.
Nevertheless there are some straightforward relations among them.
\begin{lemma}\label{defKapwell}
The constants $\Ap, \Bq, \Bp, \ap$ and $\bp$ are finite. They
satisfy $|\ap|\leq \bp$, $\Bp\geq \Ap$, $\ap\geq \Ap/N$ and $\Bp\geq
N\ap>0$.

In addition, if $0<\overline{\r}\leq \r$ and denoting by
$\overline{\Ap},\overline{\Bp},
\overline{\Bq},\overline{\ap},\overline{\bp}$ the corresponding
constants for $\overline{\r}$, we have that
$$
\overline{\Ap} \geq \Ap,\;\; \overline{\Bp}\leq \Bp,\;\;
\overline{\Bq}\geq \Bq,\;\;\overline{\ap}\geq \ap,\;\;\;
\overline{\bp}=\bp.
$$
\end{lemma}

This lemma is proven in Section 3.1 of~\cite{BFM2015b} (in a slightly
different set up)

As usual for parabolic points, their invariant manifolds are defined
over a subset $V$ such that $0 \notin \text{int} V$. For this
reason, in order to study the regularity of the invariant manifold
at the origin, we define the following natural concept:
\begin{definition}
\label{difrem} Let $V\subset \RR^l$ be an open set with
$x_0\in \overline{V}$ and $f:V\cup\{x_0\}\subset \RR^l \to \RR^k$. We say
that $f$ is $C^1$ at $x_0$ if $f$ is $C^1$ in $V\cap
(B_{\epsilon}(x_0)\setminus\{x_0\})$, for some $\varepsilon>0$ and $
\lim_{x\to x_0, \; x\in V} Df(x) $ exists.
\end{definition}

We finally introduce a quantity related with the minimum
differentiability degree we require to $F$:
\begin{equation}\label{defri}
\ri:=N-1 + \frac{\Bp}{\ap}+\max \left
\{\eta-\frac{\Ap}{\ddp},0\right \}.
\end{equation}
Note that $\ri \ge 2N-1 \ge N+1$.

\subsubsection{A posteriori result}\label{results1}
Let $V\subset \RR^{n}$ be open, star-shaped with respect
to~$0$. Assume that there exist appropriate norms in $\RR^{n}$ and
$\RR^{m}$ and $\r>0$ small enough such that
\begin{enumerate}
\item[H1]
The homogenous polynomial $ p$ satisfies that $\ap>0$,
\item[H2]
$ q( x,0)=0$, for $x\in \Vr$ and
\begin{equation*}
\begin{aligned}
&\Bq>0, & \qquad &\text{if  } M<N,\\
& \Bq>-N\ap, & \qquad &\text{if  } M=N,
\end{aligned}
\end{equation*}
\item[H3]
There exists a constant $\CIn >0$ such that, for all $ x\in \Vr$,
\begin{equation*}
\text{dist}( x+p( x,0), (\Vr)^{c}) \geq \CIn \Vert  x\Vert^{N}.
\end{equation*}
\end{enumerate}
\begin{remark} It is easily checked that if hypotheses H1, H2 and H3 hold true for $\r>0$, they
also hold for any $0< \overline{\r}\leq \r$,
so that we will take $\r$ as small as we need.
\end{remark}
\begin{theorem}\label{maintheorem}
Let $F:U\subset \RR^{n+m} \to \RR^{n+m}$ be a $\CC^r$ map, (the case
$r=\infty$ is also included) of the form \eqref{defF} with $U$ an
open set such that $0\in U$.

Assume that, there exists an open set $V$ and $\ro>0$ such that:
\begin{enumerate}
\item[(a)] Hypotheses H1, H2 and H3 hold for $\ro>0$.
\item[(b)] The degree of differentiability satisfies $r> \ri$ with $\ri$ defined in \eqref{defri}.
\item[(c)] There exist $K^{\leq}:\Vro \to U$ and $R:\Vro \to \Vro$, $\CC^{\rfm}$ functions, for some $\rfm \ge 1$,
of the form
\begin{equation*}
\begin{aligned}
&\Delta K^{\leq}(x) :=K^{\leq}(x) -(x,0) =\OO(\Vert x \Vert^{2}),&\;\; &D^j \Delta K^{\leq}(x)=\OO(\Vert x \Vert^{2-j}),& \\
& \Delta R(x):=R(x) -x-p(x,0) =\OO(\Vert x \Vert^{N+1}) ,&\;\; &D^j
\Delta R(x)=\OO(\Vert x \Vert^{N+1-j}),&
\end{aligned}
\end{equation*}
for $0\leq j \leq \rfm$, satisfying the invariance equation up to
order $\ell$ for $\ri <\ell \leq r$, i.e.:
$$
F \circ K^{\leq} - K^{\leq} \circ R = \OO(\Vert x \Vert^{\ell}).
$$
\end{enumerate}
Then, there exists $\r>0$ small enough and a unique function
$K^{>}:\Vr\to U$ such that $K^{>}(x) = \OO(\Vert x
\Vert^{\ell-N+1})$ and $K=\Kl{\leq} + \Kl{>}$ satisfies the
invariance equation
\begin{equation}\label{invcondtheorem}
F\circ K= K\circ R.
\end{equation}
Moreover, $R^k(x) \to 0$ as $k\to\infty$, $K_x$ is invertible and,
as a consequence,
\begin{equation}\label{KWs}
\{K(x)\}_{ x\in (K_x)^{-1} (\Vr)} \subset W^{{\rm s}}_{V,\r}.
\end{equation}

Concerning regularity, the parametrization~$K$ and the
reparametrization~$R$ on $W^{{\rm s}}_{V,\r}$ are $\CC^1$ functions
at the origin in the sense of Definition~\ref{difrem}. Moreover,
they are $\CC^{\rv}$ functions on $\Vr$ according to the cases
\begin{enumerate}
\item[(1)] If $\Ap\geq \eta \ddp$, $\rv=\min\{r,\rfm\}$.
\item[(2)] If $\Ap<\eta \ddp$, $\rv=\min\{r,\gd ,\rfm\}$ with $\gd$ defined by
\begin{equation}\label{defgd}
\gd = \max\left \{\displaystyle{k\in \NN : \left( \eta-
\frac{\Ap}{\ddp} \right) }k <r-\frac{\Bp}{\ap}-N+1\right\}.
\end{equation}
\item[(3)] If $F\in \CC^{\infty}$, then $\rv=\rfm$, where the case $\rfm=\infty$ is also included.
\end{enumerate}
In addition, if $F,\Kl{\leq}$ and $R$ are real analytic, $\Ap>\bp$
and item (c) is true for $j=0$, then $K$ is also real analytic.
\end{theorem}

\subsubsection{Existence results of invariant manifolds}
As a corollary of Theorem~\ref{maintheorem} and the
work~\cite{BFM2015b} we can prove an existence result. We first
formulate the new set of hypotheses which are (as usual) slightly
stronger than the previous ones. They coincide with the ones assumed
in~\cite{BFM2015b} for the existence of approximated solutions of
the invariance equation~\eqref{eq:inv_equation}. We include them
here for completeness. We summarize the algorithm to find these
approximated solutions in Section~\ref{sectionalgorithm}.

Let $V\subset \RR^{n}$ be an open set such that $V\cup\{0\}$ is
convex. Assume that, with the appropriate norms in $\RR^{n}$ and
$\RR^{m}$, there exists $\r$ is small enough such that Hypothesis H3
is satisfied and
\begin{enumerate}
\item[H1']
The homogenous polynomial $ p$ satisfies that
\[
\ap>0.
\]
If $M>N$, we further ask $\Ap/\ddp>-1$.
\item[H2']
The homogenous polynomial $ q$ satisfies $q( x,0)=0$, for $x\in \Vr$
and
\begin{equation*}
\begin{aligned}
&\Bq>0, & \qquad &\text{if  } M<N,\\
& 2+\frac{\Bq}{\cp} >\max \left \{1 -\frac{\Ap}{\ddp} ,0
\right\},&\qquad & \text{if  } M=N.
\end{aligned}
\end{equation*}
\end{enumerate}

Unlike the hyperbolic case, as we claimed in Theorem~\ref{maintheorem}, here we can lose
differentiability  in
the case $\Ap <\eta \ddp$ even at points $x\in \Vr$ with $x\neq0$. In fact,
the formal approximation is only
$\CC^{\gdf}$ when $\Ap<\ddp$ and $M\geq N$, being~$\gdf$:
\begin{equation*}
\gdf = \begin{cases} \max\left \{\displaystyle{k\in \NN : \left( 1-
\frac{\Ap}{\ddp} \right) }
k <2+\frac{\Bq}{\cp}\right\}, & \quad \text{if  } M=N, \\
\max\left \{\displaystyle{k\in \NN : \left( 1- \frac{\Ap}{\ddp}
\right)} k < 2 \right \},& \quad \text{if  } M>N.
\end{cases}
\end{equation*}
See~\cite{BFM2015b}.

The existence result is as follows:
\begin{corollary} \label{maintheorem2}
Let $F:U\subset \RR^{n+m} \to \RR^{n+m}$ be a $\CC^r$ map, of the
form \eqref{defF}. Assume that, for some $\ro>0$, $r>\ri$ and that
hypotheses H1', H2' and H3 are satisfied in an open star-shaped with respecto to~$0$ set $V$.

Then, there exist $\r>0$ small enough and maps $K:\Vr \to U$ and
$R:\Vr\to \Vr$ solutions of the invariance
equation~\eqref{invcondtheorem} satisfying~\eqref{KWs}.

In addition, $K=\Kl{\leq}+ \Kl{>}$ with $\Kl{\leq}$ and $R$ provided
by \TFBFM.

The parametrization $K$ and the reparametrization $R$ on $W^{{\rm
s}}_{V,\r}$ are only $\CC^1$ functions at the origin restricting
them to the set $\Vr$ and they are $\CC^{\rv}$ functions on $\Vr$
and $\rv$ takes the values:
\begin{enumerate}
\item[(1)] If $\Ap\geq \eta \ddp$, $\rv=r$.
\item[(2)] If either $\ddp \leq \Ap<\eta \ddp$ or $M<N$, $\rv=\min\{r,\gd\}$ with $\gd$ defined in \eqref{defgd}.
\item[(3)] If $\Ap<\ddp$ and $M\geq N$, $\rv=\min\{r,\gd,\gdf\}$.
\item[(4)] If $F\in \CC^{\infty}$ and $\Ap\geq \ddp$, then $\Kl{>}\in \CC^{\infty}$.
\end{enumerate}

Moreover, if $F$ is real analytic and $\Ap>\bp$, $K$ is also real
analytic.

Finally, substituting H1' and H2' by the new conditions $\Ap>0$ and
$\Bq>0$, we have that $W^{{\rm s}}_{V,\r}=\{K(x)\}_{ x\in (K_x)^{-1}
(\Vr)}$.
\end{corollary}
\begin{proof}
Obviously H1' implies H1. It remains to check that when $M=N$, the
condition in H2' implies that $\Bq>-N\ap$. This is immediate if $\Bq\geq 0$.
When $\Bq<0$, H2' implies $2\ap +\Bq>0$ and
hence $\Bq>-2\ap \geq -N\ap$.

Now we set a good enough initial approximation of the invariant
manifold $K$ by means of \TFBFM.

We take $\k\in \NN$ such that $\ri<\k\leq r$ with $\ri$ introduced
in~\eqref{defri} and we decompose our map $ F$ into
\begin{equation}
\label{def:FdecompPG}  F(x,y)=\FT(x,y) + \FR{\k}(x,y),
\end{equation}
where $\FT$ is the Taylor expansion of $ F$ up to degree~$\k-1$ and
$\FR{\k}(x)=\oo(\Vert x \Vert^{\k-1})$. In fact, since $\k\leq r$,
we actually have $\FR{\k}(x)=\OO(\Vert x \Vert^{\k})$.  We
apply \TFBFMe to $\FT$ to obtain $\Kl{\leq}$ and $ R$ such that
\begin{equation}\label{eqFTKlamain}
\FT \circ \Kl{\leq} - \Kl{\leq} \circ  R =\TT{\k},\qquad
\TT{\k}(x)=\oo(\Vert x \Vert^{\k-1}).
\end{equation}
Moreover, both $\Kl{\leq}$ and $R$ are sums of homogeneous functions
satisfying that $\Delta \Kl{\leq} :=\Kl{\leq}(x)-(x,0)=\OO(\Vert x
\Vert)^{2})$ and $\Delta R(x):R(x) - x - p(x,0)=\OO(\Vert x
\Vert^{N+1})$. By \TFBFM, $\Kl{\leq}$ and $R$ are analytic functions
if $\Ap>\bp$, $\CC^{\infty}$ functions if $\Ap=\bp$ and $\CC^{\gdf}$
if $\Ap<\bp$, therefore, $F, \Kl{\leq}$ and $R$ are $\CC^{r}$
functions if $\Ap\geq \bp$ and $\CC^{\min\{r,\gdf\}}$ functions
otherwise. We use the symbol $\rfm$ to denote the degree of differentiability in
each case.

Since $\FT$ is a polynomial, the remainder $\TT{\k}(x) = \OO(\Vert x
\Vert^{\k})$ is also a finite sum of homogeneous functions.
Therefore, using that the derivative of a homogeneous function of
degree $j$ is also a homogeneous function of degree $j-1$, we have
that, for any $0\leq j\leq \rfm$,
$$
D^j\Delta \Kl{\leq}(x) =\OO(\Vert x \Vert^{2-j}),\; \;\;\; D^j
\Delta  R(x) =\OO(\Vert x \Vert^{N+1-j}),\;\;\;\; D^j \TT{\k} (x)
=\OO(\Vert x \Vert^{\k-j}).
$$
Therefore, we are under the conditions of Theorem~\ref{maintheorem}
 which implies the stated existence and the
regularity in the present results.

The last statement follows from Theorem~3.1 in~\cite{BF2004} which
states that the stable set $W^{{\rm s}}_{V,\r}$ defined
in~\eqref{defstablemanifoldlocal} is the graph of a Lipschitz
function. Since $K_x(x) = x + O(\Vert x \Vert^2)$, it is invertible and the
result follows immediately since the new conditions
$\Ap,\Bq>0$ imply the hypotheses of the results in~\cite{BF2004}.
\end{proof}

Now we state a corollary from Theorem~\ref{maintheorem} and \TFBFMa.
\begin{corollary}\label{corollary2}
Assume the conditions in Corollary~\ref{maintheorem2} and take $\k$
such that $\ri < \k\leq r$. For $j=2,\cdots, \k-N$, let $K_{x}^j
:\Vr\to  \RR^{n}$ be $\CC^{\rv}$ homogeneous functions of degree
$j$. Denote
$$K_x^{*}(x)= x +\sum_{j=2}^{\k-N}K_x^j(x).$$

Then there exists $\FN:\Vr\to \RR^{n}$, a finite sum of $\CC^{\rv}$
homogeneous functions of order less than $\k-1$, of the form $\FN(x)
- x - p(x,0)=\OO(\Vert x \Vert^{N+1})$ such that for any $\CC^{\rv}$
function $\Delta R :\Vr \to \RR^{n}$ with $\Delta R (x) = \OO(\Vert
x \Vert^{\k})$ there exists a $\CC^{\rv}$ function $K$ satisfying
the invariance equation \eqref{invcondtheorem} with $R=\FN+\Delta
R$ and $ K_x(x) -K_x^{*}(x) =\OO(\Vert x \Vert^{\k-N+1})$.
\end{corollary}
\begin{proof}
We proceed as in the proof of Corollary~\ref{maintheorem2}
decomposing $F$ as in~\eqref{def:FdecompPG} and applying \TFBFMae
instead of \TFBFMe which assures the existence of $\Kl{\leq}$ and
$\FN$ satisfying the invariance equation~\eqref{eqFTKlamain} up to
order $\k$. Moreover, $\Kl{\leq}_x(x) - K_x^{*}(x) =\OO(\Vert x
\Vert^{\k-N+1})$. Since $\Delta R(x)=\OO(\Vert x
\Vert^{\k})$, we have that $\Kl{\leq} (\FN (x) + \Delta R(x)) =
\Kl{\leq}\circ \FN (x)  + \OO(\Vert x \Vert^{\k})$ and, consequently,
writing $R=\FN + \Delta R$,
$$
F\circ \Kl{\leq}(x) - \Kl{\leq} \circ R (x)= \OO(\Vert x
\Vert^{\k}).
$$
Applying Theorem~\ref{maintheorem}, we get the result.
\end{proof}

\subsection{Dependence on parameters}
\label{sec:parameters} In this section we deal with the dependence
on parameters of the parametrization~$K$ and the reparametrization $R$
provided by Theorem~\ref{maintheorem} and
Corollary~\ref{maintheorem2}.
\subsubsection{Set up}
\label{sec:setup_parameters} Let $\Lambda\subset \RR^{n'}$ be an
open set of parameters and $U\subset \RR^{n}$ be an open set. We
consider $\CC^r$ maps $F: U\times \Lambda\to \RR^{n+m}$ having the
form \eqref{defF} for any $\lambda \in \Lambda$, namely:
\begin{equation}\label{defFparam}
 F(x,y,\lambda) = \left ( \begin{array}{c} x+  p(x,y,\lambda) +  f(x,y,\lambda) \\
y+ q(x,y,\lambda)+ g(x,y,\lambda)\end{array}\right ),\qquad
(x,y,\lambda)\in \RR^n \times\RR^m\times \RR^{n'},
\end{equation}
where $p,q$ are homogeneous polynomials for any fixed $\lambda$ of
degree $N,M\geq 2$ respectively and $f(x,y,\lambda) =\OO(\Vert (x,y)\Vert^{N+1})$,
$g(x,y,\lambda) =\OO(\Vert (x,y)\Vert^{M+1})$ uniformly
in $\lambda$.

In this context, the constants introduced in~\eqref{defconstants},
\eqref{defri} and Hypothesis H3, depend on~$\lambda$. We denote this
dependence by a superindex, for instance
$\Ap^{\lambda},\ri^{\lambda}$, etc. We redefine the constants
(independent of $\lambda$) $\Ap,\Bp,\ap,\bp,\Bq,\CIn, \cp,\ddp,\ri$
by
\begin{equation}\label{defconstantslambda}
\begin{aligned}
&\Ap =\inf_{\lambda \in \Lambda} \Ap^{\lambda},\qquad \ap=\inf_{\lambda \in \Lambda} \ap^{\lambda},
\qquad \Bq=\inf_{\lambda \in \Lambda} \Bq^{\lambda}, \\
&\Bp =\sup_{\lambda \in \Lambda} \Bp^{\lambda},\qquad \bp=\inf_{\lambda \in \Lambda} \ap^{\lambda},
\qquad \CIn = \inf_{\lambda \in \Lambda} \CIn^{\lambda}, \\
&\cp  =
\begin{cases}
\;\;\ap,  & \text{if}\;\; \Bq \leq 0,\\
\;\;\bp, &  \text{otherwise,}
\end{cases}
 \qquad \ddp  =
\begin{cases}
\;\;\ap,  & \text{if}\;\; \Ap\leq 0,\\
\;\;\bp, &  \text{otherwise,}
\end{cases} \\
&\ri = N-1 + \frac{\Bp}{\ap}+\max \left
\{\eta-\frac{\Ap}{\ddp},0\right \}.
\end{aligned}
\end{equation}
\begin{lemma}\label{lemmaconstantslambda}
If the conditions in H1, H2, H1', H2' and H3 hold true for the
constants $\Ap, \Bp,\ap,\bp,\Bq, \cp,\ddp, \CIn$, they are also true
for $\Ap^{\lambda},
\Bp^{\lambda},\ap^{\lambda},\bp^{\lambda},\Bq^{\lambda},
\cp^{\lambda},\ddp^{\lambda},\CIn^{\lambda}$ for any $\lambda \in
\Lambda$.

In addition $\ri^{\lambda}\leq \ri$.
\end{lemma}
The proof of this lemma is straightforward from the definitions.

The differentiability class we work in  was used in
\cite{CabreFL03b} and is the one considered in~\cite{BFM2015b} for
the approximate solutions. For any $\rl,\rx\in (\ZZ^+)^2$, we define
the set
\begin{equation*}
\DS = \big \{ (i,j) \in (\ZZ^+)^2 : i+j\leq \rx+\rl,\; i\leq \rl\big
\}
\end{equation*}
and for an open set $\mathcal{U}\subset \RR^l \times \RR^{n'}$, the
function space
\begin{equation}
\label{defCDS}
\CC^{\DS} = \big \{ f: \mathcal{U}\to \RR^{k} \;:\;
\forall
(i,j)\in \DS, \;D_{\mu}^i \Dz^j f \;
\text{exists, is continuous and bounded} \big\}.
\end{equation}
Here $D_{\mu}$ and $\Dz$ means the derivative with respect to $\mu$
and $z$ respectively. We also denote by
$$
\CC^{\Sigma_{s,\omega}} = \big \{ f: \mathcal{U}\to \RR^{k} \;:\;
f(\cdot,\mu) \;\text{ is analytic and } \; f\in \CC^{s}\}.
$$

We note that $\CC^r \subset \CC^{\Sigma_{r,r}}$.

\subsubsection{Dependence on parameters results}
Note that assuming that the conditions in both
Theorem~\ref{maintheorem} and Corollary~\ref{maintheorem2} are
satisfied for any $\lambda \in \Lambda$, we obtain the existence of
$K, R$ solutions of the invariance equation
\begin{equation}\label{inv:equation:param}
F(K(x,\lambda),\lambda) = K(R(x,\lambda),\lambda).
\end{equation}
To have regularity with respect to $\lambda$ we need
to impose some uniformity conditions.

Let $V$ an open set as in Section~\ref{results1} and $\r>0$. We
rewrite H1, H2 and H3 to become uniform with respect to $\lambda\in
\Lambda$ and we add an extra condition:
\begin{enumerate}
\item[H$\lambda$] The constants $\ap,\CIn>0$. Moreover $q(x,0,\lambda)=0$ for $(x,\lambda) \in \Vr \times \Lambda$ and
either $\Bq>0$ if $M>N$ or $\Bq>-N\ap$ if $M=N$.
\item[HP]$D_z^j f(x,y,\lambda) = \OO(\Vert (x,y)\Vert^{N+1-j})$  and $D_z^j g(x,y,\lambda)= \OO(\Vert (x,y)\Vert^{M+1-j})$  uniformly in $\Lambda$ with $z=(x,y)$ and $j=0,1$.
\end{enumerate}

We introduce
\begin{equation}\label{defril}
\ril:=N-1 + \frac{\Bp}{\ap}+(\eta-1).
\end{equation}

\begin{theorem}\label{maintheoremparam}
Let $F\in \CC^{\DS}$ be a map of the form \eqref{defFparam}. Let
$\ro>0$ be such that Hypotheses H$\lambda$, HP hold true and $\rx
>\max\{\ri,\ril\}$, $s\ge 0$.

Assume that there exist $\Kl{\leq}:\Vro\times \Lambda \to U$ and
$R:\Vro\times \Lambda \to \Vro$ such that
\begin{enumerate}
\item [(a)] $\Kl{\leq},R \in \CC^{\DSf}$.
\item [(b)] For $(i,j)\in \DSf$, uniformly over $\Lambda$,
\begin{equation*}
\begin{aligned}
&\Delta K^{\leq}(x,\lambda) :=K^{\leq}(x,\lambda) -(x,0) =\OO(\Vert x \Vert^{2}), &\Dl^i \Dx^j \Delta K^{\leq}(x,\lambda)=\OO(\Vert x \Vert^{2-j}),& \\
& \Delta R(x,\lambda):=R(x,\lambda) -x-p(x,0,\lambda) =\OO(\Vert x
\Vert^{N+1}),  &\Dl^i \Dx^j \Delta R(x,\lambda)=\OO(\Vert x
\Vert^{N+1-j}).&
\end{aligned}
\end{equation*}
\item [(c)] The invariance equation~\eqref{inv:equation:param} is satisfied up to order $\ri<\k \leq r$:
$$
F(\Kl{\leq}(x,\lambda),\lambda) - \Kl{\leq}(R(x,\lambda),\lambda) =
\OO(\Vert x \Vert^{\k}), \;\;\;\text{uniformly for}\;\; \lambda\in
\Lambda.
$$
\end{enumerate}

Then the unique function $\Kl{>}:\Vr \times \Lambda \to \RR^{n+m}$ found in
Theorem~\ref{maintheorem} belongs to $\CC^{\Sigma_{\rl^>,\rv}}$
where $\rl^>$ and $\rv$ have the following values according to the
cases
\begin{enumerate}
\item[(1)] If $\Ap\geq \ddp \eta$, $\rv=\min \{r,\rfm\}$ and $s^{>}\leq \min\{s,s^{\leq}\}$ satisfies
\begin{equation}\label{s>cond1}
s^> (\eta -1) < \rx-\frac{\Bp}{\ap}-N+1.
\end{equation}
\item[(2)] If $\ddp< \Ap \leq \ddp \eta$, then $\rv\leq  \min\{r,\rfm\}$, $s^> \leq \min \{s,s^{\leq}\}$ and
\begin{equation}\label{s>cond2}
\rx-\frac{\Bp}{\ap}-N+1 -\rv \left (\eta- \frac{\Ap}{\ddp}\right ) >
s^> (\eta-1).
\end{equation}
\item[(3)] If $\Ap < \ddp$, then $\rv\leq  \min\{r,\rfm\}$, $s^> \leq \min \{s,s^{\leq}\}$ and
\begin{equation}\label{s>cond3}
\rx-\frac{\Bp}{\ap}-N+1 -\rv \left (\eta- \frac{\Ap}{\ddp}\right ) >
s^> \left (\eta-\frac{\Ap}{\ddp}\right ).
\end{equation}
\item[(4)] If $F\in  \CC^{\Sigma_{s,\infty}}$, then $\rv = \rfm$ and $\rl^> = \rl^{\leq}$.
\end{enumerate}

Finally, if either $F, \Kl{\leq}$ and $R$ are real analytic or they
belong to $\CC^{\Sigma_{s,\omega}}$ and $\Ap>\bp$, then $K^>$ is
either real analytic if item (b) holds true for $i=j=0$ or $K^> \in
\CC^{\Sigma_{s,\omega}}$ if item (b) holds true for $j=0$
respectively.

\end{theorem}

To finish this section, we formulate an existence result as a
corollary of Theorem~\ref{maintheoremparam} and Theorem~2.7
in~\cite{BFM2015b} which includes the regularity with respect to
parameters of the approximate solutions. The following new condition is necessary
to ensure the existence of solutions of the invariance
equation~\eqref{inv:equation:param} for any value of
$\lambda\in\Lambda$:
\begin{enumerate}
\item [H$\lambda$']  $\ap,\CIn>0$, $q(x,0,\lambda)\equiv 0$ and the conditions in
hypotheses H1', H2' are satisfied for the constants $\Ap, \ddp, \Bq,
\cp$ redefined in~\eqref{defconstantslambda}.
\end{enumerate}
As we claim in Lemma~\ref{lemmaconstantslambda}, we have that H1',
H2' and H3 are satisfied if H$\lambda$' holds true. Therefore, by
the existence Corollary~\ref{maintheorem2} there exist $K$ and $R$
satisfying the invariance equation~\eqref{inv:equation:param}.
Moreover, by construction, $K=\Kl{\leq} + K^{>}$ with $\Kl{\leq}$
provided by Theorem~2.7 in~\cite{BFM2015b}.

\begin{corollary}\label{corollary:param}
Let $F\in \CC^{\DS}$ be a map of the form~\eqref{defFparam}. Assume
that there exists $\ro>0$ such that H$\lambda$' holds true.

\begin{itemize}
\item Parametric version of Corollary~\ref{maintheorem2} :
The solutions $K:\Vr \times \Lambda \to \RR^{n+m}$, $R:\Vr \times
\Lambda \to \Vr$ of the invariance equation provided by Corollary~\ref{maintheorem2} belong to
$\CC^{\Sigma_{s^>,\rv}}$ with $s^>$ and $\rv$ satisfying
\begin{enumerate}
\item[(1)] If $\Ap \geq \ddp \eta$, $\rv=r$ and $s^>\leq s$ satisfying~\eqref{s>cond1}.
\item[(2)] If $\ddp \leq \Ap <\eta \ddp$ or $M<N$, $\rv\leq r$, $s^>\leq s$ satisfying~\eqref{s>cond2}.
\item[(3)] If $\Ap<\ddp$ and $M\geq N$, $\rv\leq r$, $s^>\leq s$, $\rv+s^> \leq \gdf$
satisfying~\eqref{s>cond3}.
\item[(4)] If $F\in \CC^{\Sigma_{s,\infty}}$ and $\Ap\geq \ddp$, then $\rv =\infty$ and $s^>=s$.
\end{enumerate}

Moreover, if either $F$ is real analytic or it belongs to
$\CC^{\Sigma_{s,\omega}}$ and $\Ap>\bp$, then $K$ is either real
analytic or $K \in \CC^{\Sigma_{s,\omega}}$ respectively.

\item Parametric version of Corollary~\ref{corollary2}:
Let $K_{x}^j :\Vr\times \Lambda\to  \RR^{n}$ be
$\CC^{\Sigma_{s^>,\rv}}$ homogeneous functions of degree~$j$ with
respect to $x$. We introduce $K_x^{*}(x,\lambda)= x
+\sum_{j=2}^{\k-N}K_x^j(x,\lambda)$ as in
Corollary~\ref{corollary2}.

Then, the function $\FN:\Vr \times \Lambda \to \RR^n$ provided by
Corollary~\ref{corollary2} belongs to $\CC^{\Sigma_{s^>,\rv}}$.
Moreover, if $\Delta R :\Vr\times \Lambda \to \RR^{n}$ with $\Delta
R (x,\lambda) = \OO(\Vert x \Vert^{\k})$ uniformly in $\lambda\in
\Lambda$, belongs to $\CC^{\Sigma_{s^>,\rv}}$, then the function $K$
satisfying the invariance equation~\eqref{inv:equation:param} for
$R=\FN+\Delta R$ given in Corollary~\ref{corollary2} also belongs to
$\CC^{\Sigma_{s^>,\rv}}$.
\end{itemize}
\end{corollary}
\begin{proof} For any fixed $\lambda_0\in \Lambda$, the existence and uniqueness of $K(x, \lambda_0)=\OO(\Vert x \Vert^{\ell-N+1})$
satisfying the invariance equation~\eqref{inv:equation:param} is
guaranteed by Corollary~\ref{maintheorem2}. To obtain the regularity
with respect to the parameter we have to apply
Theorem~\ref{maintheoremparam}. To do so we need to discuss
Hypothesis HP. Since for any $\lambda\in \Lambda$, $F$ has the form
in~\eqref{defF}, we have that $D f(x,y,\lambda) = \OO(\Vert
(x,y)\Vert^{N})$ and $Dg(x,y,\lambda)=\OO(\Vert (x,y)\Vert^M)$ but
the bounds are not necessarily uniform in $\lambda$. Nevertheless,
by continuity, for any $\lambda_0\in \Lambda$ there exists an open
ball centered at $\lambda_0$, $B_{\rho_0}(\lambda_0) \subset
\RR^{n'}$, in such a way that HP is satisfied when we restrict the
domain of $\lambda$ to $\Lambda_{\lambda_0} = \Lambda \cap
B_{\rho_0}(\lambda_0)$. In addition, restricting $\rho_0$ if
necessary, we can get the approximate solutions $\Kl{\leq}, R$
satisfying items (a), (b) and (c) in Theorem~\ref{maintheoremparam},
that is, with uniform bounds in $\lambda \in \Lambda_{\lambda_0}$.

In conclusion, $K\in \CC^{\Sigma_{s^{>},\rv}}$ with $(x,\lambda)\in
\Vr\times \Lambda_{\lambda_0}$. Since $K(\cdot, \lambda)$ is the
unique solution of~\eqref{inv:equation:param} of order $\OO(\Vert x
\Vert^{\ell-N+1})$, $K\in \CC^{\Sigma_{s^{>},\rv}}$ in the full
domain $(x,\lambda)\in \Vr\times \Lambda$.
\end{proof}
\subsection{Existence results for invariant manifolds. The flow case} \label{sec:flow}

We deduce the analogous result to Corollary~\ref{maintheorem2} in
the case of time periodic flows, that is, in the case of a flow with
a parabolic periodic orbit. To study invariant objects associated to
periodic orbits of vector fields (in our case invariant manifolds),
one possibility is to consider a Poincar\'e map in a section
transversal to the orbit and then apply the results for fixed points
of maps. In this way, one gets the invariant manifolds $W^{{\rm s},
{\rm u}}$ of the Poincar\'e map and, from them, the invariant
manifolds of the periodic orbit by considering all the solutions
starting in $W^{{\rm s}, {\rm u}}$. Nevertheless this approach has a
drawback: in the applications, it is not easy to compute the
Poincar\'e map. Hence, if one wants to compute effectively the invariant
manifolds, it is better to have a statement already adapted to the
vector field itself.

To shorten the exposition we deal directly with the parametric case.
Let $U\subset \RR^{n+m}$ be an open neighborhood of the origin,
$\Lambda\subset \RR^{n'}$ a set of parameters and $ X:U\times\RR
\times  \Lambda\to \RR^{n+m}$ a $T$-periodic vector field:
\begin{equation}
\label{XTperiodic} \dot{z}=  X(z,t,\lambda) ,\qquad X(z,t+T,\lambda)
= X(z,t,\lambda)
\end{equation}
with $z=(x,y)\in U$
having the form
\begin{equation}
\label{defX}
 X(z,t,\lambda) =  X(x,y,t,\lambda) = \left ( \begin{array}{c}  p(x,y,\lambda) +  f(x,y,t,\lambda) \\
 q(x,y,\lambda)+ g(x,y,t,\lambda)\end{array}\right ),
\end{equation}
where $p$, $q$, $f$ and $g$ are as in Section~\ref{sec:setup_parameters}.
We have this form after having translated the parabolic orbit to the origin.

Let $\varphi(t;t_0, x,y,\lambda)$ be the flow of~\eqref{XTperiodic}.
Given a subset $V\subset \RR^n$, we define the stable set of the
origin over~$V$:
\begin{equation*}
W^{{\rm s}}_V = \{(x,y) \in U : \varphi_x(t;t_0, x,y,\lambda) \in
V,\; t\geq 0 ,\; \varphi(t;t_0, x,y,\lambda)\to 0 \;\text{as}\; t\to
\infty\}
\end{equation*}
and its local version, when we restrict $W^{{\rm s}}_V$ to the open
ball $B_{\r}$:
\begin{equation*}
W^{{\rm s}}_{V,\r} = \{(x,y) \in U : \varphi_x(t;t_0, x,y,\lambda)
\in \Vr,\; t\geq 0 ,\; \varphi(t;t_0, x,\lambda)\to 0 \;\text{as}\;
t\to \infty\}.
\end{equation*}

In the case of flows, a parametrization~$K(x,t,\lambda)$  is
invariant by the flow if there exists a vector
field~$Y(x,t,\lambda)$ such that
\begin{equation}\label{homequationflow}
X(K(x,t,\lambda),t,\lambda) = \Dx K(x,t,\lambda) Y(x,t,\lambda) +
\partial_t K(x,t,\lambda)
\end{equation}
or, equivalently
\begin{equation}
\label{invcondtheoremflow} \varphi(u;t,K(x,t,\lambda),\lambda) =
K(\psi(u;t,x,\lambda),u,\lambda),\qquad \forall u\geq t, \;\;
\forall (x,\lambda)\in \Vr \times \Lambda,
\end{equation}
where $\varphi$ and $\psi$ are the flows of the vector fields~$X$
and $Y$, respectively.

In this section, we will write that a function $f$ belongs to
$\CC^{\DS}$ if it satisfies definition~\eqref{defCDS} with $z=(x,y)$
and $\mu=(\lambda,t)$.

\begin{theorem}\label{maintheoremflow}
Let $ X\in \CC^{\DS}$ be a vector field of the form \eqref{defX}.
Assume that  Hypotheses H$\lambda$ and
HP hold true for some $\ro>0$ and $\rx >\max\{\ri,\ril\}$.
Assume also that there exist $\Kl{\leq}:\Vro\times \RR/(T\ZZ) \times \Lambda \to U$
and $Y:\Vro\times \Lambda \to \Vro$ such that
\begin{enumerate}
\item [(a)] $\Kl{\leq}, Y \in \CC^{\DSf}$, for some $s^{\le}, r^{\le}
\ge 1$.
\item [(b)] For $(i,j)\in \DSf$, uniformly over $\Lambda$,
\begin{equation*}
\begin{aligned}
&\Delta K^{\leq}(x,t,\lambda) :=K^{\leq}(x,\lambda) -(x,0) =\OO(\Vert x \Vert^{2}), &\Dl^i \Dx^j \Delta K^{\leq}(x,t,\lambda)=\OO(\Vert x \Vert^{2-j}),& \\
& \Delta Y(x,\lambda):=Y(x,\lambda) -x-p(x,0,\lambda) =\OO(\Vert x
\Vert^{N+1}),  &\Dl^i \Dx^j \Delta Y(x,\lambda)=\OO(\Vert x
\Vert^{N+1-j}).&
\end{aligned}
\end{equation*}
\item [(c)] The invariance equation~\eqref{homequationflow} is satisfied up to order $\ell$, $\ri<\k \leq r$:
\begin{equation}
\label{HIKlflow} X(\Kl{\leq}(x,t,\lambda),t,\lambda) - \Dx \Kl{\leq}
(x,t,\lambda) Y(x,\lambda) -\partial_t \Kl{\leq} (x,t,\lambda)=
\OO(\Vert x \Vert^{\k}),
\end{equation}
uniformly in $\lambda\in \Lambda$.
\end{enumerate}

Then, there exists $\r>0$ small enough and a unique function
$K^{>}:\Vr \times \RR/(T\ZZ) \times \Lambda\to U$ such that
$K^{>}(x,t,\lambda) = \OO(\Vert x \Vert^{\ell-N+1})$ uniformly in
$(t,\lambda)$ and $K=\Kl{\leq} + \Kl{>}$ satisfies the invariance
equation~\eqref{homequationflow} with the prescribed vector
field~$Y$ (or, equivalently, \eqref{invcondtheoremflow} with
$\psi(u;t,x,\lambda)$ the flow of $\dot{x}=Y(x,\lambda)$).

Moreover, since $\psi(u;t,x,\lambda) \to 0$ as $u\to \infty$, and
$K_x$ is invertible for any fixed $(t,\lambda)$, we have
\begin{equation}
\label{KWsflow} \{K(x,t,\lambda)\}_{ x\in \Vr\times \RR \times
\Lambda} \subset W^{{\rm s}}_{V,\r}.
\end{equation}

Concerning the regularity of $K$, we have the same results as the
ones stated in Theorem~\ref{maintheoremparam}.
\end{theorem}

To finish this section we formulate the existence result for the
flow case based on the approximated solutions provided in~\cite{BFM2015b}. The proof follows the same lines as the proof of
Corollary~\ref{maintheorem2}
\begin{corollary}
Let $ X\in \CC^{\DS}$ be a vector field of the form \eqref{defX}.
Assume that there exists $\ro>0$ such that Hypotheses H$\lambda$'
and HP hold true and $\rx >\max\{\ri,\ril\}$.

Then, there exist $\r>0$ small enough, a map $K:\Vr\times \RR/(T\ZZ) \times
\Lambda \to U$ and a vector field $Y:\Vr\times \RR \to \Vr$
solutions of the invariance equation~\eqref{homequationflow}
satisfying~\eqref{KWsflow}.

In addition, $K=\Kl{\leq}+ \Kl{>}$ with $\Kl{\leq}$ and $Y$ provided
by \TFBFMf.

The parametrization $K$ and the vector field $Y$ are $\CC^1$
functions at the origin in the sense of Definition~\ref{difrem}. The
regularity on $\Vr \times \RR \times \Lambda$ is the same as the one
stated in Corollary~\ref{corollary:param}.
\end{corollary}

\section{An algorithm to compute approximations of the invariant manifolds}\label{sectionalgorithm}

In this section we present the algorithm developed
in~\cite{BFM2015b} to compute approximate solutions of the
invariance equations~\eqref{eq:inv_equation}
and~\eqref{homequationflow}.

\subsection{Homological equations in the case of maps}
\label{sec:algotithmmaps} Let~$F$ be the map given by~\eqref{defF},
which we assume to be of class~$\CC^r$, with~$r$ large enough.
Taking advantage of the fact that~$F$ can be written as a Taylor
polynomial plus some higher order remainder, we look for approximate
solutions which are finite sum of homogeneous functions of
increasing degree. Then, for any $j\leq r-N+1$, we look for
$\KMil{j}$ and $\RMil{j+N-1}$ of the form
\begin{equation}\label{Step1KfRf}
\KMil{j}( x) = \sum_{l=1}^{j}  \Kl{l}( x), \qquad \RMil{j+N-1}( x) =
 x+\sum_{l=N}^{j+N-1} \Rl{l}( x),
\end{equation}
with $\Kl{1}(x)=(x,0)^{\top}$, $\Rl{N}( x)= p( x,0)$ and
$\Kl{l},\Rl{l} \in \Hog{l}$, satisfying
\begin{equation}\label{HIKlbis}
\begin{aligned}
\Em{j}(x) := F\circ \KMil{j}(x) & - \KMil{j} \circ \RMil{j+N-1}(x) \\
& =(\Em{j}_x, \Em{j}_y)(x) = \left(\oo\big(\Vert x
\Vert^{j+N-1}\big),\oo\big(\Vert x \Vert^{j+L-1}\big)\right),
\end{aligned}
\end{equation}
where the constant $L = \min\{N,M\}$ was introduced
in~\eqref{defeta}. We stress that the superscripts in the above
formula have two different meanings. While in~$\KMil{j}$
and~$\RMil{j+N-1}$ the superscript indicates that they are sums of
homogeneous functions of degree less or equal than~$j$ and $j+N-1$,
respectively, in~$\Em{j}$ denotes that it is the $j$-th error term.
Of course, the order of~$\Em{j}$ depends on~$j$ but also on $N$ and
$M$ and, as it is indicated in formula~\eqref{HIKlbis}, the $x$ and
$y$-components of~$\Em{j}$ may have different orders.

If, by induction, we assume that $\Em{j-1} =
(\Ef{x}^{j+N-1},\Ef{y}^{j+L-1})+\Ew{j}{}$, where $\Ef{*}^{\ell }$,
$* =x,y$, is a homogeneous function of degree $\ell$ and
\begin{equation}
\label{def:notationEs} \Ew{j}{}(x)  = \left(\oo\big(\Vert x
\Vert^{j+N-1}\big),\oo\big(\Vert x \Vert^{j+L-1}\big)\right),
\end{equation} then the
functions~$\Kl{j}=:(\Kl{j}_x,\Kl{j}_y) $ and~$\Rl{j+N-1}$ must
satisfy
\begin{equation}
\label{eqlKx}
D\Kl{j}_{ x}( x)  p( x,0) -D_x p( x,0) \Kl{j}_{ x}( x) - D_{y}p(
x,0) \Kl{j}_{y}( x)+ \Rl{j+N-1}( x)  = \Ef{ x}^{j+N-1}( x)
\end{equation}
and, depending on the values of~$N$ and~$M$,
\begin{align}
\label{eqlKyN<M}
\text{if $N<M$,} &  &D\Kl{j}_{y}( x)   p( x,0)   \phantom{ -D_y q( x,0) \Kl{j}_{y}( x) }   &=   \Ef{y}^{j+L-1}( x),  \\
\label{eqlKyN=M}
\text{if $N=M$,} & &D\Kl{j}_{y}( x)   p( x,0)   -D_y q( x,0) \Kl{j}_{y}( x)    &=  \Ef{y}^{j+L-1}( x), \\
\label{eqlKyN>M} \text{if $N>M$,} & &  -D_y q( x,0) \Kl{j}_{y}( x) &
=  \Ef{y}^{j+L-1}( x).
\end{align}

At this point it is worth to remark that one could try to find
solutions of the above equations in the space of homogeneous
polynomials of degree~$j$ and $j+N-1$, respectively. In the case
that $\KMil{j-1}$ and $\RMil{j+N-2}$ are sums of homogeneous
polynomials, the error term $\Ef{}^{j}$ is also a homogeneous
polynomial. But when~$N>M$ it is clear that $\Kl{j}_{y}( x) = -D_y
q( x,0)^{-1} \Ef{y}^{j+L-1}( x)$ cannot be, in general, a
polynomial, but a rational function. When~$N\le M$,
equations~\eqref{eqlKyN<M} and~\eqref{eqlKyN=M} are
$m\binom{j+L+n-2}{n-1}$ conditions while $\Kl{j}_{y}$, if assumed to
be a polynomial, would have only $m\binom{j+n-1}{n-1}$ free
coefficients. Hence, since $L\ge 2$, generically these equations
only admit polynomial solutions in the case that $n=1$ (which is the
case studied in~\cite{BFdLM2007}). It is easy to construct examples
where these obstructions do appear. See Section~6
in~\cite{BFM2015b}.

Now we summarize how we solve equations~\eqref{eqlKx},
\eqref{eqlKyN<M}, \eqref{eqlKyN=M} and~\eqref{eqlKyN>M}.

In the case $N>M$, since, as a consequence of hypothesis~H2, $D_y q(
x,0)$ is invertible, equation~\eqref{eqlKyN>M} is trivially solvable
in the space of homogeneous functions of degree~$j$.

In the case $N\le M$, let $\fpa( t, x)$ be the flow of
\begin{equation*}
\dot x=   p ( x,0).
\end{equation*}
As a consequence of~H3, $\fpa( t, x) \in V$, for
all $x\in V$ and $t>0$. We consider the homogeneous linear equations
\begin{align*}
\frac{d \fqa}{d  t}( t, x) = Dp(\fpa( t, x),0) \fqa( t, x), \\
\frac{d \fqa}{d  t}( t, x)  = Dq(\fpa( t, x),0) \fqa( t, x)
\end{align*}
and we denote by $\Mx( t, x)$ and $\My( t, x)$ their fundamental
matrices such that $\Mx(0, x)=\Id$, $\My(0, x)=\Id$, respectively.
From Theorem~3.2  in~\cite{BFM2015b}, the unique homogeneous
solution of equations~\eqref{eqlKyN<M} and \eqref{eqlKyN=M}  for
$\Klp{j}{y}$ is given by
\begin{equation}
\label{eq:solKy}
\begin{aligned}
&\Klp{j}{y}( x) &=& {\displaystyle \int_{\infty}^0 \El{j+L-1}{y}(\fpa( t, x))\,d t, } & \qquad \text{if   } N<M, \\
&\Klp{j}{y}( x) &= &{\displaystyle \int_{\infty}^0 \My^{-1}( t, x)
\El{j+L-1}{y}(\fpa( t, x))\,d t, } & \qquad   \text{if   }N=M.
\end{aligned}
\end{equation}
Theorem~3.2 in~\cite{BFM2015b} ensures that the above formulas
define homogeneous functions of degree~$j$.

The homogeneous solution of~\eqref{eqlKyN>M}, clearly unique, is
\[
\Klp{j}{y}( x) =  {\displaystyle \big (D_yq(x,0)\big )^{-1}
\El{j+L-1}{y}( x),}  \qquad  \text{if   }N>M.
\]

As for~\eqref{eqlKx}, notice that it is always possible to solve it
by choosing $\Klp{j}{ x}$ an arbitrary homogeneous function of
degree~$j$ and taking
\begin{equation}
\label{eq:solR} \Rl{j+N-1}( x) = \El{j+L-1}{ x}( x)-D_{y}p( x,0)
\Kl{j}_{y}( x) +D_x p( x,0)  \Klp{j}{ x}( x)- D\Klp{j}{ x}( x) p(
x,0).
\end{equation}
However, we prove in~\cite{BFM2015b} that, provided that $r$ is
large enough, there exists~$\df$ (which depends explicitly on the constants defined in~\eqref{defconstants})  such that if $\ell_*-N+2 \leq j$,
$\Rl{j+N-1}$ can be chosen as an arbitrary homogeneous function of
degree~$j+N-1$ and
\begin{multline}
\label{eq:solKx} \Klp{j}{x}( x) =\int_{\infty}^0 \Mx^{-1}( t, x)
\big [\El{j+L-1}{ x} (\fpa( t, x))-\Rl{j+N-1} (\fpa( t, x))  \\ -
D_yp(\fpa( t, x),0) \Klp{j}{y}(\fpa( t, x))\big ]\,d t.
\end{multline}
For instance, one can choose $\Rl{j+N-1}$ to be $0$,  if $j \ge
\ell_*-N+2$, which implies that the function~$\RMil{j+N-1}$
in~\eqref{Step1KfRf} can be taken as a \emph{finite} sum of
homogeneous functions.

\subsection{Homological equations in the case of flows}
\label{sec:algotithmflows} Let $U\subset \RR^{n+m}$ a neighborhood
of the origin and $ X: U \times \RR \to \RR^{n+m}$ be a $T$-periodic
vector field of the form~\eqref{defX}.
We look for $ K$ and $  Y$ of the form
\begin{equation*}
\KMil{j}( x,t) = \sum_{l=1}^{j}  \Kalg{l}( x,t), \qquad
\YMil{j+N-1}( x) = \sum_{l=N}^{j+N-1} \Yl{l}( x)
\end{equation*}
with $ K_{1}( x,t) = (x,0)^{\top}$, $\Yl{N}( x)= p( x,0)$ and
$\Kalg{l}$ a sum of two homogeneous functions: one of degree~$l$
independent of~$t$ and the other of order $\left(\oo\big(\Vert x
\Vert^{j+N-1}\big),\oo\big(\Vert x \Vert^{j+L-1}\big)\right)$.
The homogeneous terms $K^l$ in the statement of the
theorem are obtained by rearranging the sum above. They have to
satisfy the invariance equation \eqref{homequationflow} up to some
order~$j$ in the sense that the error term
\begin{equation*}
\Em{j}( x,t):= X(\KMil{j}( x,t),t)- D \KMil{j}( x,t) \YMil{j+N-1}(
x) -
\partial_t \KMil{j}( x,t)
\end{equation*}
satisfies
\begin{equation}\label{HIKlbisflow}
\Em{j} (x)=  (\Em{j}_x,\Em{j}_y)(x) = \left(\oo\big(\Vert x
\Vert^{j+N-1}\big),\oo\big(\Vert x \Vert^{j+L-1}\big)\right) .
\end{equation}

If, by induction, we assume that~\eqref{def:notationEs} is satisfied
(taking into account the time dependence) the
functions~$\Kalg{j}=(\Kalg{j}_x,\Kalg{j}_{y}) $ and~$\Yl{j+N-1}$
must satisfy
\begin{multline}
\label{Kxellflow} D\Kalg{j}_{x} (x,t)p(x,0) - D_x p(x,0)
\Kalg{j}_{x}(x,t) -D_y p(x,0)\Kalg{j}_{y}(x,t) \\ + \Yl{j+N-1} (x) +
\partial_t \Kalg{j}_{x}(x,t)  - \Ef{x}^{j+N-1}(x,t) = \oo\big(\Vert x
\Vert^{j+N-1}\big),
\end{multline}
and
\begin{equation}
\label{Kyellflow} D\Kalg{j}_{y}(x,t) p(x,0) - D_y q(x,0)
\Kalg{j}_{y}(x,t)  +\partial_t \Kalg{j}_{y}(x,t) -  \Ef{y}^{j+L-1}(
x,t)= \oo\big(\Vert x \Vert^{j+L-1}\big).
\end{equation}
Equation~\eqref{Kyellflow}, depending on the values of~$N$ and~$M$,
reads
\begin{align*}
\text{if $N<M$,} & \quad  D\Kalg{j}_{y}( x,t)   p( x,0) & &+ \partial_t \Kalg{j}_{x}(x,t) -   \Ef{y}^{j+L-1}( x,t) = \oo\big(\Vert x \Vert^{j+L-1}\big),  \\
\text{if $N=M$,} & \quad  D\Kalg{j}_{y}( x,t)   p( x,0)   &- D_y q( x,0)  \Kalg{j}_{y}( x,t)  &  + \partial_t \Kalg{j}_{x}(x,t)- \Ef{y}^{j+L-1}( x,t)= \oo\big(\Vert x \Vert^{j+L-1}\big), \\
\text{if $N>M$,} & \quad  & -D_y q( x,0) \Kalg{j}_{y}( x,t)
&+ \partial_t \Kalg{j}_{x}(x,t)- \Ef{y}^{j+L-1}( x,t)= \oo\big(\Vert x \Vert^{j+L-1}\big).
\end{align*}
We remark that, unlike the case of equations~\eqref{eqlKx}
to~\eqref{eqlKyN>M}, the functions $\Kalg{j}$ and $\Yl{j+N-1}$ we
obtain cancel out the term $\Ef{}^{j}$ in~\eqref{Kxellflow}
and~\eqref{Kyellflow} but introduce new terms of higher order.

For a $T$-periodic function $h$, we denote by $\ME{h}$ its mean, that is,
\[
\ME{h}(x) = \frac{1}{T}\int_{0}^{T} h(x,t) \,dt,
\]
and $\ZE{h}=h-\ME{h}$ its oscillatory part. If equations \eqref{Kxellflow} and \eqref{Kyellflow}
are satisfied for some $\Kalg{j}$ periodic, then it is clear that
the mean $\ME{\Kalg{j}}$ has to satisfy the equations
\begin{equation}
\label{eq:inductionflowsmean}
\begin{aligned}
D\ME{\Kalg{j}_{x}} (x)p(x,0) - D_x p(x,0) \ME{\Kalg{j}_{x}}(x)  -D_yp(x,0)\ME{\Kalg{j}_y}(x) & \\
 + \Yl{j+N-1} (x) - \ME{\Ef{x}^{j+N-1}}(x) & = \oo\big(\Vert x \Vert^{j+N-1}\big),\\
D\ME{\Kalg{j}_{y}}(x) p(x,0) - D_y q(x,0)  \ME{\Kalg{j}_{y}}(x)  -
\ME{\Ef{y}^{j+L-1}}( x) & = \oo\big(\Vert x \Vert^{j+L-1}\big).
\end{aligned}
\end{equation}
These equations can be solved in the same way as~\eqref{eqlKx},
\eqref{eqlKyN<M}, \eqref{eqlKyN=M} and~\eqref{eqlKyN>M}, in the
previous section. We conclude that $\ME{\Kalg{j}}$ and $\Yl{j+N-1}$
exist and they both have the appropriate orders, i.e, degree $j$ and
$j+N-1$ respectively.

Now we impose that
\begin{equation}
\label{eq:oscillatorypart}
\partial_ t \ZE{\Kalg{j}}( x,t) = (\ZE{\Ef{x}^{j+N-1}}( x,t),\ZE{\Ef{y}^{j+L-1}}( x,t) )
\end{equation}
and that $\ZE{\Kalg{j}}$ has zero mean. Consequently,
$\ZE{\Kalg{j}}(x) =\left(\oo\big(\Vert x
\Vert^{j+N-1}\big),\oo\big(\Vert x \Vert^{j+L-1}\big)\right)$.

We conclude that $\Kalg{j}=\ME{\Kalg{j}} + \ZE{\Kalg{j}}$ and
$\Yl{j+N-1}$ satisfy equations \eqref{Kxellflow} and
\eqref{Kyellflow} and then \eqref{HIKlbisflow} is satisfied.

\begin{remark}
The $\Kalg{j}$ found are not homogeneous functions, but sums of
homogeneous functions. Concretely, $\Kalg{j}_x$ has a term of order
$j$ and another of order $j+N-1$. Analogously, $\Kalg{j}_{y}$ has a
term of order $j$ and another of order $j+L-1$.
\end{remark}

As in Section~\ref{sec:algotithmmaps}, if $L=N$, then
\eqref{HIKlbisflow} with $j=\ell-N+1$, implies that
$$
\Em{j}(x,t):=X(\KMil{j}(x,t),t) - D \KMil{j}( x,t) \YMil{j+N-1}( x)
- \partial_t \KMil{j}( x,t) =   \oo\big(\Vert x \Vert^{j+N-1}\big)
$$
and we are done in this case. The case $L=M < N$ requires an extra
argument which is totally analogous as in the case of maps, in
Section~\ref{sec:algotithmmaps}.

\section{Example. The elliptic spatial restricted three body
problem}\label{sec:esrtbp}

We have pointed out in the previous section that, in general, the
invariant manifolds of a parabolic fixed point do not have
polynomial expansions if their dimension is greater than one,
regardless of the regularity of the map. However, it may be possible
that the system of equations defined by~\eqref{eqlKx}
and~\eqref{eqlKyN<M}--\eqref{eqlKyN>M} admits polynomial homogenous
solutions.
Here we take advantage of the expressions~\eqref{eq:solKy},
\eqref{eq:solR} and~\eqref{eq:solKx} to show that this is the case
of the \emph{parabolic infinity} in the elliptic spatial restricted
three body problem.

The spatial elliptic restricted three body problem is a simplified
version of the spatial three body problem where one of the bodies is
assumed to have zero mass while the other two, named the primaries,
evolve describing Keplerian ellipses around their center of mass.

We introduce $\hat q (f) = (\rho (f) \cos f, \rho(f) \sin f,0)$,
where, for a given eccentricity~$0\le e <1$,
\[
\rho(f) = \frac{1-e^2}{1+e\cos f}.
\]
Rescaling time and mass units, we can assume that the masses of the
primaries are $\mu$ and $1-\mu$, respectively, and their positions
are given by $q_1 = \mu \hat q$ and $q_2 = -(1-\mu) \hat q$, where
$f$ denotes the so-called true anomaly which satisfies
\begin{equation*}
\frac{df}{dt} = \frac{(1+e \cos f)^2}{(1-e^2)^{3/2}}.
\end{equation*}
Then, denoting by $q\in \RR^3$ the position of the third body, the
equations for $q$ are
\[
\ddot{q} = -(1-\mu) \frac{q-q_1}{r_1^3} -\mu \frac{q-q_2}{r_2^3},
\]
where $r_i = \|q - q_i\|$, $i=1,2$. Introducing the momenta $p= \dot
q$, this system is Hamiltonian with respect to
\begin{equation*}
H(q,p,t) = \frac{\|p\|^2}{2} - U(q,t), \qquad U(q,t) =
\frac{1-\mu}{r_1} + \frac{\mu}{r_2}.
\end{equation*}

Our aim is to study the parabolic invariant manifolds of infinity.
To this end, we start by considering spherical
coordinates~$(r,\alpha,\theta)$ in $\RR^3$, namely $q = (r \cos
\alpha \cos \theta, r \sin \alpha \cos \theta, r \sin \theta)$. Let
$(R, A, \Theta)$ be their conjugated momenta, which can be obtained
through a Mathieu transformation. They satisfy
\[
p = m(r,\alpha,\theta) \begin{pmatrix} R \\ A \\ \Theta
\end{pmatrix}, \qquad
m(r,\alpha,\theta) = \begin{pmatrix} \cos \alpha \cos \theta &
-\frac{\sin
\alpha}{r \cos \theta} & -\frac{\cos \alpha \sin \theta}{r} \\
\sin \alpha \cos \theta & \frac{\cos \alpha}{r \cos \theta} &
-\frac{\sin \alpha \sin \theta}{r} \\
\sin \theta & 0 & \frac{\cos \theta}{r}
\end{pmatrix}.
\]
The new Hamiltonian is
\[
\hat H(r,\alpha,\theta,R,A, \Theta,t) = \frac{1}{2} \left(
\frac{A^2}{r^2 \cos^2 \theta}+\frac{\Theta^2}{r^2}+ R^2\right) -
\hat U(r,\alpha,\theta,t),
\]
with
\begin{equation}
\label{potentialinpolarcoordinates}
\begin{aligned}
\hat U (r,\alpha,\theta,t)= & \frac{1-\mu}{\sqrt{ r^2-2\mu
\rho(f)r\cos(\alpha-f)\cos \theta + \mu^2 \rho^2(f) }} \\ & +
\frac{\mu}{\sqrt{ r^2+2(1-\mu) \rho(f)r\cos(\alpha-f)\cos \theta +
(1-\mu)^2 \rho^2(f) }} \\
= & \frac{1}{r} - \frac{\mu(1-\mu)}{2}(1-3 \cos(\alpha
-f))\frac{\rho^2(f)\cos^2 f \cos^2\theta}{r^3} + \OO
\left(\frac{1}{r^4}\right).
\end{aligned}
\end{equation}
To study the behavior of the system at $r=\infty$, we perform the
non-canonical change of variables due to McGehee $r = 2/z^2$. Since
$\dot r = R$ and the change does not involve the remaining
variables, the equations of motion in the new variables are
\[
\begin{aligned}
\dot z & = -\frac{1}{4}z^3 R \\
\dot \alpha & = \partial_A \hat H_{\mid r=2/z^2} = \frac{A z^4}{4 \cos^2 \theta} \\
\dot \theta & = \partial_\Theta \hat H_{\mid r=2/z^2} = \frac{1}{4} \Theta z^4 \\
\dot R & =  -\partial_r \hat H_{\mid r=2/z^2} = \frac{A^2 z^6}{8
\cos^2 \theta }+\frac{\Theta^2 z^6}{8}+
\partial_r \hat U(2/z^2,\alpha,\theta,t) = -\frac{1}{4} z^4 + \OO(z^6)\\
\dot A & =  -\partial_\alpha \hat H_{\mid r=2/z^2} =
\partial_{\alpha} \hat U(2/z^2,\alpha,\theta,t) = \OO(z^6) \\
\dot \Theta & =  -\partial_\theta \hat H_{\mid r=2/z^2} = -\frac{A^2
z^4 \sin \theta}{4 \cos^3 \theta} + \partial_{\theta} \hat
U(2/z^2,\alpha,\theta,t) = -\frac{A^2  z^4 \sin \theta}{4 \cos^3
\theta} +\OO(z^6).
\end{aligned}
\]
Notice that the set $\{z=0, \;R=0\}$ is invariant and foliated by
fixed points. We focus on those with $\theta= \Theta = 0$, $\alpha =
\alpha_0$, $A=A_0$. To apply our theory, we perform the following
local change of variables
\[
\hat \theta = \frac{\theta}{z}, \quad \hat \Theta =
\frac{z\Theta}{\theta}, \quad \hat \alpha =
\frac{\alpha-\alpha_0+AR}{z}, \quad \hat A = \frac{A-A_0}{z},
\]
which transforms the system into
\[
\begin{aligned}
\dot z & = -\frac{1}{4}z^3 R & \qquad \dot R & =   -\frac{1}{4} z^4 + z^6 \OO_0 \\
\dot{\hat \alpha} &  = \frac{1}{4}z^2 R \hat \alpha + z^5 \OO_0  &
\qquad
 \dot{\hat A} & = \frac{1}{4} \hat A z^2 R + z^5 \OO_0 \\
\dot{\hat \theta} &  = \frac{1}{4} z^2 R \hat \theta + \frac{1}{4}
z^3 \hat \theta \hat \Theta &  \qquad
 \dot{\hat \Theta} &
= -\frac{1}{4}z^2 R \hat \Theta -\frac{1}{2} z^3 \hat \Theta^2 + z^5
\OO_0,
\end{aligned}
\]
where all the terms up to degree $6$ in the local variables are
shown (we write $\OO_k$ meaning $\OO(\|(z,R,\hat \alpha, \hat A,\hat
\theta,\hat \Theta)\|^k)$). Notice that the leading terms are of
degree~$4$. Finally, it will be convenient to introduce
\[
u = \frac{1}{2}(z+R), \qquad v = \frac{1}{2}(z-R)
\]
so that the system, reordering equations,  becomes
\begin{equation}
\label{def:R3BPfinalcoordinates}
\begin{aligned}
\dot u & = -\frac{1}{4}(u+v)^3 u + (u+v)^6\OO_0 \\
\dot{\hat\Theta} & = -\frac{1}{4}(u+v)^2 (u-v) \hat \Theta
-\frac{1}{2} (u+v)^3 \hat \Theta^2+(u+v)^5 \OO_0 \\
\dot v & =   \frac{1}{4} (u+v)^3 v + (u+v)^6 \OO_0 \\
 \dot{\hat \alpha} &  = \frac{1}{4}(u+v)^2 (u-v) \hat \alpha +
(u+v)^5 \OO_0 \\
\dot{\hat A} & = \frac{1}{4}  (u+v)^2 (u-v) \hat A +  (u+v)^5  \OO_0\\
\dot{\hat \theta} &  = \frac{1}{4} (u+v)^2 (u-v) \hat \theta
+\frac{1}{2} (u+v)^3 \hat \theta \hat \Theta.
\end{aligned}
\end{equation}
We emphasize that the leading terms do not depend on~$t$, but the
remainders do depend $2\pi$-periodically on~$t$.

Let $X$ denote the vector field defined
by~\eqref{def:R3BPfinalcoordinates}. We can write the vector field
in the form~\eqref{defX} taking $x=(u,\hat \Theta)$, $y=(v,\hat
\alpha, \hat A, \hat \theta)$ and
\begin{equation}
\label{def:pqrestricted}
p(x,y) = \begin{pmatrix} -\frac{1}{4}(u+v)^3 u \\
-\frac{1}{4}(u+v)^2 (u-v) \hat \Theta \end{pmatrix}, \qquad q(x,y) =
\begin{pmatrix}
\frac{1}{4} (u+v)^3 v \\ \frac{1}{4}(u+v)^2 (u-v) \hat \alpha \\
\frac{1}{4}  (u+v)^2 (u-v)\hat A \\ \frac{1}{4} (u+v)^2 (u-v) \hat
\theta
\end{pmatrix}.
\end{equation}

\begin{theorem}
Let $W$ be a perturbation of~$\hat U$
in~\eqref{potentialinpolarcoordinates} of the form $W = \hat U + V$,
where
\[
V(r,\alpha,\theta,t) = \frac{1}{r^3} \hat V(r,\alpha,\theta,t)
\]
(in spherical variables) is such that the equations of motion leave
the plane $\theta = \Theta = 0$ invariant (that is,
$\partial_{\theta} V_{\mid\theta=0} =0$) and $\hat V$ is analytic in
$1/r$ and the rest of its arguments.

Then, after the changes of variables described above, the equations
of motion are given by~\eqref{def:R3BPfinalcoordinates}. The origin
is a parabolic fixed point. It has an analytic  stable invariant two
dimensional manifold which admits a parametrization of the form
$K(x,t) = (x,0) + \tilde K(x,t)$, where
\[
\tilde K (x,t) = \OO(\|x\|^2), \qquad u>0,\;\tilde \Theta >0
\]
such that
\begin{equation*}
X(K(x,t),t) = D K(x,t) Y(x) + \partial_t K(x,t),
\end{equation*}
with $Y(x)=p(x,0) + \OO(\|x\|^5)$ is a polynomial of degree~$7$.

The function $\tilde K(x,t)$ is $2\pi$-periodic in~$t$ and, for all
$\ell \ge 7$,
\begin{equation*}
\tilde K(x,t) = \sum_{j=2}^{\ell} \breve  K^{j}(x,t)+
\OO(\|x\|^{\ell+1})
\end{equation*}
where $\breve  K^{j}$ are homogeneous polynomials of degree~$j$.
That is, the stable invariant manifold admits polynomial
approximation up to any order.
\end{theorem}

\begin{proof}
System~\eqref{def:R3BPfinalcoordinates} satisfies hypotheses (a) and
(b) of Theorem~\ref{maintheoremflow}. Hence, in order to obtain the
claim, we only need to check hypothesis (c). It is enough to find
approximate solutions of the invariance equation
\begin{equation}
\label{eq:homological}
 X(K(x,t),t) -
D K(x,t) Y(x) -
\partial_t K(x,t) = 0.
\end{equation}
We show that there are indeed approximate solution of these equation
up to any order and that these solutions are sums of homogeneous
polynomials.

We use the construction described in
Section~\ref{sec:algotithmflows} to find approximate solutions of
the above equation. The procedure applies  in the region $\{u>0,\hat
\Theta
>0\}$.

The explicit expression of the flow of the vector field
$p(x,0)$, with $p$ on \eqref{def:pqrestricted}, is
\begin{equation*}
\fpa( t, x) = \frac{1}{(1+\frac{3}{4}tu^3)^{1/3}}
\begin{pmatrix}
u \\ \hat \Theta
\end{pmatrix}.
\end{equation*}
Let $\Mx( t, x)$ and $\My( t, x)$ be the fundamental matrices of the
linear equations
\[
\begin{aligned}
\frac{d \fqa}{d  t}( t, x)  = D_x p(\fpa( t, x),0) \fqa( t, x) \\
\frac{d \fqa}{d  t}( t, x)  = D_y q(\fpa( t, x),0) \fqa( t, x)
\end{aligned}
\]
such that $\Mx(0, x)=\Id$ and $\My(0, x)=\Id$, respectively. We have
that
\begin{equation*}
\Mx^{-1} ( t, x) =
\begin{pmatrix}
\left( 1+\frac{3}{4}tu^3 \right)^{4/3} & 0 \\
\frac{3}{4} t u^2 \hat \Theta \left( 1+\frac{3}{4}tu^3 \right)^{1/3}
& \left( 1+\frac{3}{4}tu^3 \right)^{1/3}
\end{pmatrix}
\end{equation*}
and
\begin{equation*}
\My ( t, x) = \left(1+\frac{3}{4}tu^3\right)^{1/3} \Id_{4\times 4}.
\end{equation*}

Along this proof we will deal with several objects that will be
homogeneous polynomials. Their superscripts will denote their
degree. A slightly different notation is used for $E^{>j}$.
See~\eqref{HIKlbis}-\eqref{def:notationEs}.

We write the vector field in~\eqref{def:R3BPfinalcoordinates} as~$X
= \sum_{l\ge 4} X^l$, where $X^l$  depends $2\pi$-periodically
on~$t$. Following the algorithm described in
Section~\ref{sec:algotithmflows} with $N=M=4$, we look for solutions
of the equation~\eqref{eq:homological} of the form
\[
K(x,t) = \sum_{l\ge 1} \Kalg{l}(x,t), \qquad Y(x) = \sum_{l\ge 4}
Y^l(x),
\]
where  $\Kalg{l}$ depends $2\pi$-periodically in~$t$ and it is of the
form
\[
\Kalg{l}(x,t)=  K^l(x) + \widetilde{ K}^{l+3}(x,t),\qquad
\text{with}\;\; \widetilde{ K}^{l+3}=\ZE{\Kalg{l}}.
\]
and $K^1(x) = (x,0)$, $\widetilde{ K}^{4}(x,t)= 0$, $Y^4(x) =
p(x,0)$. The homogeneous functions $\breve K^l$ in the statement of
the theorem are obtained by rearranging the sum above.

We recall that $x=(u, \hat \Theta)$ and $y=(v,\hat \alpha, \hat A,
\hat \theta)$. It is clear from~\eqref{def:R3BPfinalcoordinates}
that the homogeneous polynomials $X^l$ satisfy that
\begin{equation} \label{eq:Xu5factor}
\begin{aligned}
X_{\xi}^5(x,0,t) & = u^5 \hat X_{\xi}^{0}(x,t),
 \qquad  \qquad \xi = \hat \alpha, \hat A, \hat \theta,\\
 X_{u}^5(x,0,t) & = X_{v}^5(x,0,t)  = 0,\\
 X_{\hat \Theta}^5 (x,0,t)  & = -\frac{1}{2} u^3 \hat \Theta^2 + u^5
\hat X_{\hat \Theta}^0 (x,t),
\end{aligned}
\end{equation}
and, for $l\ge 6$,
\begin{equation*}
\begin{aligned}
X^l (x,0,t)  & = u^5 \hat X^{l-5} (x,t) \\
X_{\xi}^l(x,0,t) & = u^6 \hat X_{\xi}^{l-6}(x,t), \qquad \xi = u,v.
\end{aligned}
\end{equation*}

The statement is a consequence of the following claim. We make the
convention that $\OO_j = 0$ if $j<0$.

\begin{claim}
\begin{enumerate}
\item[(i)]
$K^j(x) = u^2\OO_{j-2}$, $K^j_{u,v}(x) = u^3 \OO_{j-3}$, $j\ge 2$.
$K_x^j = 0$ if $2\le j \le 7$.

$\widetilde K^{j+3}(x,t) = u^5 \OO_{j-2}$, $\widetilde
K^{j+3}_{u,v}(x,t) = u^6 \OO_{j-2}$, $j\ge 3$.
\item[(ii)] $Y^5(x) = \begin{pmatrix} a_1 u^5 \\ u^3 (a_2 \hat \Theta^2 + a_3 u^2) \end{pmatrix}$, with
 $a_1,a_2,a_3\in \RR$,
$Y^j(x) = (u^6 \OO_{j-6},u^5 \OO_{j-5})^{\top}$, $6\le j \le 7$ and
$Y^j = 0$ for $j\ge 8$.
\item[(iii)] Denoting $K^{\le j} = \sum_{l=1}^{j} (K^l+ \widetilde K^{l+3})$, $Y^{\le j} =
\sum_{l=4}^{j} Y^l$,
\[
E^{>j}(x,t) = X(K^{\le j}(x,t),t) - D K^{\le j}(x,t)   Y^{\le
j+3}(x) - \partial_t K^{\le j}(x,t),
\]
and $E^{>j} = E^{j+4} + \hat E^{>j+1}$ with $E^{j+4}(x,t) = (\OO(\|x\|^{j+4}))$,
$\hat E^{>j+1} (x,t) = (\oo(\|x\|^{j+4}))$, then,
\[
E^{j+4}(x,t) = u^5 \OO_{j-1}, \quad  E^{j+4}_{u,v}(x,t) = u^6
\OO_{j-2}, \quad j \ge 2.
\]
\end{enumerate}
In (i) and (ii) the terms $\OO_j$ are \emph{homogenous polynomials}
in~$x$ of degree~$j$ while in~(iii) $\OO_j$ are analytic functions
in $x$ of order~$j$.
\end{claim}

The following fact will be used repeatedly without mention. Given
any monomial $Z(x) = u^{j_1}\hat \Theta^{j_2}$ and denoting
$\{e_1,e_2\}$ the canonical basis of $\RR^2$, there exist $c_i\in
\RR$, depending only on $j_1$ and $j_2$, such that
\[
\begin{aligned}
\int_{\infty}^0 \Mx^{-1} ( t, x)  Z( \fpa( t, x)) e_1
\,dt & = c_1 \frac{Z(x)}{u^3} e_1 + c_2 \frac{Z(x)}{u^4}\hat \Theta e_2,\\
\int_{\infty}^0 \Mx^{-1} ( t, x)  Z ( \fpa( t, x)) e_2 \,dt & = c_3
\frac{Z(x)}{u^3} e_2
\end{aligned}
\]
and, denoting $\{e_j'\}_{j=1,\dots,4}$ the canonical basis of
$\RR^4$, there exists $c\in \RR$, depending only on $j_1$ and $j_2$,
such that
\[
 \int_{\infty}^0 \My^{-1} ( t, x)
Z ( \fpa( t, x))e_j' \,dt  =c \frac{Z(x)}{u^3} e_j', \qquad
j=1,\dots,4.
\]
Indeed, it suffices to make the change $s= tu^3$ in the integrals.
Obviously, the previous integrals are only convergent when
$j_1+j_2\ge 8$, for the first one, when $j_1+j_2\ge 5$, for the
second one and $j_1+j_2\ge 3$ for the last one.

We prove the claim by induction. We start with the case $j=2$.

According to the algorithm, using that $X^4 = (p,q)^\top$ and
$Y^4(x) = p(x,0)$ in equations~\eqref{eq:inductionflowsmean} for
$j=2$, the functions $K^2$ and $Y^5$ must satisfy
\[
D K^2 Y^4 -(D X^4 \circ K^1) K^2 +\begin{pmatrix} \Id \\ 0
\end{pmatrix} Y^5 = \ME{ E^{5}},
\]
where $E^{5} = X^5 \circ K^1$ denotes the terms of degree $5$ of
$E^{>1}$ and we recall that $\ME{Z}$ denotes the mean of a periodic
function~$Z$. Using~\eqref{eq:solKy}, the equation for $K_y^2$ has
the homogeneous solution
\begin{equation}
\label{eq:piyK2} K_y^2 ( x) = \int_{\infty}^0 \My^{-1} ( t, x) \ME{
E_y^{5}} ( \fpa( t, x)) \,dt.
\end{equation}
Since, in view of~\eqref{eq:Xu5factor},  $\ME{X_y^5} \circ K^1 (x) =
\ME{X_y^5}  (x,0)= a_0 u^5$, we have that $K_y^2(x) = b_0 u^2$,
where $a_0,b_0\in\RR^4$. Furthermore, since $\ME{X_v^5} = 0$ and
$\My$ is a diagonal matrix, we deduce that $K_v^2 = 0$.

Once $K_y^2$ is found, we take $K_x^2 =0 $ and choose appropriately
$Y^5$, that is,
\[
Y^5(x) = D_y X^4_x (x,0)K_y^2 ( x)  +\ME{X_x^5} \circ K^1 (x) =
\begin{pmatrix} a_1 u^5 \\ a_2 u^3 \hat \Theta^2 + a_3 u^5
\end{pmatrix},
\]
with $a_1,a_2, a_3 \in \RR$, where we have used that, since $K_v^2 =
0$, and
\begin{equation}
\label{eq:D_y p(x,0)}
D_y X^4_x (x,0) = D_y p(x,0) = \begin{pmatrix} -\frac{3}{4} u^3 & 0 & 0 & 0 \\
-\frac{1}{4} u^2\hat \Theta & 0 & 0 & 0
\end{pmatrix},
\end{equation}
we have that $D_y p(x,0) K_y^2 ( x) = 0$.  This accounts for the
first part of~(ii).

To cancel the oscillatory part of $E^5$ we
use~\eqref{eq:oscillatorypart} and we choose~$\widetilde K^{5}$ with
zero mean such that
\[
\partial_t \widetilde K^{5}(x,t) =  \widetilde{E^5}(x,t).
\]
From~\eqref{eq:Xu5factor} we get that $\widetilde K^{5} (x,t) =
u^5\OO_0$ and $\widetilde K^{5}_{u} (x,t) = \widetilde K^{5}_{v}
(x,t) =0$.

With this choice of $K^{\le 2}$ and $Y^{\le 5}$ the algorithm
ensures that the remainder $E^{>2}(x,t)= \OO_6$. We have that
\[
\begin{aligned}
E^{>2}(x,t)  = & X(K^{\le 2}(x,t),t) - D K^{\le 2}(x,t)   Y^{\le
5}(x) -
\partial_t K^{\le 2}(x,t) \\
 = & X^6(x,0,t) +D X^5 (x,0,t) K^2(x) +\frac{1}{2} D^2 X^4(x,0) K^2(x)^{\otimes 2} \\
  & - D K^2(x) Y^5(x) + u^5 \OO_1 \\
 = & u^5 \OO_1.
\end{aligned}
\]
The last equality uses that $K_x^2 = 0$, $K_v^2 = 0$, $K_y^2(x) =
u^2 \OO_0$, the particular form of $Y^5$, $X^5$, $X^6$ and the fact
that $\frac{\partial^2 X^4}{\partial y^2} (x,0) = u \OO_1$.
Moreover, using that $K_x^2 = 0$, $K_v^2 = 0$, $X^5_{u,v} =0$ and
$X^6_v(x,0,t) = u^6 \OO_0$ one obtains that
\begin{equation}
\label{E6} E_{u,v}^{>2}(x,t) = u^6 \OO_0.
\end{equation}
This proves the claim for $j=2$.

Now we assume that we have obtained $K^{\le j-1}$, $Y^{\le j+2}$ and
$E^{>j-1}$, with $j\ge 3$, satisfying the induction hypotheses. The
equation for $K^j$ and $Y^{j+3}$ is
\[
D K^j Y^4 -(D X^4 \circ K^1) K^j+ \begin{pmatrix} \Id \\ 0
\end{pmatrix} Y^{j+3} =  \ME{E^{j+3}} .
\]
The function $K_y^j$ is obtained as we did for $K_y^2$
in~\eqref{eq:piyK2}.

Since $E^{j+3}(x,t) =u^5\OO_{j-2}$, we obtain that $K_y^j(x) = u^2
\OO_{j-2}$. By the same argument, using~\eqref{E6} and that $\My$ is
a diagonal matrix, one has $K_v^j(x) = u^3 \OO_{j-3}$.

To find $K^j_x$ and $Y^{j+3}$ we proceed in two different ways
according to whether $j\le 4$ or $j \ge 5$. The point is that for
$j\ge 5$ we can take $Y^{j+3} = 0$ choosing appropriately $K^{j}_x$.
However,  for $j \le 4$, the integrals involved in the computation
of $K^{j}_x$ do not converge.

If $j \le 4$, we choose $K_x^j = 0$ and
\[
Y^{j+3}(x) = D_y X^4_x \circ K^1 (x)K_y^j ( x)  +\ME{E_x^{j+3}} (x).
\]
Formula~\eqref{eq:D_y p(x,0)} and the induction hypothesis gives
that $Y^{j+3}(x) = (u^6\OO_{j-3},u^5\OO_{j-2})^{\top}$.

Instead, if $j \ge 5$, we choose $Y^{j+3} = 0$ and
\[
K_x^j ( x) = \int_{\infty}^0 \Mx^{-1} ( t, x) \ME{E_x^{j+3}} ( \fpa(
t, x)) \,dt.
\]
The induction hypothesis on $E^{j+3}$ gives $K_x^j ( x) = (u^3
\OO_{j-3},u^2\OO_{j-2})^{\top}$.

We choose~$\widetilde K^{j+3}$ with zero mean such that $\partial_t
\widetilde K^{j+3}(x,t) =  \widetilde E^{j+3}(x,t)$. Again from the
induction hypothesis we get that $\widetilde K^{j+3} (x,t) =
u^5\OO_{j-2}$ and $\widetilde K^{j+3}_{u,v} (x,t) = u^6\OO_{j-3}$.

Finally, we need to check the properties of~$E^{j+4}$. From the
definition of~$E^{>j}$,
\[
\begin{aligned}
E^{>j}(x,t)  = & E^{>j-1}(x,t) + X(K^{\le j}(x,t),t)-X (K^{\le j-1}(x,t),t) \\
 & -( D K^{\le j}(x,t)   Y^{\le
j+3}(x) - D K^{\le j-1}(x,t)   Y^{\le
j+2}(x)) \\
&  - (\partial_t K^{\le j}(x,t)  -
\partial_t K^{\le j-1}(x,t)) \\
 = & E^{>j-1}(x,t) + T_1(x,t) - T_2(x,t)- T_3(x,t),
\end{aligned}
\]
where $T_i$ are defined in the obvious way.

We have
\[
T_1  = \int_0^1 DX(K^{\le j-1}+s (K^j+\tilde K^{j+3}),t) (K^j+\tilde
K^{j+3})\,ds.
\]
Taking into account the structure of~$X$
in~\eqref{def:R3BPfinalcoordinates} a long but straightforward
computation gives
\[
T_1 (x,t) = u^5\OO_{j-2}, \quad (T_1)_{u,v} (x,t) = u^6\OO_{j-3}.
\]

For~$T_2$ we have
\[
T_2 = D K^{\le j-1}   Y^{j+3} + (D K^{j}+D\tilde K^{j+3})   Y^{\le
j+3}.
\]
A simple calculation gives that for $j\ge 3$,
\begin{align*}
DK^{\le j-1}(x,t) &  = \begin{pmatrix} 1 \OO_0 &   u \OO_0 & u^2
\OO_0 & u \OO_0 & u \OO_0 & u \OO_0
\\
u^3 \OO_0 & 1 \OO_0 & u^3 \OO_0 & u^2 \OO_0 & u^2 \OO_0 & u^2 \OO_0
\end{pmatrix}^{\top} ,\\
Y^{j+3}(x) &  = \begin{pmatrix} u^6 \OO_0
\\
u^5 \OO_0
\end{pmatrix}, \\
(DK^{\le j}+D\tilde K^{j+3})(x,t) &  = \begin{pmatrix} u^2 \OO_0 &
u \OO_0 & u^2 \OO_0 & u \OO_0 & u \OO_0 & u \OO_0
\\
u^3 \OO_0 & u^2 \OO_0 & u^3 \OO_0 & u^2 \OO_0 & u^2 \OO_0 & u^2
\OO_0
\end{pmatrix}^{\top} \\
\intertext{and} Y^{\le j+3}(x) &  = \begin{pmatrix} u^5 \OO_0
\\
u^3 \OO_2
\end{pmatrix},
\end{align*}
where here $\OO_j$ denotes a polynomial in~$x$ of order~$j$. This
implies
\[
T_2 (x,t) = u^5\OO_{j-2}, \qquad (T_2)_{u,v} (x,t) = u^6\OO_{j-3}.
\]
Finally, $T_3 = \partial_t \tilde K^{j+3}$ and the induction
hypotheses gives (iii) for~$j$.

Note however that some terms of $T_1$, $T_2$ and $T_3$ do not
contribute because their order is less or equal than $j+3$ and are
compensated by the choice of the $K$'s and $Y$'s.
\end{proof}

\section{Examples}\label{sectionexamples}
In this section we provide examples showing that hypotheses H1, H2
and H3 are necessary to the existence of the invariant manifolds. We
also show that the manifolds may be much less regular than the map.

\subsection{A toy model}
Here we construct an example of a map without stable invariant
manifold but satisfying both H1 and H2.

Let $\varphi$ be the flow of the equations in $\RR^2 \times \RR$
\begin{equation*}
\dot{x}_1 = -x_1^2, \qquad \dot{x}_2 = - a x_1 x_2, \qquad
\dot{y}=bx_1 y + x_2^3,
\end{equation*}
being $a,b>0$, and $F(x,y) = \varphi(1;x,y)$, with $x=(x_1,x_2)$,
its time $1$ map.

\begin{claim}
There exist $V\subset \RR^2$, star-shaped with respect to the
origin, such that $F$ satisfies hypotheses H1 and H2 in $V$.

If $b+3a\leq 1$, $F$ has no invariant stable manifold over $V$ of
the origin.
\end{claim}

The map $F$  has the form~\eqref{defF} with $p(x,y) = (-x_1^2,-a x_1
x_2)$ and $q(x,y) = b x_1 y$.

We introduce
\begin{equation*}
W = \left \{ x=(x_1, x_2) \in \RR^2 :  |x_2| <
(1-a)x_1<\frac{2}{a+1} \right \}.
\end{equation*}
First we note that the map~$F$ satisfies hypotheses H1 and H2 with
the supremum norm in any open set $V$ contained in $W$. Of course
the constants $\Ap, \Bq$ will depend on $V$. However we claim that
there is no invariant set for $F_x$ contained in $W$. As a
consequence, hypothesis H3 can not be satisfied. Indeed, assume that
$x^0 =(x_1,x_2) \in W$, and consider
\begin{equation*}
x^n = F_x(x^{n-1}) = F_x^n(x^0) = \left (\frac{x_1}{1 +
nx_1},  \frac{x_2}{\big (1 + nx_1\big )^{a}}\right ) .
\end{equation*}
The sequence $x^n \in W$ if and only if $x_1 \geq |x_2| \big (1 + n
x_1 \big )^{1-a}$, $\forall n\geq 0$, which is not true since
$x_1>0$ and $a<1$.

Now we check that the map $F$ has no stable invariant manifold.
Indeed, if such a manifold exists, then, for any $(x,y)$ belonging
to it, $F_y^n (x,y) \to 0$ as $n\to \infty$. Since
$$
\varphi_y( t, x,y) = \big (1+ t x_1\big)^b \left [y +
x_2^3\int_{0}^{ t} \frac{1}{\big (1+  s x_1\big )^{b+3a}} \, ds
\right ],
$$
we deduce that
\begin{equation*}
F_y^n(x,y) = \big (1+n x_1\big)^b \left [y + x_2^3\int_{0}^n
\frac{1}{\big (1+ s x_1\big )^{b+3a}} \, ds \right ].
\end{equation*}
Therefore, since $(1+n x_1)^b\to \infty$ as $n\to \infty$ a
necessary condition for $F^n_y(x,y) \to 0$ as $n\to \infty$, is
that
\begin{equation*}
y= x_2^3\int_{\infty}^0\frac{1}{\big (1+ sx_1\big )^{b+3a}}\, ds,
\end{equation*}
and the claim follows because the above integral is not convergent
when $b+3a\leq 1$.

\subsection{The loss of differentiability}

The following example shows that the invariant manifolds of a
parabolic fixed point may be of finite order of differentiability.
This maximum order of differentiability is attained when the
manifold is written (locally) as a graph, since  if the invariant
manifold possesses a parametrization of the form given by
Theorem~\ref{maintheorem} with some regularity, by performing a
close to the identity change of variables, its representation as a
graph will be also of the same regularity.

Let $a,b >0$. Let $F$ be the the time $1$ map of
\begin{equation*}
\dot{x} =  p ( x), \qquad \dot{y} =  q_1( x) y + g( x),
\end{equation*}
where  $ x=( x_1, x_2) \in \RR^2$, $y \in \RR$, $p$ is such that the
equation $\dot x = p(x)$ in polar coordinates $(x_1,x_2) = (r \cos
\theta, r \sin \theta)$ becomes
\begin{equation}\label{exH4polar}
\dot{r} = -a r^5,\qquad \dot{\theta} = r^4 \sin 4 \theta
\end{equation}
($p$ is a homogeneous polynomial of degree~$5$) and
$$
q_1(x)= b (x_1^2 + x_2^2)^2 ,\;\;\; g(x) = 4 (x_1^2 + x_2^2) x_1 x_2
(x_1^2 -x_2^2).
$$

\begin{claim}
Let $\nu \in (0,1)$. There exists $a_0>0$ such that, for any
$a>a_0$, the map $F$ satisfies
hypotheses H1, H2 and H3 in $V_{\r}$, for some $\r>0$, where
\begin{equation*}
V=\{  x\in \RR^2 :  \nu| x_1|\leq  x_2\}.
\end{equation*}

Furthermore, for any $m,n \in\NN$ satisfying $n>a_0$ and $2m>n+1$,
the stable manifold over~$V$ of the origin with $a = 2n$ and
$b=2m-n-1$ is only $2m-2\ge 1$ times differentiable.
\end{claim}

\begin{proof} Let $\varphi( t, x)$ be the flow of
$\dot{x}= p( x)$ and $M( t, x)$ the solution of $\dot{M} =
q_1(\varphi( t, x)) M$ such that $M(0, x)=1$. The stable manifold (if
it exists) has to be the graph of $y=h( x)$ with
\begin{equation}\label{invmanifoldexemple}
h( x) = \int_{\infty}^0 M^{-1}( t, x) g(\varphi( t, x)) \,d t.
\end{equation}

We note that, for any value of $a>0$, the map $ x\mapsto  p( x)$ has
exactly five invariant lines in the set $\{x_2 \geq 0\}$
corresponding to the values of $\theta=0,\pi/4,\pi/2,3\pi/4,\pi$.

It is straightforward to check that, taking $a>0$ big enough, there
exists $\r>0$ small enough and a norm in $\RR^2$ such that $ p$
satisfies $H1, H3$ in $V_{\r}$ with the usual Euclidean norm $\Vert
v \Vert = \sqrt{v_1^2 + v_2^2}$.

Moreover, a simple computation shows that $0<\Ap<\bp=a$. We recall
that these constants were defined in~\eqref{defconstants}.

Using polar coordinates $(r,\theta)$ in the $(x_1,x_2)$-plane, $q_1$
and $g$ have the simpler expressions
$$
q_1(r) = br^4,\qquad g(r,\theta) = r^6 \sin 4\theta.
$$
In what follows we will write with the same letter a function $f(x)$
and its expression in polar coordinates $f(r,\theta)=f(r\cos \theta,
r\sin\theta)$.

The stable manifold over~$V$ of the origin (which exists and it is
$\CC^1$) is the graph of $y=h( x)$ with $h$ given in
\eqref{invmanifoldexemple}. Let $\varphi(t;r,\theta)$ be the flow
associated to \eqref{exH4polar} in polar coordinates. We denote
$\varphi_r$ and $\varphi_{\theta}$ the first and second component
respectively of $\varphi$ when written in polar coordinates. Then
$$
h(x) = \int_{\infty}^0 [M_y(t,x)]^{-1} \big
[\varphi_r(t;r,\theta)\big ]^6 \sin (4 \varphi_{\theta}(t;r,\theta))
\, dt
$$
with $M_y$ the solution of the linear system $\dot{M}_y = b
(\varphi_r (t;r,\theta))^4 M_y$ such that $M_y(0,r,\theta)=1$.

We first note that, if $\theta=\pi/4,\pi/2, 3\pi/4$ (that is, $x$
belongs to an invariant line), then
$\varphi_{\theta}(t;r,\theta)=\theta$ and consequently $\sin
(4\varphi_{\theta}(t;r,\theta))\equiv 0$. This implies that the
stable manifold evaluated at points with argument
$\theta=\pi/4,\pi/2,3\pi/4$ is $h(x) \equiv 0$. If the argument of
$x$, $\theta \neq \pi/4,\pi/2, 3\pi/4$,
$$
h(x) = 4 c_{\theta} \int_{\infty}^0 \frac{1}{\big (1+ 4a t
r^4)^{\frac{b}{4a} + \frac{6}{4} -\frac{1}{a} }\cdot \big
[c_{\theta}^2 + \big (1+ 4a t r^4)^{\frac{2}{a}}\big ]}\, dt,
$$
where
\[
c_\theta= \frac{1+ \cos(4\theta)}{\sin (4\theta)} =
\frac{x_1^2 -x_2^2}{2x_1 x_2} = c_{x}
\]
and $(r,\theta)$ are the polar coordinates of $x=(x_1,x_2)$.
In particular, if $\theta=\pi/4,3\pi/4$, then $c_{x}\equiv 0$ and
hence the above expression for $h$ is also valid in these invariant
lines.

We perform the change of variables $(1+4atr^4)^{2/a}= c_x^2/w$ and
we obtain that
$$
h(x) = -\frac{x_1^2 + x_2^2}{4  |c_{x}|^{a\big
(\frac{b}{4a} + \frac{6}{4} -\frac{1}{a}\big)}} \int_{0}^{c_{x}^2}
\frac{w^{\frac{a}{2} \big (\frac{b}{4a} + \frac{6}{4}
+\frac{1}{a}-1\big) }}{w(w+1)}\,dw.
$$

Now we take $m,n$ as in the claim and choose $a,b$ accordingly.
Then
$$
h(x) = -\frac{x_1^2 + x_2^2}{4  c_x^{2m-1}}
\int_{0}^{c_{x}^2} \frac{w^{m-1}}{w+1} \, dw.
$$
Using the elementary identity
$$
\frac{w^{m-1}}{w+1}= \sum_{j=2}^{m} (-1)^j w^{m-j} +
\frac{(-1)^{m+1}}{w+1},
$$
we obtain
$$
h(x) = -\frac{x_1^2 + x_2^2}{4 c_x^{2m-1}} \left
(\sum_{j=2}^m (-1)^{j} \frac{c_x^{2(m-j+1)}}{m-j+1} + (-1)^{m+1}
\log(c_x^2 +1)\right ).
$$
Now we are going to look for the differentiability of $h$ at points
of the form $(0,x_2)$, $x_2\neq 0$.  To determine the regularity
with respect to $x_1$, we only need to study the auxiliary function
$$
\tilde{h}(x) = x_1^{2m-1} \left (\sum_{j=2}^m (-1)^{j}
\frac{x_1^{-2(m-j+1)}}{m-j+1} + (-1)^{m+1} \log\left
(\frac{1}{x_1^{2}} +1\right)\right ).
$$
This function is only $2m-2\ge 1$ times differentiable.
\end{proof}

\section{Decomposition of $\Vr$}\label{decomposition}
In this section we describe a decomposition of the set $\Vr$
associated to a map of the form $\Ra( x) = x+ p(
x,0)+\OO(\|x\|^{N+1})$. Moreover we will obtain a quantitative
estimate of the rate of convergence of $\Vert \Ra^k( x) \Vert$ to $0$ as $k \to
\infty$.

We introduce the constant
\begin{equation*}
\a=\frac{1}{N-1}.
\end{equation*}
For a given $\r>0$, let $u>0$ and $a_0>0$ be such that $a_0
u^{-\a}=\r$. Consider two sequences $a_k\in \RR$, $k\ge 0$ and
$b_k\in \RR$, $k\ge 1$, such that
\begin{equation}\label{akbkcond}
\frac{b_{k+1}}{(u+k+1)^{\a}} < \frac{a_{k}}{(u+k)^{\a}}, \qquad k\ge
0.
\end{equation}
We introduce the sets
\begin{equation}\label{defVk}
V_{k}= \left \{  x\in \Vr \;\;:\;\; \Vert  x\Vert \in I_{k}:=\left [
\frac{b_{k+1}}{(u+k+1)^{\a}} , \frac{a_{k}}{(u+k)^{\a}}\right
]\right \}.
\end{equation}
\begin{lemma}\label{Rlemma}
Let $p$ be the homogeneous polynomial defined in~\eqref{defF}. Let
$\Ra:\Vr \to \RR^n$ be a continuous map such that $\Ra( x) -  x- p(
x,0) =\OO(\Vert x \Vert^{N+1})$.

Assume that $ p$ satisfies H1 and H3 and let $\ap\leq \bp$ be the
constants defined in \eqref{defconstants}.

Then for any $\ka<\ap$ and $\kb>\bp$, there exists $\r$ small enough
such that
\begin{enumerate}
\item[(1)] \label{firstitem:Rlemma} if $ x\in \Vr$,
\begin{equation*}
\Vert \Ra( x) - x\Vert \leq \kb \Vert  x\Vert^N,\qquad  \Vert  \Ra (
x) \Vert \leq \Vert  x\Vert \big ( 1-\ka \Vert  x\Vert^{N-1}\big ).
\end{equation*}
\item[(2)] \label{seconditem:Rlema} Let $a_0,b_0,u>0$ be such that $a_{0}^{N-1} = \a \ka^{-1}$, $b_{0}^{N-1} = \a \kb^{-1}$ and $a_0 u^{-\a}=\r$.
There exist two sequences $a_{k},b_{k}\in \RR$,  satisfying
\eqref{akbkcond}, such that $a_{k} = a_0\big (1+\OO(k^{-\beta})\big
)$, $b_k=b_0\big (1+\OO(k^{-\beta})\big )$ for some $\beta >0$.
Moreover
\begin{equation}\label{decompositionV}
\overline{\Vr}\backslash \{0\} = \cup_{k=0}^{\infty} V_k \;\;\;
\text{ and } \;\;\; \Ra(V_k) \subset V_{k+1}.
\end{equation}
Consequently, if $x\in V_k$, then one has that
$$
\frac{\a}{b(u+k+1+j)}\big (1+ \OO(k^{-\beta})\big ) \leq \Vert
\Ra^j(x) \Vert^{N-1} \leq \frac{\a}{a(u+k+j)} \big (1+
\OO(k^{-\beta}) \big ).
$$
\end{enumerate}
\end{lemma}
\begin{proof}
The proof of item~(1) is straightforward from the definitions of
$\ap$ and $\bp$.

Now we check (2). We define the auxiliary functions of real
variable, $\Ra_{\ka} (v) = v-\ka v^{N}$ and $\Ra_{\kb}(v)=v-\kb
v^{N}$. We first observe that, if $\r$ is small enough,
\begin{equation*}
\Ra_{\kb}(\Vert  x\Vert ) \leq \Vert \Ra( x) \Vert \leq
\Ra_{\ka}(\Vert  x\Vert ).
\end{equation*}
Indeed, the right hand side inequality follows from the definition
of $\ka$ and the left hand side inequality is a straightforward
consequence of the definition of $\kb$ and the triangular inequality
$\Vert \Ra ( x) \Vert \geq \Vert  x\Vert - \Vert \Ra( x) -  x\Vert$.

For $k\geq 0$ we define the sequences $a_{k}, b_{k}$ by the recurrences
\[
\frac{a_{k+1}}{(u+k+1)^{\a}}  = \Ra_{\ka} \left
(\frac{a_{k}}{(u+k)^{\a}}\right ), \qquad
\frac{b_{k+1}}{(u+k+1)^{\a}} = \Ra_{\kb} \left
(\frac{b_{k}}{(u+k)^{\a}}\right ), \qquad k\ge 0,
\]
and also $\bar{a}_{k}, \bar{b}_{k}$ by
\[
 \bar{a}_{k}=\frac{a_{k}}{(u+k)^{\a}}, \qquad
\bar{b}_{k}=\frac{b_{k}}{(u+k)^{\a}}, \qquad k\ge 0.
\]

We have that $a<b$. We choose $\r$ small enough such that  both
$\Ra_{\ka}$ and $\Ra_{\kb}$ are monotonically increasing functions
in $[0,\r]$ and $0< \Ra_{\kb}(v) < \Ra_{\ka}(v) < v$,  for $v\in
(0,\r]$. From the choice of $a_0, b_0$ we have $\bar{b}_0
<\bar{a}_0$ and $\bar{a}_0 = \r$. We easily check by induction
\begin{equation*}
0<\bar{b}_{k}<  \bar{a}_{k},\qquad \bar{a}_{k+1}<\bar{a}_k,\qquad
\bar{b}_{k+1}<\bar{b}_{k}\qquad \text{and} \qquad \lim_{k\to
\infty} (\bar{a}_k^2 +\bar{b}_k^2) =0.
\end{equation*}
As an immediate consequence, the sets $V_k$ in \eqref{defVk} are
well defined for this choice of sequences $b_k$ and $a_{k}$ and, in
addition, equality~\eqref{decompositionV} holds. Moreover, we
note that if $u\in I_{l}=[\bar{b}_{l+1},\bar{a}_{l}]$, then, by the
definition of the sequences $\bar{a}_k, \bar{b}_k$, since
$\Ra_{\ka}$ and $\Ra_{\kb}$ are increasing functions in $[0,\r]$ and
$\Ra_{\kb}(v) \leq \Ra_{\ka}(v)$,
\begin{align*}
\Ra_{\ka}(v) &\in [\Ra_{\ka}(\bar{b}_{l+1}),\Ra_{\ka}(\bar{a}_{l})]
\subset [\Ra_{\kb}(\bar{b}_{l+1}), \Ra_{\ka}(\bar{a}_{l})]
 =[\bar{b}_{l+2},\bar{a}_{l+1}]=I_{l+1},\\
 \Ra_{\kb}(v) &\in [\Ra_{\kb}(\bar{b}_{l+1}),\Ra_{\kb}(\bar{a}_{l})] \subset [\Ra_{\kb}(\bar{b}_{l+1}), \Ra_{\ka}(\bar{a}_{l})]
 =[\bar{b}_{l+2},\bar{a}_{l+1}]=I_{l+1}.
\end{align*}
Therefore, if $ x\in V_{l}$ (which is equivalent to $\Vert  x\Vert
\in I_{l}$), then $\Ra( x)\in I_{l+1}$ since $ R_{\kb}(\Vert  x\Vert
)\leq \Vert \Ra( x) \Vert \leq \Ra_{\ka}(\Vert  x\Vert)$.

In \cite{BFdLM2007} it was proven that there exist two analytic
function $\varphi_{\ka}, \varphi_{\kb}$ of the form
\begin{equation}\label{formvarphiAB}
\varphi_{\ka}(w)= \frac{a_{0}}{w^{\a}} + \OO\left
(\frac{1}{w^{\a+\beta}}\right ),\qquad \varphi_{\kb}(w)=
\frac{b_{0}}{w^{\a}} + \OO\left (\frac{1}{w^{\a+\beta}}\right )
\end{equation}
with $\beta>0$ which conjugate both $\Ra_{\ka}$ and $\Ra_{\kb}$ to
$w\mapsto w+1$, namely
\begin{equation}\label{conjugation}
\Ra_{\ka}(\varphi_{\ka}(w)) = \varphi_{\ka}(w+1),\qquad
\Ra_{\kb}(\varphi_{\kb}(w)) = \varphi_{\kb}(w+1).
\end{equation}
Let $w_k^{\ka}, w_{k}^{\kb}$ be such that $\varphi_{\ka}(w_k^{\ka})
(u+k)^{\a}=a_{k}$ and  $\varphi_{\kb}(w_k^{\kb})(u+k)^{\a}=b_{k}$.
We observe that, by definition of $a_{k},b_{k}$ and
\eqref{conjugation}
\begin{equation*}
\varphi_{\ka}(w_{k}^{\ka}) = \frac{a_{k}}{(u+k)^{\a}} =
\Ra_{\ka}\left (\frac{a_{k-1}}{(u+k-1)^{\a}}\right ) =
\Ra_{\ka}(\varphi_{\ka}(w_{k-1}^{\ka}))
=\varphi_{\ka}(w_{k-1}^{\ka}+1)
\end{equation*}
which implies (by the injective property of $\varphi_{\ka}$) that
$w_{k}^{\ka} = w_{k-1}^{\ka} + 1 = w_{0}^{\ka} + k$. Analogously one
can see that $w_{k}^{\kb}=w_0^{\kb} + k$. Now we notice that, by the
form \eqref{formvarphiAB} of $\varphi_{\ka},\varphi_{\kb}$, one has
that
\begin{equation*}
w_{0}^{\ka} = u + \OO(u^{1-\beta}),\qquad w_{0}^{\kb} = u +
\OO(u^{1-\beta}).
\end{equation*}
Therefore,
\begin{align*}
a_{k} &= (u+k)^{\a} \varphi_{\ka}(w_{k}^{\ka}) = (u+k)^{\a}
\varphi_{\ka}(w_0^{\ka} + k) \\ &= \frac{a_0(u+k)^{\a}}{\big
[u+k+\OO(u^{1-\beta})\big ]^{\a}} +
\OO\left (\frac{(u+k)^{\a} }{\big [u+k+\OO(u^{1-\beta})\big ]^{\a+\beta}}\right )\\
&=\frac{a_0}{\big [1+\OO\big (u^{1-\beta} (u+k)^{-1}\big )\big
]^{\a}} +
\OO\left (\frac{1}{(u+k)^{\beta} \big [1+ \OO\big (u^{1-\beta}(u+k)^{-1}\big)\big ]^{\a+\beta}}\right ) \\
&=a_0 + \OO\left (\frac{1}{(u+k)^{\beta}}\right ).
\end{align*}
Analogously, one checks that $b_{k} = b_0 + \OO((u+k)^{-\beta})$ and
the proof of the lemma is concluded.
\end{proof}
\begin{remark}
Note that as a simple consequence of this technical lemma, we have
that for $ x\in \Vr$, $\Ra^k(x) \to 0$ as $k\to \infty$. Hence, if
we are able to prove the existence of a parametrization $K$
satisfying the invariance equation $F\circ K - K \circ \Ra=0$,
since $F^k(K(x))= K(\Ra^k(x))$, the image of~$K$ will represent a
subset of the stable invariant manifold .
\end{remark}

\section{The invariant manifold. The differentiable case}\label{sec:theinvariantmanifold}

In this section we prove Theorem~\ref{maintheorem} in the
differentiable case. This is accomplished by stating and solving a
fixed point equation in some appropriate Banach spaces. The proof
follows along the same lines of the equivalent result in
\cite{BFdLM2007}, but there are technical differences that prevent
to apply directly that proof. However, these differences are not
important enough to justify the inclusion of the whole proof. For
this reason, in this section we include a series of technical
lemmas, equivalent to those in~\cite{BFdLM2007}, with the suitable
hypothesis in our current case. We sketch their proofs when they are
different enough of their counterpart in~\cite{BFdLM2007}. The
existence of the manifold follows directly from this set of lemmas.

Along this section we will assume that all the hypotheses of
Theorem~\ref{maintheorem} hold. We will denote by $C$ a positive constant which
may take different values at different places.

\subsection{Preliminary facts}\label{preliminaryfacts}
We take $\k\in \NN$ such that $\ri<\k\leq r$ with $\ri$ introduced
in~\eqref{defri} and we decompose our map $ F$ into
\begin{equation*}
F(x,y)=\FT(x,y) + \FR{\k}(x,y),
\end{equation*}
where $\FT$ is the Taylor expansion of $ F$ up to degree~$\k-1$ and
$\FR{\k}(x)=\oo(\Vert x \Vert^{\k-1})$. In fact, since $\k\leq r$,
we actually have that $\FR{\k}(x)=\OO(\Vert x \Vert^{\k})$.  By
hypothesis, there exist $\Kl{\leq}$ and $ R$, $\CC^{\rfm}$ functions
such that
\begin{equation}\label{eqFTKla}
\FT \circ \Kl{\leq} - \Kl{\leq} \circ  R= \TT{\k},\qquad \TT{\k}(x)
= \OO(\Vert x \Vert)^{\k}.
\end{equation}

Since $\FT$ is a polynomial and $\Kl{\leq}, R$ satisfy item (c) in
Theorem~\ref{maintheorem}, the remainder $\TT{\k}$ satisfies
\begin{equation*}
D^j \TT{\k}(x)=\OO(\Vert x \Vert^{\k-j}), \qquad  j=0,\cdots, \rfm.
\end{equation*}
Finally, using that $D^j \FR{\k}$ is the Taylor's remainder of $D^j
F$,
\begin{equation*}
D^j \FR{\k}(x,y)=\OO(\Vert (x,y) \Vert^{\k-j}),\qquad j=0,\cdots, r.
\end{equation*}
We will use these simple facts without special mention.

As a consequence of \eqref{eqFTKla}, the purpose of this section is
to prove that  there is only one solution $\Kl{>}$ of
\begin{equation}
\label{eqFTKfull}  F \circ (\Kl{\leq} + \Kl{>}) - (\Kl{\leq} +
\Kl{>}) \circ  R = 0.
\end{equation}
We will see that equation \eqref{eqFTKfull} can be rewritten as a
fixed point equation. Then, a solution of this fixed point equation
will be found.

\subsection{The Banach spaces and the main statement}
Given $E$ a Banach space, we will denote
\begin{equation*}
\X{\nu }{k}(E) = \{ h:\Vr\subset \RR^{n} \to E : h \in
\mathcal{C}^{\nu}, \; \max_{0\leq j \leq \nu} \sup_{ x\in \Vr}
\frac{\Vert D^j h( x) \Vert }{\Vert  x\Vert^{k-j\eta}} <\infty\}
\end{equation*}
with $\eta =1-L+N$ defined in~\eqref{defeta}. This quantity was
already introduced in \cite{BFdLM2007}, jointly with  a
motivating example showing that, if $K^> (x) =\OO(\Vert
x \Vert^{k})$, then $D K^> (x)$ is not necessarily $\OO(\Vert x \Vert^{k-1})$.

With this definition, if $h\in \X{\nu}{k}(E)$, then $Dh \in
\X{\nu-1}{k-\eta}(L(\RR^{n};E))$. Thus we
understand by $\Vert D^{j} h( x) \Vert $ the norm of the $j$-linear
map induced by the norm in $E$.

We endow $\X{\nu}{k}$ with the norm
\begin{equation*}
\Vert h \Vert_{\nu,k} = \max_{0\leq j \leq \nu} \sup_{x\in \Vr}
\frac{\Vert D^j h( x) \Vert }{\Vert  x\Vert^{k-j\eta}}
\end{equation*}
and it becomes a Banach space. We denote by
$\BB_{k}^{\nu}(\varsigma)\subset \X{\nu}{k}$ the open ball of radius
$\varsigma$.

\begin{proposition}\label{prop:fixedpoint} Assume all the conditions in Theorem~\ref{maintheorem}.
Let $\k \in \NN$ be such that $\ri<\ell\leq r$ (the case $r=\infty$
is included). Then there exists $\varsigma_*>0$ such that for any
$\varsigma\geq \varsigma_*$ there exists $\r$ small enough such that
equation \eqref{eqFTKfull} has a unique solution $\Kl{>}:\Vr \to
\RR^{n+m}$ belonging to $\BB^{\rvl}_{\k-N+1}(\varsigma)$ with $\rvl
\leq \min\{r,\rfm\}$ and satisfying
\begin{equation*}
\rvl  \max \left \{ \eta- \frac{\Ap}{\ddp},0 \right\}
<\k-N+1-\frac{\Bp}{\ap}.
\end{equation*}
\end{proposition}
Note that when $\eta \ddp \leq \Ap$, the maximum differentiability
degree is $\rvl = \min \{ r,\rfm\}$. In addition $\rv=\rvl$ for
$\k=r$ is the value stated in Theorem~\ref{maintheorem}.

In the next sections we prove this proposition by using  the same
scheme as in \cite{BFdLM2007}.

Next proposition proves the uniqueness statement of Theorem~\ref{maintheorem}. This
proposition ends the proof of Theorem~\ref{maintheorem} in the
differentiable case.
\begin{proposition}\label{cor:fixedpoint} Assume the hypotheses of Proposition~\ref{prop:fixedpoint}.
We denote by $\r_*>0$ the corresponding quantity provided in
Proposition~\ref{prop:fixedpoint} for the radius $\varsigma_*$. Then
equation~\eqref{eqFTKfull} has a unique solution $\Kl{>}:V_{\r_*}
\to \RR^n$ in $\X{\rvl}{\k-N+1}$.
\end{proposition}
\begin{proof}
Let $K_1=\Kl{\leq} + \Kl{>}_1$ and $K_2 = \Kl{\leq}+\Kl{>}_2$ be two
solutions of the invariance equation $F\circ K = K\circ R$ with
$\Kl{>}_1,\Kl{>}_2 \in \X{\rvl}{\k-N+1}$. We denote by $V_{\r_0}$
their common domain (all the suprema will be taken in this domain) and we
consider
$$
\varsigma= \varsigma_* + \max\{\Vert \Kl{>}_1
\Vert_{\rvl,\k-N+1},\Vert \Kl{>}_2
\Vert_{\rvl,\k-N+1}\}>\varsigma_*.
$$
By Proposition~\ref{prop:fixedpoint}, there exists $\r\leq \r_0$
small enough and a unique function $\Kl{>}:\Vr \to \RR^{n+m}$,
belonging to $\X{\rvl}{\k-N+1}$ with norm $\Vert \Kl{>}
\Vert_{\rvl,\k-N+1}\leq \varsigma$. Since, for $i=1,2$, $\Vert
\Kl{>}_i\Vert_{\rvl,\k-N+1}<\varsigma$ they have to coincide in
$\Vr$. In addition, we can extend $\Kl{>}$ to $V_{\r_*}$ by using
the invariance equation. Indeed, let $K=\Kl{\leq}+ \Kl{>}$. First,
we notice that by~(2) of Lemma~\ref{Rlemma}, there exists~$k$ such
that $R^k(V_{\r_*}\backslash \Vr) \subset \Vr$. Second, the relation
$K=F^{-k} \circ K \circ R^k$ extends $K$ to $V_{\r_*}$ and the
result is proven.
\end{proof}

In $\RR^{n+m}$ we will
use the norm
\begin{equation}\label{defnormxy}
\Vert (x,y) \Vert = \max \{ \Vert x \Vert , \Vert y \Vert\},\qquad
(x,y) \in \RR^{n+m},
\end{equation}
where the chosen norms in $\RR^{n}$ and $\RR^{m}$ are such that hypotheses
H1, H2, and H3 hold.

\subsection{A compilation of technical lemmas}

The lemmas in this section are the translation to our current
setting of the lemmas in \cite{BFdLM2007}.

We first present the following elementary properties of the Banach
spaces $\X{\nu}{k}$.

\begin{lemma} \label{Banachprop}
The Banach spaces $\X{\nu}{k}$ satisfy:
\begin{enumerate}
\item[(1)] Let $f( x) \in L(X_1,X_2)$ with $f\in \X{\nu}{k}$ and $g( x) \in X_1$ with
$g\in  \X{\nu}{l}$, then $f \cdot  g \in  \X{\nu}{k+l}$ and $\Vert f
\cdot g\Vert_{\nu,k+l} \leq 2^{\nu}\Vert f\Vert_{\nu,k} \Vert g
\Vert_{\nu,l}$.
\item[(2)] Let $f:U\subset \RR^{n+m} \to E$ be a $\CC^{\nu}$ map, with $E$ a Banach space, such that
$\Vert D^l f(x) \Vert =\OO(\Vert x \Vert^{j-l})$ for all $0\leq
l\leq \nu$. Then, for any map $g :\Vr \to U$ such that $g\in
\X{i}{1}$ for some $0\leq i \leq \nu$ we have that $f\circ g \in
\X{i}{j}$.
\end{enumerate}
\end{lemma}

For any $\ka<\ap,\kb>\bp$, we define the auxiliary constant
\begin{equation*}
\kd  =
\begin{cases}
\;\;\ka,  & \text{if}\;\; \Ap\leq 0,\\
\;\;\kb, &  \text{otherwise}.
\end{cases}
\end{equation*}

From now on we fix values $\ka<\ap$, $\kb>\bp$ and $\B>\Bp$ such
that if either a) $\Ap>\eta \ddp$, or b) $\ddp <\Ap<\eta \ddp$ or c)
$\Ap < \ddp$ then a) $\Ap>\eta \kd$,  b) $\kd< \Ap < \eta \kd$ or c)
$\Ap<\kd$ respectively. We also choose the constants $\ka,\kb$ such
that the cases $\Ap=\eta \kd$ or $\Ap=\kd$ can be skipped even when
either $\Ap=\eta \ddp=\eta \bp$ or $\Ap=\ddp$ respectively. Below we
introduce~$\lo$ and we further impose that
\begin{equation}\label{KAKBB}
\begin{aligned}
\ri < \lo:=N-1 + \frac{\B}{\ka}+\max \left \{\eta-\frac{\Ap}{\kd},0\right \}& <\ell \leq r, \\
\k-N+1-\frac{\B}{\ka}-\rvl \max \left \{\eta-\frac{\Ap}{\kd},0\right
\}& >0.
\end{aligned}
\end{equation}
The first property holds because $\ri<\k\leq r$. The second one
holds by the definition of $\rvl$ in Proposition~\ref{prop:fixedpoint}.
The constant $\lo$ depends on the values $\ka,\kb,\B$ but it
can be chosen arbitrarily close to $\ri$ (see~\eqref{defri} for the
definition of $\ri$).

\subsubsection{Scaling}\label{sec:scaling}
We perform a scaling in the $y$-variables by the change
$S_{\delta}(x,y) = (x, \delta y)$. Then, equations~\eqref{eqFTKla}
and~\eqref{eqFTKfull} become
\begin{equation}\label{eqFTKltilde}
\tilde{\FT} \circ \Klt{\leq} - \Klt{\leq} \circ  R= \tilde{T}_{\ell}
\end{equation}
and
\begin{equation}\label{eqFTKfulltilde} \tilde{ F} \circ (\Klt{\leq} + \tilde{\Kl{>}}) - (\Klt{\leq} +
\tilde{\Kl{>}}) \circ  R = 0,
\end{equation}
where $\tilde{\FT} = S^{-1}_{\delta} \circ \FT \circ S_{\delta}$,
$\tilde{ F} = S^{-1}_{\delta} \circ  F \circ S_{\delta}$,
$\Klt{\leq} = S^{-1}_{\delta} \circ \Kl{\leq} $ and $\Klt{>} =
S^{-1}_{\delta} \circ \Kl{>} $.

We observe that
\[
\tilde{\FT}_x (x,y) = x +  p(x,0) +  p(x,\delta y) - p(x,0) +
\hat{f}(x,\delta y),
\]
where, by hypothesis, $\tilde p (x,y) =  p(x,\delta y) - p(x,0)$ is
a homogeneous polynomial of degree~$N$ and $\hat{f}(x,\delta y)
=\OO(\Vert (x,\delta y)\Vert^{N+1})$ . We have that $\tilde p (x,y)
= \hat p_{N-1} (x,y) y$, where
\[
\hat p_{N-1}(x,y) = \delta \int_0^1 D_y p(x,\tau \delta y)\,d\tau
\]
is a matrix whose coefficients are homogeneous polynomials of
degree~$N-1$. It satisfies $\hat p_{N-1}(x,0) = \delta D_y p(x,0)$.

\begin{lemma}
\label{lem:weakexpansion} With~$\B$ given in~\eqref{KAKBB}, there
exist $\r,\delta>0$ small enough such that
\[
\|(D\tilde{\FT})^{-1}(\tilde{\Kl{\leq}}(x))\| \le
1+\B\|x\|^{N-1},\qquad \text{for all  }x\in \Vr .
\]
\end{lemma}
\begin{proof}
The proof of this lemma is analogous to Lemma~4.5 in
\cite{BFdLM2007}. However, we sketch it. Let $\r>0$ be such that
$\r^{1/2} \delta^{-1} <1$. Taking into account the above
considerations about the scaling,
the norm of the matrix $(D\tilde{\FT})^{-1}(\tilde{\Kl{\leq}}(x))$
is
$$
\Vert (D\tilde{\FT})^{-1}(\tilde{\Kl{\leq}}(x)) \Vert \leq \max\{1+
(\Bp + \OO(\r)+\OO(\delta)) \Vert x \Vert^{N-1}, 1-(\Bq +
\OO(\delta^{-1}\r))\Vert x \Vert^{M-1}\}.
$$
Recall that we are using in $\RR^{n}$ the norm given
in~\eqref{defnormxy}. Since $\r^{1/2} \delta^{-1}<1$, taking
$\r,\delta$ small enough, the constant $\B$ in~\eqref{KAKBB}
satisfies
$$
\Vert (D\tilde{\FT})^{-1}(\tilde{\Kl{\leq}}(x)) \Vert \leq \max\{1+
\B \Vert x \Vert^{N-1}, 1-(\Bq + \OO(\r^{1/2}))\Vert x \Vert^{M-1}\}.
$$
To obtain the result, we need to check that $\B\Vert x \Vert^{N-1}
\geq -(\Bq + \OO(\r^{1/2}))\Vert x \Vert^{M-1}$. If $N\neq M$, the
result follows from H2 and the smallness of~$\r$. The case $N=M$,
follows from  $-\Bq + \OO(\r^{1/2}) < N\ap + \OO(\r^{1/2})\leq \Bp +
\OO(\r^{1/2})$,  by~H2 and Lemma~\ref{defKapwell}. Again taking $\r$
small enough, we are done.
\end{proof}
From now on, we suppress the ``tilde'' from the scaled functions.

We fix $\delta,\r>0$ small enough and $\ka,\kb,\B$ in \eqref{KAKBB}
such that the conclusions of Lemma~\ref{Rlemma} applied to $R$ and
Lemma~\ref{lem:weakexpansion} hold true.

\subsubsection{Weak contraction of the nonlinear terms}\label{subsecweakcontraction}

Since the fixed point is parabolic there is no contraction from the
linear part of the map at the point. In the following lemma we
measure the contraction provided by the nonlinear terms.

\begin{lemma}
\label{lem:weakcontraction} Let $V_k\subset \Vr$ be the sets defined
in~\eqref{defVk}. There exists a constant $C>0$, depending only on
$\delta,\r$ and $\ell$ (which are fixed a priori), such that for any
$k \ge 0$, $x\in V_k$ and $i\ge 0$
\begin{align}
\label{bound:weakcontraction1} \prod_{m=0}^i \|
(D\FT)^{-1}(\Kl{\leq}( R^m(x)))\| & \le C \left(
\frac{u+k+i}{u+k}\right)^{\a \B\ka^{-1}}, \\
\label{bound:weakcontraction2}
\| D[(D\FT)^{-1} \circ \Kl{\leq}](x)\| & \le C (u+k)^{-\a(L-2)}, \\
\label{bound:weakcontraction3} \|DR^i(x)\| \le \prod_{m=0}^{i-1}\| D
R \circ  R^m(x)\| & \le  C \left(
\frac{u+k}{u+k+i} \right)^{\a \Ap  \kd^{-1}}. \\
\intertext{Finally, if $\Ap<\kd$}
\label{bound:weakcontraction4Ap<bp} \| D^2  R^i(x)\| & \le
C(u+k+i)^{\alpha} \left( \frac{u+k}{u+k+i} \right)^{2\a \Ap
\kd^{-1}}
\\ \intertext{and in the case $\Ap > \kd=\kb$}
\label{bound:weakcontraction4Ap>bp} \| D^2  R^i(x)\|  &\le
C(u+k)^{\alpha} \left( \frac{u+k}{u+k+i} \right)^{\a \Ap \kd^{-1}}.
\end{align}
\end{lemma}
\begin{remark}
The proof of this lemma is analogous to the one of Lemma~4.6 in
\cite{BFdLM2007} using the estimates for $\Vert  R^i(x) \Vert$ given
in Lemma~\ref{Rlemma}. However, the exponents in
inequalities~\eqref{bound:weakcontraction1},
\eqref{bound:weakcontraction3}--\eqref{bound:weakcontraction4Ap>bp}
are different from their counterpart in \cite{BFdLM2007} due to the
fact that here the invariant manifold is not one dimensional. In
particular, the constant analogous to $\Ap \kd^{-1}$ was exactly $N$
in~\cite{BFdLM2007}. We also are forced to separate the cases $\Ap<
\kd$ and $\Ap > \kd$ in the bound of $\Vert D^2 R^j(x)\Vert$.
\end{remark}
\begin{proof}
We begin with~\eqref{bound:weakcontraction1}. By Lemma~\ref{Rlemma},
if $x\in V_k$, then $ R^m(x)\in V_{k+m}$. Therefore, using
Lemma~\ref{lem:weakexpansion} and  item~(2) of Lemma~\ref{Rlemma} we
have that
$$
\Vert (D\FT)^{-1}(\Kl{\leq}( R^{m}(x)))\Vert \leq 1 + \frac{\a\B
}{\ka(u+k+m)}  \big (1+ \OO((k+m)^{-\beta}) \big ),
$$
for $x\in V_k$. Then, since
\begin{align*}
 \sum_{m=0}^i \log \left (\|
(D\FT)^{-1}(\Kl{\leq}( R^m(x)))\|\right )& \leq \sum_{m=0}^{i} \log \left (1+\frac{\a\B}{\ka(u+k+m)}  \big (1+ \OO((k+m)^{-\beta}) \big )\right )\\
&=\frac{\a\B}{\ka}\sum_{m=0}^i \frac{1}{u+k+m}\big (1+ \OO((k+m)^{-\beta}) \big ) \\
&=\frac{\a\B}{\ka} \left[\log \left (\frac{u+k+i}{u+k}\right )
+\OO(k^{-\beta}) \right] ,
\end{align*}
and~\eqref{bound:weakcontraction1} is proven.

Bound~\eqref{bound:weakcontraction2} is an straightforward
computation. To prove estimate~\eqref{bound:weakcontraction3} we
first notice that since $R(x)=x+p(x,0)+\OO(\Vert x \Vert^{N+1})$, by
Lemma~\ref{Rlemma}, if $x\in V_k$,
$$
\Vert DR(x) \Vert \leq 1 - \frac{ \a \Ap }{\kd (u+k)} +
\frac{C}{(u+k)^{1+\beta}}.
$$
Then, again using Lemma~\ref{Rlemma},
$$
\Vert D R^i(x) \Vert \leq \prod_{m=0}^{i-1}\| D R \circ  R^m(x)\|
\leq \prod_{m=0}^{i-1} \left (1-\frac{\a\Ap }{\kd(u+k+m)} +
\frac{C}{(u+k+m)^{1+\beta}}\right ).
$$
Finally, estimate~\eqref{bound:weakcontraction3} follows from the
fact that
$$
\sum_{m=0}^{i-1} \log \left(1-\frac{\a \Ap}{\kd(u+k+m)} +
\frac{C}{(u+k+m)^{1+\beta}}\right ) \leq \frac{ \a\Ap}{\kd}\log
\left (\frac{u+k}{u+k+i}\right) +  \frac{C}{(u+k)^{1+\beta}}.
$$

To bound $\Vert D^2 R^i(x)\Vert$ we first note that
\[
\|D^2  R^i (x) \| = \| D(\prod_{m=0}^{i-1} D R \circ  R^m) \| \le
\sum_{m=0}^{i-1} \|D^2  R \circ  R^m \|\|D R^m\|\prod_{l=0}^{i-1} \|
D R \circ  R^l\| \|D R \circ  R^m\|^{-1}.
\]
Then, taking into account that $\|D R \circ  R^m(x)\| \ge 1/2$ and
that,
\[
\| D^2  R ( R^m(x))\| \le C \| R^m(x)\|^{N-2},
\]
using again~\ref{seconditem:Rlema} of Lemma~\ref{Rlemma},  we have
that
\[
\begin{aligned}
\|D^2  R^i (x) \| \le & C\prod_{l=0}^{i-1} \| D R \circ  R^l\|
\sum_{m=0}^{i-1} \| R^m(x)\|^{N-2}\|D R^m\| \\
\le & C \prod_{l=0}^{i-1} \| D R \circ  R^l\| \sum_{m=0}^{i-1}
\frac{(u+k)^{\a\Ap  \kd^{-1}}}{(u+k+m)^{ \a\Ap \kd^{-1}+\a(N-2)}}.
\end{aligned}
\]
Now we distinguish two cases. First, when $\Ap > \kd=\kb$, we have
$\a \Ap \kd^{-1}+\a(N-2)>1$ and then
$$
\sum_{m=0}^{i-1} \frac{(u+k)^{\a \Ap \kd^{-1}}}{(u+k+m)^{\a\Ap
\kd^{-1}+\alpha(N-2)}} \leq C (u+k)^{\a}.
$$
This bound together with~\eqref{bound:weakcontraction3},
implies~\eqref{bound:weakcontraction4Ap>bp} in this case. On the
other hand, when $\Ap<\kd$,
$$
\sum_{m=0}^{i-1} \frac{(u+k)^{\a\Ap  \kd^{-1}}}{(u+k+m)^{\a\Ap
\kd^{-1}+\a(N-2)}} \leq C (u+k)^{\a} (u+k+i)^{\a(1-\Ap\kd^{-1})}
$$
and, using again~\eqref{bound:weakcontraction3}, we
get~\eqref{bound:weakcontraction4Ap<bp}.
\end{proof}

\subsubsection{Operators for higher order derivatives and their inverses}\label{subsecoperatorsS}
Now we proceed to rewrite equation \eqref{eqFTKfulltilde}, which we
recall here
\begin{equation}
\label{eqFTKfulloperator}  F \circ (\Kl{\leq} + \Kl{>}) - (\Kl{\leq}
+ \Kl{>}) \circ  R = 0,
\end{equation}
 as a fixed point equation. We skip the symbol $\tilde{}$ of our notation, although we work
with the rescaled map. That is, since $\Kl{\leq}$ satisfies
\eqref{eqFTKltilde}: $ \FT \circ \Kl{\leq} - \Kl{\leq} \circ R =
\TT{\ell}$, $\Kl{>}$ has to satisfy
\begin{multline*}
\big (D \FT \circ \Kl{\leq} \big ) \Kl{>} - \Kl{>} \circ  R = \\
-\TT{\ell} - \FR{\ell} \circ \big (\Kl{\leq} + \Kl{>}\big) - \FT
\circ \big(\Kl{\leq} + \Kl{>}\big) + \FT \circ \Kl{\leq} + \big
(D\FT \circ \Kl{\leq} \big ) \Kl{>}.
\end{multline*}
To shorten the notation, we introduce the operators
\begin{equation}\label{defL0}
\L{0}( S) = \big (D \FT \circ \Kl{\leq} \big )  S -  S \circ  R
\end{equation}
and
\begin{equation}
\label{defFcalparam} \FF(K) = -\TT{\k} - \FR{\k} \circ \big
(\Kl{\leq} + K\big) - \FT \circ \big(\Kl{\leq} + K\big) + \FT \circ
\Kl{\leq} + \big (D\FT \circ \Kl{\leq} \big ) K.
\end{equation}
Then equation~\eqref{eqFTKfulloperator} for~$\Kl{>}$ can be
rewritten as
\begin{equation}
\label{eq:Kmes} \L{0}(\Kl{>}) = \FF(\Kl{>}).
\end{equation}
The formal inverse of $\L{0}$ is
\begin{equation}\label{defS0param}
\S{0}( T) = \sum_{i=0}^{\infty} \left [ \prod_{m=0}^i \big (D
\FT)^{-1} \circ \Kl{\leq} \circ  R^{m} \right ]  T \circ  R^i
\end{equation}
and consequently, we can formally write equation \eqref{eq:Kmes} as
the fixed point equation
\begin{equation}\label{fpeqfirst}
\Kl{>} = \S{0} \circ \FF (\Kl{>}).
\end{equation}
Following the same arguments as the ones in the proof of Lemma 4.9
in \cite{BFdLM2007}, one can check that the operator $\S{0} :
\X{0}{\ell} \to  \X{0}{\ell-N+1}$ is continuous. Therefore, the
operator $\L{0}$, introduced in~\eqref{defL0}, is suitable to prove
the existence of a continuous invariant manifold. In order to obtain
the higher order derivatives, we  introduce the operators
\begin{equation*}
\L{j}( S) = \big (D \FT \circ \Kl{\leq} \big )  S -  S \circ  R (D
R)^j, \qquad j\ge 1.
\end{equation*}
The key property is that if $ S$ is a $\CC^{\nu}$ solution of
$\L{0}( S) = T$, with $T$ a $\CC^{\nu}$ function, then $D^j  S$ is a
solution of
\begin{equation*}
\L{j}(H) = T^j, \qquad 0 \le j \le \nu ,
\end{equation*}
where $T^j$ is defined by the recurrence relation
\begin{equation*}
\begin{aligned}
T^0 & = T, \\
T^{j+1} & = DT^j - D(D\FT \circ \Kl{\leq} ) D^j  S + j D^j  S \circ
 R (D R)^{j-1} D^2  R.
\end{aligned}
\end{equation*}

Recall the parameters $L = \min\{N,M\}$ and $\eta = 1+N-L$ defined
in \eqref{defeta}. The following lemma is analogous to Lemma~4.7
in~\cite{BFdLM2007}, with an appropriate change in the hypothesis.
\begin{lemma}
\label{lem:Ljonetoone} Let $\k>N-1 +\B\ka^{-1}$ and $j\geq 0$ be
such that
$$
\k-N+1 -\frac{\B}{\ka} - j\left (\eta-\frac{\Ap}{\kd} \right)>0.
$$
Then, the operators $\L{j}: \X{0}{\k-N+1-j\eta} \to C^0$, $j\ge 0$,
are well defined, continuous and one to one.
\end{lemma}
\begin{proof}
Since $ R(x) = x+p(x,0)+ \OO(\|x\|^{N+1})$, $N\ge 2$, and $\ap>0$,
$\|R(x)\|\le \|x\|$ and then $\L{j}$ is well defined and continuous.

Let $j\ge 0$ and $ S \in  \X{0}{\k-N+1-j\eta}$ be such that $\L{j}(
S) = 0$, that is, $ S = \big (D \FT \circ \Kl{\leq}  \big )^{-1} S
\circ  R (D R)^j$, or, using this condition iteratively,
\begin{equation*}
S = \left( \prod_{m=0}^i \big (D \FT \big )^{-1}\circ \Kl{\leq}
\circ  R^{m} \right)  S \circ  R^{i+1} (D R^{i+1})^j, \qquad i \ge
0.
\end{equation*}
Now, using that $\Vert S\circ R^{i+1}(x) \Vert \leq C \Vert S
\Vert_{0,\k-N+1-j\eta} \Vert R^{i+1}(x) \Vert^{\k-N+1-j\eta}$ and
Lemmas~\ref{Rlemma} and~\ref{lem:weakcontraction}, we obtain that,
for $x \in V_k$,
$$
\Vert S(x) \Vert \leq C \Vert S \Vert_{0,\k-N+1-j\eta}
\frac{(u+k)^{\a(j\Ap\kd^{-1} -
\B\ka^{-1})}}{(u+k+i)^{\a(\k-N+1-\B\ka^{-1}-j(\eta-\Ap\kd^{-1}))}}.
$$
By hypothesis, the right hand side of the above expression tends to~$0$ when $i$ tends to~$\infty$, which implies that $S=0$ and,
consequently, that $\L{j}$ is one to one.
\end{proof}

A formal inverse of the operator~$\L{j}$ obtained recursively from
$\L{j} ( S) = T$ is given by the formula
\begin{equation}
\label{def:Sj} \S{j} (T) = \sum_{i\ge 0} \left( \prod_{m=0}^i
(D\FT)^{-1} \circ \Kl{\leq} \circ  R^m \right) T \circ  R^i \cdot
(DR^i)^j.
\end{equation}
Notice that $\S{j}$ acts on $j$-linear maps. If this formula is
absolutely convergent, it is a simple computation to check that
$\L{j}(\S{j}(T)) = T$.

In the next lemma, equivalent to Lemma~4.9 in \cite{BFdLM2007} with
adjusted hypotheses, we check that $\S{j}$ is indeed well defined
and bounded between appropriate spaces.

\begin{lemma}
\label{lem:Skonlowregularity} Assume that $\k> N-1+\B\ka^{-1}$ and
that $j\geq 0$ satisfies
$$
\k  -N+1 - \frac{\B}{\ka} -j \left (\eta-\frac{\Ap}{\kd}\right )>0.
$$
Then the operator $\S{j}:  \X{0}{\k-j\eta} \to  \X{0}{\k-j\eta-N+1}$
is well defined and bounded. Also we have $\L{j} \circ \S{j} = \Id$
on $ \X{0}{\k-j\eta}$.

Moreover, if $\k>\lo$, with $\lo$ defined in \eqref{KAKBB} and
$j\geq 0$ is such that
$$
\k-\lo -j \left (\eta-\frac{\Ap}{\kd}\right )>0,
$$
the operator $\S{j}:  \X{1}{\k-j\eta} \to  \X{1}{\k-j\eta-N+1}$ is
well defined and
\[
D\left( \S{j}(T)\right) = \S{j+1}(\tilde T),\qquad \text{if $ T\in
\X{1}{\k-j\eta}$,}
\]
where
\[
\tilde T = DT - D(D\FT \circ \Kl{\leq}) \S{j}(T) + j \S{j}(T) \circ
 R(D R)^{j-1} D^2  R.
\]
\end{lemma}
\begin{proof}
Let $T\in  \X{0}{\k-j\eta}$ and $S=\S{j}(T)$. Following the same
lines as the ones in the proof of Lemma~4.9 in \cite{BFdLM2007}, a
direct computation shows that, for $x\in V_k$,
$$
\Vert S(x)\Vert \leq C \Vert T \Vert_{0,\k-j\eta} \sum_{i\geq 0}
\frac{(u+k)^{\a(j\Ap \kd^{-1} - \B\ka^{-1})}}{(u+k+i)^{\a(\k
-j(\eta-\Ap\kd^{-1}) -\B\ka^{-1})}}.
$$
Therefore, since by hypothesis $\a(\k -j(\eta-\Ap\kd^{-1})
-\B\ka^{-1})>1$,  if $x\in V_k$,
$$
\Vert S(x) \Vert \leq C \Vert T \Vert_{0,\k-j\eta}
(u+k)^{-\a(\k-j\eta-N+1)}\leq C \Vert x \Vert^{\k-j\eta-N+1} \Vert T
\Vert_{0,\k-j\eta}.
$$
Hence $\Vert S\Vert_{0,\k-j\eta-N+1} \leq \Vert
T\Vert_{0,\k-j\eta}$, that is, $\S{j}:  \X{0}{\k-j\eta} \to
\X{0}{\k-j\eta-N+1}$ is well defined and bounded. It also proves
that $\L{j} \circ \S{j} = \Id$ on $ \X{0}{\k-j\eta}$.

Following the proof of Lemma~4.9 in \cite{BFdLM2007}, we argue that,
if $\S{j}(T)$ is differentiable and its derivative belongs to $
\X{0}{\k-(j+1)\eta -N+1}$, then $D\S{j}(T) = \S{j+1}(\tilde T)$. The
trick is to check that both are solutions of the same equation
$\L{j+1}(H)=\tilde T$ belonging to $ \X{0}{\k-(j+1)\eta-N+1}$.
Indeed, first we note that if $T \in  \X{1}{\k-j\eta}$, then $\tilde
T \in  \X{0}{\k-(j+1)\eta}$ provided $DT \in \X{0}{\k-(j+1)\eta}$,
$D(D\FT \circ \Kl{\leq}) \in
 \X{0}{L-2}$, $D^2  R\in  \X{0}{N-2}$ and the definition of~$\eta$.
This implies that $\S{j+1}(\tilde T)\in  \X{0}{\k-(j+1)\eta-N+1}$.
It only remains to check that $D(\S{j}(T))$ is a solution of
$\L{j+1}(H)=\tilde T$ which can be proven by taking derivatives in
$\L{j}(\S{j}(T))=T$. Hence the uniqueness result,
Lemma~\ref{lem:Ljonetoone}, proves that $D(\S{j}(T)) =
\S{j+1}(\tilde T)$.

Now we prove that $\S{j}(T)$ is differentiable and belongs to~$
\X{0}{\ell-(j+1)\eta-N+1}$. In order to do so, we take derivatives
formally in~\eqref{def:Sj}. We have $D(\S{j}(T)) = S_1+S_2+S_3$,
where
\begin{equation*}
\begin{aligned}
S_1 = & \sum_{i\ge 0} \left(\prod_{m=0}^i (D\FT)^{-1}
\circ \Kl{\leq} \circ  R^m \right) D\T \circ  R^i (D R^i)^{j+1},\\
S_2 = & \sum_{i\ge 0} \left(\prod_{m=0}^i (D\FT)^{-1}
\circ \Kl{\leq} \circ  R^m \right) \T \circ  R^i j(D R^i)^{j-1} D^2 R^i,\\
S_3 = & \sum_{i\ge 0} \sum_{m=0}^i\left(\prod_{l =0}^{m-1}
(D\FT)^{-1} \circ \Kl{\leq} \circ  R^l \right) D\left((D\FT)^{-1}
\circ \Kl{\leq} \circ  R^m\right)
\\ & \times \left(\prod_{l =m+1}^{i}
(D\FT)^{-1} \circ \Kl{\leq} \circ  R^l \right)\T \circ  R^i (D
R^i)^{j},
\end{aligned}
\end{equation*}
and check that the above expressions are absolutely convergent,
belong to~$ \X{0}{\k-(j+1)\eta-N+1}$ and are bounded.

Since $DT\in
 \X{0}{\k-(j+1)\eta}$, then, by the first part of the lemma,  $S_1 = \S{j+1}(DT)$ belongs to
$ \X{0}{\k-(j+1)\eta-N+1}$ and we are done with $S_1$.

Next we deal with $S_2$. Let $x\in V_k$. Assume that $\Ap>\kd=\kb$.
Then, by Lemma~\ref{lem:weakcontraction}, we have that:
$$
\Vert S_2(x) \Vert \leq C  \Vert T \Vert_{1,\k-j\eta} \sum_{i\geq 0}
\frac{(u+k)^{\a(j\Ap\kd^{-1}+1-\B \ka^{-1})}}{(u+k+i)^{\a(\k -j(\eta
-\Ap\kd^{-1})-\B \ka^{-1})}}.
$$
Since $\k-j(\eta-\Ap\kd^{-1})-\B\ka^{-1} > N-1$, the sum is
convergent and we obtain
$$
\Vert S_{2}(x) \Vert \leq C\frac{ \Vert T \Vert_{1,\k-j\eta}
}{(u+k)^{\a(\k-j\eta -N)}} = C \Vert x \Vert^{\k-j\eta-N} \Vert T
\Vert_{1,\k-j\eta} \leq C\Vert x \Vert^{\k-(j+1)\eta -N+1}
\r^{\eta-1} \Vert T \Vert_{1,\k-j\eta}
$$
which implies that $S_2\in  \X{0}{\k-N+1-(j+1)\eta}$. Here we have
used that $\eta\geq 1$. If $\Ap<\kb$, then again by
Lemma~\ref{lem:weakcontraction},
$$
\Vert S_2(x) \Vert \leq C \Vert T \Vert_{1,\k-j\eta} \sum_{i\geq 0}
\frac{(u+k)^{\a((j+1)\Ap\kd^{-1}-\B \ka^{-1})}}{(u+k+i)^{\a(\k
-(j+1)(\eta -\Ap\kd^{-1})-\B \ka^{-1} +\eta-1)}}.
$$
Proceeding as in the previous case, one gets that $ \Vert S_{2}(x)
\Vert \leq K\Vert x \Vert^{\k-(j+1)\eta -N+1} \r^{\eta-1} \Vert T
\Vert_{1,\k-j\eta}$ and the study for $S_2$ is finished.

Finally we consider $S_3$. Using again
Lemma~\ref{lem:weakcontraction}, if $x\in V_k$ we have that
$$
\Vert S_3(x) \Vert \leq C \Vert T \Vert_{1,\k-j\eta} \sum_{i\geq 0}
\frac{(u+k)^{\a((j+1)\Ap\kd^{-1} - \B
\ka^{-1})}}{(u+k+i)^{\a(\k-j(\eta -\Ap\kd^{-1})-\B\ka^{-1})}}
\sum_{m=0}^i \frac{1}{(u+k+m)^{1-\a(\eta- \Ap \kd^{-1})}}.
$$
We have different estimates for the sum with respect to $m$ if
either $\Ap\kd^{-1} \leq \eta$ or $\Ap\kd^{-1} > \eta$.
Nevertheless, the sum with respect to $i$ is convergent
provided~$\k$ satisfies the current hypothesis. Performing
straightforward computations, we obtain that
$$
\Vert S_3(x) \Vert \leq C\frac{ \Vert T \Vert_{1,\k-j\eta}
}{(u+k)^{\a(\k-(j+1)\eta-N+1)}} \leq C \Vert T \Vert_{1,\k-j\eta}
\Vert x \Vert^{\k-(j+1)\eta-N+1}
$$
and the lemma is proven.
\end{proof}
The last result of this section is the following.

\begin{proposition}
\label{prop:S0S1higherregularity} Let $\rfm$ be the
differentiability degree of $\Kl{\leq}$ and $R$ assumed in
Theorem~\ref{maintheorem}. Take $\k>N-1+\B \ka^{-1}$ and $\ss$ such
that $0\le \ss \le \rfm$ and
\begin{equation}\label{defrmax}
\k -N+1-\frac{\B}{\ka} - \ss\max \left \{ \eta-
\frac{\Ap}{\kd},0\right \}>0.
\end{equation}
Then,
\[
\S{0}: \X{\ss}{\k} \to  \X{\ss}{\k-N+1} \quad \text{ and } \quad
\S{1}: \X{\ss}{\k-\eta} \to  \X{\ss}{\k-\eta-N+1}
\]
are bounded linear operators.
\end{proposition}
\begin{proof}
The proof of this proposition is analogous to the corresponding one
Proposition~4.10 in~\cite{BFdLM2007}. Let $T\in \X{\ss}{\k} \subset
\X{0}{\k}$. The key point of the proof is to deduce that
\begin{equation}
\label{eq:DSjm1Tjm1} D[\S{j-1}(T^{j-1})] = \S{j}(T^j), \qquad 1 \le
j \le \ss,
\end{equation}
being $\{T^j\}_{0\leq j \leq \ss}$ the sequence defined inductively
by
\[
\begin{aligned}
T^0 = & T,\\
T^{j+1} = & DT^{j} - D(D\FT \circ \Kl{\leq}) \S{j}(T^{j}) + j
\S{j}(T^{j}) \circ  R(D R)^{j-1} D^2  R,
\end{aligned}
\]
for $0\le j \le \ss-1$. Indeed, one checks by induction that $T^j$
belongs to $ \X{1}{\k-j\eta}$ if $j\le \ss-1$. For $j=\ss$ we have
that $T^{\ss}\in  \X{0}{\k-\ss\eta}$ and therefore, by
Lemma~\ref{lem:Skonlowregularity}, $\S{j}(T^j)\in
\X{1}{\k-j\eta-N+1}$ and $\S{\ss}(T^{\ss})\in \X{0}{\k-\ss\eta}$.
Note that, if $j\leq \ss-1$ with $\ss$ satisfying \eqref{defrmax},
then
$$
\k -\lo -j \left (\eta -\frac{\Ap}{\kd}\right ) \geq \k-N+1
-\frac{\B}{\ka} -(j+1)\max\left \{ \eta- \frac{\Ap}{\kd},0\right
\}>0.
$$
Then, for $j\leq \ss-1$, the results of
Lemma~\ref{lem:Skonlowregularity} on the operators
$\S{j}:\X{1}{\k-j\eta} \to \X{1}{\k-j\eta -N+1}$ apply.

Applying iteratively~\eqref{eq:DSjm1Tjm1} we have that
$D^j[\S{0}(T)] = S^{j}(T^j) \in  \X{0}{\k-j\eta-N+1}$, for $j\le
\ss$ and, hence, $\S{0}(T) \in  \X{\ss}{\k-N+1}$.

Finally, to prove that the operator $\S{0} :  \X{\ss}{\k}\to
\X{\ss}{\k-N+1}$ is bounded, we refer the reader
to~\cite{BFdLM2007}, Proposition~4.10. The proof that $\S{1}:
\X{\ss}{\k-\eta } \to  \X{\ss}{\k-\eta-N+1}$ is also bounded is very
similar to the one for $\S{0}$.
\end{proof}
\subsection{End of the proof of Proposition~\ref{prop:fixedpoint}. Fixed point equation}
Using Proposition~\ref{prop:S0S1higherregularity} we are able to
prove that the fixed point equation~\eqref{fpeqfirst},
\begin{equation}
\label{def:fpe} \Kl{>} = \S{0} \circ \FF(\Kl{>}),
\end{equation}
is well defined in the appropriate Banach spaces and it is a
contraction. Concretely, we prove Proposition~\ref{prop:fixedpoint}.
That is, that there exist a unique solution $\Kl{>}$ of
equation~\eqref{def:fpe} belonging to $\X{\rvl}{\k-N+1}$ for any
$\ri<\k\leq r$. To do so, we follow the same steps as the ones in
Section~4.10 of~\cite{BFdLM2007}, . We sketch them without proofs,
only given the essential information. The main tool is
Lemma~\ref{Banachprop}.

Let $\r>0$ be such that all the results in
Sections~\ref{subsecoperatorsS} and~\ref{subsecweakcontraction} are
valid. Recall that we fix this quantity at the end of
Section~\ref{sec:scaling} satisfying the results in
Section~\ref{decomposition} and~\eqref{KAKBB} for $\ka,\kb$ and
$\B$.

Since $\FF(0)=-\TT{\k}-\FR{\k} \circ \Kl{\leq}$, using that
$\TT{\k}$ and $\FR{\k}$ are $\CC^r$ functions, that $\Kl{\leq}$ is a
$\CC^{\rfm}$ function and the definition of $\rvl$, we have that
$\FF(0)\in \CC^{\min\{r,\rfm\}}\subset \CC^{\rvl}$. Then $\FF(0)\in
\X{\rvl}{\k}$ and, since $\nu=\rvl$ satisfies the condition stated
in Proposition~\ref{prop:S0S1higherregularity}, see~\eqref{KAKBB},
$\S{0}\circ \FF(0)\in \X{\rvl}{\k-N+1}$. We also have that
$$
\|\S{0} \circ \FF(0)\|_{\rvl,\k-N+1}\| \le
\|\S{0}\|\big(\|\TT\k\|_{\rvl,\k}+\|\FR{\k} \circ
\Kl{\leq}\|_{\rvl,\k}\big ) =:\frac{\varsigma_*}{2}.
$$
Since the domain of $\Kl{\leq}$ is $\Vr\subset \Vro$ we will work
with this domain in the spaces  $\X{\nu}{k}$.

We will find the solution $K^{>}$ of equation~\eqref{def:fpe} in
$\BB_{\k-N+1}^{\rvl-1,\varsigma} \subset
 \X{\rvl-1}{\k-N+1}$, the ball of radius $\varsigma\geq \varsigma_*$. First we note that for any $\varsigma\geq \varsigma_*$
there exists $\r'$ small enough such that if $\Kl{>}\in
\BB_{\k-N+1}^{\rvl-1,\varsigma}$ and $x\in V_{\r'}$, then
$(\Kl{\leq}+\Kl{>})(x) \in U$, the domain of $F$. Indeed, we deduce
this property because $U$ is an open set, $\text{dist}(V_{\r'},
\partial U)>0$ and $\Vert(\Kl{\leq}+\Kl{>})(x)-x \Vert \leq C
(\r')^2$ with $C>0$ a constant. Note that $\r'$ depends on $\varsigma$,
$\r$ and $\Kl{\leq}$.

As usual in the differentiable case, we first prove the existence of
a solution belonging to $\CC^{\rvl-1}$ defined on $V_{\r'}$. To do
so, it only remains to check that the operator $\FF:
\BB_{\k-N+1}^{\rvl-1,\varsigma} \to \X{\rvl-1}{\k-N+1}$ is a contraction.
The proof of this result follows from the analogous result
 in \cite{BFdLM2007} and in fact we obtain the same bound
for the Lipschitz constant
$$
\text{lip}(\FF) \leq C (\r')^{\k- 2N-L},
$$
with $C$ independent of $\r'$, but depending on $\varsigma$ and $\r$.

As a consequence, equation~\eqref{def:fpe} has a solution $\Kl{>} :
V_{\r'}\to \RR^{n+m}$. Applying the linear operator $\L{0}$ we
obtain that equation~\eqref{eq:Kmes} has a unique differentiable
solution $\Kl{>}$. This implies that $K=\Kl{\leq}+\Kl{>}$ and $ R$
are $\CC^{\rvl-1}$ solutions of the invariance equation
\eqref{eqFTKfull}.

Following the same arguments as the ones given in \cite{BFdLM2007}
we deduce that, if $r=\infty$ the parametrization
$K=\Kl{\leq}+\Kl{>}$ is also a $\CC^{\infty}$ as well as $ R$ is.
Moreover, the arguments to prove the sharp regularity can be also
applied in this new context. Hence we obtain $\CC^{\rvl}$
parametrizations.

Until now the function $K=\Kl{\leq}+\Kl{>}$ is defined on $V_{\r'}$
with $\r'\leq \r$. However, since $\r$ is small enough to assure
that $R$ satisfies the conclusions of Lemma~\ref{Rlemma}, we can use
the invariance equation to extend the domain of $K$ to $\Vr$ as we
did in the proof of Corollary~\ref{cor:fixedpoint}. Indeed, let
$k\in \NN$ be such that $R^k(\Vr\backslash V_{\r'})\subset V_{\r'}$.
Then $K= F^{-k} \circ K \circ R^k$ extends $K$ to $\Vr$.

\begin{remark}\label{dependencerho}
We have proven that the domain $\Vr$ of $K$ and $R$ depends on $\k$,
$\Kl{\leq}$ and on the constants $\ka,\kb,\B$ as well as
$\ap,\bp,\Ap, \Bp$.
\end{remark}

\section{Dependence on parameters}
In this section we prove Theorem~\ref{maintheoremparam} about the
dependence of the invariant manifold on parameters. Along this
section we will assume all the conditions stated in this theorem. We
will proceed in a similar way as in the proof of
Theorem~\ref{maintheorem}.
\subsection{Preliminary facts. Consequences of the previous results}
As a consequence of Lemma~\ref{lemmaconstantslambda}, if Hypothesis
H${\lambda}$ holds true, then H1, H2 and H3 are satisfied for any
$\lambda \in \Lambda$. Then, using Proposition~\ref{prop:fixedpoint}
with $\k=r$, we have that for any $\lambda\in \Lambda$, there exists
$\r_{\lambda}$ such that the invariance equation
\begin{equation*}
F(\Kl{\leq}(x,\lambda)+ \Kl{>}(x,\lambda),\lambda) -
\Kl{\leq}(R(x,\lambda),\lambda)+ \Kl{>}(R(x,\lambda),\lambda)=0
\end{equation*}
has a solution $\Kl{>}(\cdot, \lambda)\in \X{\rvl}{\k-N+1}$ defined
on $V_{\r_{\lambda}}$. However we emphasize that
\begin{itemize}
\item the degree of differentiability $\rvl$ does not depend on
$\lambda$ and
\item $\r_{\lambda}$ can be taken independent on $\lambda$ provided the constants $\ap, \bp, \Ap, \Bp,\Bq$ are
independent on the parameter (see Remark~\ref{dependencerho}). Then
$\Kl{>}$ is defined over $\Vr \times \Lambda$.
\end{itemize}
In addition, we already know that for any $\lambda$, $\Kl{>}(\cdot,
\lambda)$ is the unique solution belonging to $\X{\rvl}{\k-N+1}$ of
the fixed point equation~\eqref{def:fpe}:
\begin{equation}\label{def:fpe:lambda}
\Kl{>} = \S{0} \circ \FF(\Kl{>})
\end{equation}
being~$\S{0}$ and~$\FF$ defined by~\eqref{defS0param}
and~\eqref{defFcalparam}, respectively.

It is important to remark that all the functions
involved, $\FT, \TT{\k},\FR{\k}, \Kl{\leq},  K, R, $ and $T$, depend
on both, $x,\lambda$, but, abusing notation, we only indicate the
composition with respect to the $x$ variable. For instance $R^2$
means $R(R(x,\lambda),\lambda)$ and $\FR{\k} \circ \big (\Kl{\leq} + K\big)$ means $\FR{\k}
(\Kl{\leq}(x,\lambda) + K (x,\lambda),\lambda)$.

We restate Theorem~\ref{maintheoremparam} in a functional setting
using the space $\CC^{\Sigma_{\sigma,\nu}}$ introduced
in~\eqref{defCDS}. We also introduce the Banach space
\begin{equation*}
\XD{\sigma}{\nu}{k} = \big \{ f: \mathcal{U} \times \Lambda\to
\RR^{l} : f\in \CC^{\Sigma_{\sigma,\nu-\sigma}} \; \max_{i,j\in
\Sigma_{\sigma,\nu-\sigma}} \sup_{ (x,\lambda) \in \mathcal{U}
\times \Lambda} \frac{\Vert \Dl^i \Dx^j f(x,\lambda) \Vert }{\Vert
x\Vert^{k+i-(i+j)\eta}} <\infty\}
\end{equation*}
for $\nu\geq \sigma$, endowed with the norm
\begin{equation*}
\Vert f \Vert_{\nu,k}^{\sigma} = \max_{i,j\in
\Sigma_{\sigma,\nu-\sigma}} \sup_{(x,\lambda)\in \mathcal{U}\times
\Lambda} \frac{\Vert \Dl^i \Dx^j f(x,\lambda) \Vert }{\Vert
x\Vert^{k+i-(i+j)\eta}}.
\end{equation*}
Note that $\XD{\sigma}{\nu+\sigma}{k} \subset
\CC^{\Sigma_{\sigma,\nu}}$. The differentiability conclusions of
Theorem~\ref{maintheoremparam} are a direct consequence of the
following proposition.
\begin{proposition}\label{prop:theoremparam}
Assume all the conditions in Theorem~\ref{maintheoremparam}. Let $\k
\in \NN$ be such that $\max\{\ri,\ril\}<\k\leq \rx$ with $\ri$ and
$\ril$ defined in~\eqref{defri} and~\eqref{defril} respectively.
Then the solution $\Kl{>}:\Vr \times \Lambda \to \RR^{n+m}$ of the
fixed point equation~\eqref{def:fpe:lambda} belongs to
$\XD{s^>_{\k}}{\nu_{\k}^>}{\k-N+1}$ with $\nu_{\k}^> = \rvl +
s^>_{\k}$, $\rvl\leq \min\{\rx,\rfm\}$, $s^>_{\k}\leq \min\{\rl,
\rl^{\leq}\}$ and
$$
\k-N+1 -\frac{\Bp}{\ap} -(\nu_{\k}^>-i)\max \left \{\eta -
\frac{\Ap}{\ddp},0\right \}  - i (\eta-1)>0,\qquad 0\leq i \leq
s^>_{\k}.
$$
\end{proposition}

The remaining part of this section is devoted to prove this result.
The procedure is similar to the one we have followed in
Section~\ref{sec:theinvariantmanifold}. First we study the product
and composition of functions belonging to the functional spaces
$\XD{\sigma}{\nu}{k}$. Then, we study the linear operator $\S{0}$ defined on
$\XD{\ll}{\ss}{\k}$ and, finally, we apply the fixed point theorem
to obtain a solution $\Kl{>}$ of the fixed point
equation~\eqref{def:fpe:lambda} belonging to $\XD{\ll}{\ss}{\k-N+1}$
with appropriate values of $\ll$ and $\ss$. With standard arguments we check
the sharp regularity of the solutions.
\subsection{Technical lemmas}
Next lemma, whose proof we skip, is analogous to
Lemma~\ref{Banachprop} for~$\XD{\sigma}{\nu}{k}$.
\begin{lemma}
The Banach spaces $\XD{\sigma}{\nu}{k}$
satisfy:
\begin{enumerate}
\item[(1)] Let $f( x,\lambda) \in L(X_1,X_2)$ with $f\in \XD{\sigma}{\nu}{k}$ and $g( x,\lambda) \in X_1$ with
$g\in  \XD{\sigma}{\nu}{l}$, then $f \cdot  g \in
\XD{\sigma}{\nu}{k+l}$ and $\Vert f \cdot g\Vert_{\nu,k+l}^{\sigma}
\leq 2^{\nu}\Vert f\Vert_{\nu,k}^{\sigma} \Vert g
\Vert_{\nu,l}^{\sigma}$.
\item[(2)] Let $f:U\times \Lambda \subset \RR^{n+m+n'} \to E$ be a $\CC^{\Sigma_{\sigma,\nu-\sigma}}$ map
and~$E$ a Banach space such that $\Vert \Dl^{l'}\Dx^l f(x,\lambda)
\Vert =\OO(\Vert x \Vert^{j-l})$ for all $(l',l)\in
\Sigma_{\sigma,\nu-\sigma}$. Then, for any map $g :\Vr\times \Lambda
\to U$ such that $g\in  \XD{i'}{i}{1}$ for some $(i',i)\in
\Sigma_{\sigma,\nu}$ we have that $f\circ (g,\Id) \in
\XD{i'}{i}{j}$.
\end{enumerate}
\end{lemma}

We need to stablish the dependence on $\lambda$ of $\S{0}(K)$.

\subsubsection{Differentiability with respect to $\lambda$ of the linear operator $\S{0}$}
We first note that all the results stated in the previous sections
are valid uniformly in $\lambda\in \Lambda$ for functions belonging
to $\XD{0}{\ss}{\k}$. This is due to Hypothesis HP and to the fact
that the constants $\ap, \bp,$ etcetera, defined
in~\eqref{defconstantslambda} are independents of $\lambda\in
\Lambda$ and therefore, all the bounds in the previous sections are
uniform with respect to $\lambda \in \Lambda$. The uniformity  with respect to $\lambda\in \Lambda$ of
Lemmas~\ref{lem:weakcontraction} and~\ref{lem:Ljonetoone} and
Proposition~\ref{prop:S0S1higherregularity} is summarized in the following lemma.
\begin{lemma}\label{uniformlambda}
We have that:
\begin{itemize}
\item [(i)]
All the bounds in~Lemma~\ref{lem:weakcontraction} hold true with
constants $C$ independent of $\lambda\in \Lambda$.
\item [(ii)]
Under the hypotheses of Lemma~\ref{lem:Ljonetoone}, the formula
$$
\L{0}(S) = \big (D \FT \circ \Kl{\leq} \big )  S -  S \circ  R
$$
defines an operator $\L{0} : \XD{0}{0}{\k-N+1} \to \CC^0$,
continuous and one to one.
\item[(iii)]
If the conditions for $\ss, \k$ of
Proposition~\ref{prop:S0S1higherregularity} are satisfied, then
$$
\S{0}: \XD{0}{\ss}{\k} \to  \XD{0}{\ss}{\k-N+1} \quad \text{ and }
\quad \S{1}: \XD{0}{\ss}{\k-\eta} \to  \XD{0}{\ss}{\k-\eta-N+1}
$$
are bounded linear operators.
\end{itemize}
\end{lemma}
Now we state and prove the  differentiability results with respect
to the parameter $\lambda$.
\begin{lemma} Let $\k, \ss,\ll$ be
such that $\ri<\k\leq r$, $\ll \leq \rl^{\leq}$, $1\leq \ll\leq \ss
\leq \rfm + \rl^{\leq}$ and
\begin{equation}\label{condlparam}
\k -N+1 -\frac{\B}{\ka} - (\ss-i)\max \left \{\eta-
\frac{\Ap}{\kd},0\right \}-i(\eta-1)>0,\qquad 0\leq i\leq \sigma.
\end{equation}
We have that:
\begin{enumerate}
\item[(1)] (Low order regularity) The linear operator $\S{0}: \XD{1}{\ss}{\k} \to \XD{1}{\ss}{\k-N+1}$
is bounded if $\k,\ss$ satisfy condition~\eqref{condlparam} with
$\sigma=1$. In addition,
\begin{equation}
\label{derlambdaS0} \Dl \S{0}(T) = \S{0}(\tilde{T})
\end{equation}
with
\begin{equation*}
\tilde{T} = -\Dl (DP \circ \Kl{\leq}) \S{0}(T) + \Dx [\S{0}(T)]
\circ R \cdot \Dl R + \Dl T.
\end{equation*}
\item[(2)] (Higher order regularity) The linear operator $\S{0}: \XD{\ll}{\ss}{\k} \to \XD{\ll}{\ss}{\k-N+1}$ is bounded.
\end{enumerate}
\end{lemma}
\begin{proof}
We have to check first that for any $T\in \XD{1}{\ss}{\k-N+1}$,
$$
\S{0}(T) \in \XD{0}{\ss}{\k-N+1},\qquad \Dl \S{0}(T) \in
\XD{0}{\ss-1}{\k-N+1-(\eta-1)}.
$$
The first relation, which corresponds to $\sigma=0$, follows from
Lemma~\ref{uniformlambda}. To deal with the second one, we proceed
as in the proof of Lemma \ref{lem:Skonlowregularity}. We take
derivatives  with respect to $\lambda$ formally and we check that
the different factors we obtain, which will be infinite sums, are
absolutely convergent, belong to~$\XD{0}{\ss-1}{\k-N+1-(\eta-1)}$
and are bounded. Indeed, we formally decompose $\Dl \S{0}(T) =
S_1+S_2$ with
\begin{align*}
S_1 = &\sum_{i=0}^{\infty} \left [ \prod_{j=0}^i \big (D\FT)^{-1} \circ \Kl{\leq} \circ  R^{j} \right ]
 \big [\Dl T \circ  R^i + (\Dx T\circ R^i )\Dl R^i\big ] \\
S_2= & \sum_{i=0}^{\infty} \sum_{m=0}^i\left(\prod_{l =0}^{m-1}
(D\FT)^{-1} \circ \Kl{\leq} \circ  R^l \right) \Dl\left((D\FT)^{-1}
\circ \Kl{\leq} \circ  R^m\right)
\\ & \times \left(\prod_{l =m+1}^{i} (D\FT)^{-1} \circ \Kl{\leq} \circ  R^l \right)\T \circ  R^i.
\end{align*}

It can be checked by induction that, if $i\ge 2$,
\begin{equation*}
\Dl R^i = \Dl R \circ R^{i-1} + \sum_{j=1}^{i-1} (\Dx R^j \circ
R^{i-j}) \Dl R \circ R^{i-j-1}.
\end{equation*}
Therefore, from item (i) of Lemma~\ref{uniformlambda}, if
$(x,\lambda)\in V_k\times \Lambda$,
\begin{align*}
\Vert \Dl R^i  (x,\lambda) \Vert  &\leq \frac{C}{(u+k+i)^{\a
N}}+\frac{C}{(u+k+i)^{\a \Ap\kd^{-1}}}\sum_{j=1}^{i-1}
\frac{1}{(u+k+j)^{\a (N-\Ap \kd^{-1})}} \\
&\leq \frac{C}{(u+k+i)^{\a \Ap\kd^{-1}}(u+k)^{\a (N-\Ap
\kd^{-1})-1}}+ \frac{C}{(u+k+i)^{\a N-1}}
\end{align*}
with $C$ independent of $\lambda$. Then, if $x\in V_k$ and $\lambda
\in \Lambda$, using the definition of $\S{0}$,
\begin{align*}
\Vert S_1(x,\lambda) \Vert & \leq C \Vert \S{0}(\Dl T)(x)\Vert + C
\Vert T \Vert_{\ss,\k}^{1}\frac{1}{(u+k)^{\a \B \ka^{-1}}}
\\
& \times \sum_{i=0}^{\infty}\left(\frac{(u+k)^{-\a
(1-\Ap\kd^{-1})}} {(u+k+i)^{\a(\k -\eta-\B
\ka^{-1}+\Ap\kd^{-1})}}+\frac{1} {(u+k+i)^{\a(\k -\eta-1)}}\right)
\\
&\leq C \Vert T
\Vert_{\ss,\k}^{1}(u+k)^{-\a(\k-N+1-(\eta-1))},
\end{align*}
where we have used (iii) of Lemma~\ref{uniformlambda} to bound
$\Vert \S{0}(\Dl T)(x)\Vert$. Then,
\begin{equation}\label{difparamS1}
\Vert S_1(x,\lambda) \Vert \leq C \Vert T \Vert_{\ss,\k}^{1}\Vert x
\Vert^{\k -N+1-(\eta-1)},\qquad x\in \Vr
\end{equation}
uniformly in $\lambda \in \Lambda$.

To deal with $S_2$, we first note that if $x\in V_k$ and $m\in \NN$,
then
\begin{equation*}
\Vert \Dl \big ( (\Dx P)^{-1} \circ \Kl{\leq} \circ R^m \big )
(x,\lambda) \Vert \leq \frac{C}{(u+k+m)^{\a(L-1)}}.
\end{equation*}
Then,  using Lemma~\ref{lem:weakcontraction},
\begin{equation*}
\Vert S_2(x,\lambda) \Vert \leq C \Vert T \Vert_{\ss,\k}^{1}
\sum_{i=0}^{\infty} \frac{(u+k)^{-\a \B \ka^{-1}}}{(u+k+i)^{\a(\k -
\B \ka^{-1})}} \sum_{m=0}^i \frac{1}{(u+k+m)^{\a(L-1)}}.
\end{equation*}
If $\a(L-1)<1$, then, since $\eta = 1+N-L$,
\begin{equation*}
\Vert S_2(x,\lambda) \Vert \leq C \Vert T
\Vert_{\ss,\k}^{1}\sum_{i=0}^{\infty} \frac{(u+k)^{-\a \B
\ka^{-1}}}{(u+k+i)^{\a(\k - \B \ka^{-1}+ L-1 -N+1)}} \leq C\Vert T
\Vert_{\ss,\k}^{1}(u+k)^{-\a( \k -\eta +1 -N+1)}
\end{equation*}
and we are done in this case. When $\a(L-1)=1$, in other words
$\eta=1$, we take a positive quantity $\varepsilon>0$, such that
$\a(L-1+ \varepsilon)> 1$ and $\k -\B \ka^{-1}-\varepsilon >N-1$
(this last condition can be fulfilled by hypothesis). Then
\begin{align*}
\Vert S_2(x,\lambda)\Vert &\leq C \Vert T
\Vert_{\ss,\k}^{1}\sum_{i=0}^{\infty} \frac{(u+k)^{-\a \B
\ka^{-1}}}{(u+k+i)^{\a(\k - \B \ka^{-1}-\varepsilon)}} \sum_{m=0}^i
\frac{1}{(u+k+m)^{1+\a \varepsilon} }\\ & \leq C \Vert T
\Vert_{\ss,\k}^{1}\sum_{i=0}^{\infty} \frac{(u+k)^{-\a (\B
\ka^{-1}+\varepsilon)}} {(u+k+i)^{\a(\k - \B \ka^{-1}-\varepsilon)}}
\leq C \Vert T \Vert_{\ss,\k}^{1}(u+k)^{-\a(\k -N+1)}.
\end{align*}
In any case, $\Vert S_2 (x,\lambda) \Vert \leq C \Vert T
\Vert_{\ss,\k}^{1}\Vert x \Vert^{\k-N+1-(\eta-1)}$. This bound
together with the corresponding one for $S_1$ in \eqref{difparamS1},
leads us to conclude that $\Dl \S{0}(T) \in
\XD{0}{0}{\k-N+1-(\eta-1)}$.

On the one hand, $\Dl \S{0}(T)$ and $\S{0}(\tilde{T})$ belong to
$\XD{0}{0}{\k-N+1-(\eta-1)}$ and both are solutions of the same
linear equation $\L{0} H = \tilde{T}$. Since, by (ii) of
Lemma~\ref{uniformlambda}, $\L{0}$ is  injective,
$$\Dl\S{0}(T)=\S{0}(\tilde{T}).$$
On the other hand, it is clear that $\tilde{T} \in
\XD{0}{\ss-1}{\k-\eta+1}$ because $T \in \XD{1}{\ss}{\k}$.
Therefore, using (iii) of Lemma~\ref{uniformlambda},
$\S{0}(\tilde{T}) \in \XD{0}{\ss-1}{\k-\eta+1-N+1}$ and
consequently, $\Dl\S{0}(T)$ belongs to
$\XD{0}{\ss-1}{\k-\eta+1-N+1}$. This ends the proof of the first
item of the lemma.

To deal with the second item, we perform an induction procedure. Let
$T\in \XD{\ll}{\ss}{\k}$ and $S=\S{0}(T)$. We have to prove that $S
\in \XD{\sigma}{\ss}{\k-N+1}$. The cases $\sigma=0,1$ are already
proven. Assume that $S\in \XD{\sigma-1}{\ss}{\k-N+1}$ for
$\sigma\leq \rl^{\leq}$. We define recursively for $0\leq i \leq
\sigma-1$:
\begin{align*}
S^i &=\Dl^i S, \\
T^i &=-\Dl (DP \circ \Kl{\leq}) S^{i-1}+ \Dx S^{i-1} \circ R \cdot
\Dl R + \Dl T^{i-1}.
\end{align*}
Note that, using~\eqref{derlambdaS0}, $S^i=\Dl^i S =\S{0}(T^i)$.
Moreover, since $S\in \XD{\sigma-1}{\ss}{\k-N+1}$,
$$
S^{i} \in \XD{\sigma-1-i}{\ss-i}{\k-N+1-i(\eta-1)},\qquad \Dx
S^{i-1} \XD{\sigma-i}{\ss-i}{\k-N+1-\eta-i(\eta-1)}.
$$
Using that $\eta=N-L+1$ and the above properties,
$$
\Dl (DP \circ \Kl{\leq}) S^{i-1} \in
\XD{\sigma-i}{\ss-i+1}{\k-i(\eta-1)},\qquad \Dx S^{i-1}\circ R \cdot
\Dl R \in \XD{\sigma-i}{\ss-i}{\k-i(\eta-1)}
$$
so that, by recurrence one gets $T^i \in
\XD{\sigma-i}{\ss-i}{\k-i(\eta-1)}$, if $0\leq i \leq \sigma-1$. We
take now $i=\sigma-1$ and we obtain that
$$
T^{\sigma-1} \in \XD{1}{\ss-(\sigma-1)}{\k-(\sigma-1)(\eta-1)}.
$$
Using 1) we deduce that $\Dl^{\sigma-1} S=
S^{\sigma-1}=\S{0}(T^{\sigma-1})\in
\XD{1}{\ss-(\sigma-1)}{\k-(\sigma-1)(\eta-1)-N+1}$ and therefore,
$$
\Dl^{\sigma} S =\Dl \S{0}(T^{\sigma-1}) \in
\XD{0}{\ss-\sigma}{\k-\sigma(\eta-1)-N+1}
$$
which implies that $S \in \XD{\sigma}{\ss}{\k-N+1}$.
\end{proof}
\subsection{End of the proof of Proposition~\ref{prop:theoremparam}}
We point out that,
since $\Kl{\leq}$ and $R$ satisfy~b) of
Theorem~\ref{maintheoremparam},  if $(x,\lambda)\in \Vr \times
\Lambda$,
\begin{equation*}
\Dl^i\Dx^j \TT{\k}(x,\lambda)=\OO(\Vert x \Vert^{\k-j}), \qquad
(i,j)\in \Sigma_{\rl^{\leq},\rfm},
\end{equation*}
and, since $\FR{\k}$ is the Taylor's remainder (with respect to the
$(x,y)$ variable) of $F\in \CC^{\DS}$,
\begin{equation*}
\Dl^i\Dx^j \FR{\k}(x,y,\lambda)=\OO(\Vert (x,y) \Vert^{\k-j}),\qquad
(i,j)\in \DS .
\end{equation*}
Moreover, these bounds are uniform on $\lambda \in \Lambda$.

Standard arguments allows us to apply the fixed point theorem to
obtain the existence of a solution $\Kl{>}$ of the fixed point
equation~\eqref{def:fpe:lambda} belonging to
$\XD{\rl^>_{\k}}{\nu_{\k}-1}{\k-N+1}$. Finally we recover the last
derivative as in the analogous result in \cite{BFdLM2007}.

\section{The analytic case}
In this section we deal with the conclusions of
Theorem~\ref{maintheorem} and \ref{maintheoremparam} in the analytic
case.  We assume that $F$, of the form~\eqref{defFparam}, is a real
analytic map, that $\Ap>\ddp =\bp$ and that $\Kl{\leq}, R$ are real
analytic functions in the complex extension  $\Omega(\r,\gamma)
\times \Lambda(\gamma)$ of $\Vr\times \Lambda$ given by
\begin{align*}
\Omega(\r,\gamma)&:=\{x \in \C^n : \re x \in \Vr, \;\; \Vert \im x
\Vert \leq \gamma \Vert \re x \Vert \}
\\
\Lambda(\gamma) &:=\{\lambda \in \C^{n'} : \re \lambda \in \Lambda,
\;\;  \Vert \im \lambda \Vert \leq \gamma^2\}
\end{align*}
with the norm $\Vert \cdot \Vert$ in $\C^n$ as
\begin{equation*}
\Vert  x\Vert = \max\{ \Vert \re  x\Vert , \Vert \im  x\Vert \}.
\end{equation*}
We note that, if $ x\in \Omega(\r,\gamma)$ with $\gamma\le 1$, then
$\Vert  x\Vert = \max\{\Vert \re  x\Vert, \Vert \im x\Vert \} =
\Vert \re  x\Vert$. We will use this fact along this section without
special mention.

It is clear that the facts in Section~\ref{preliminaryfacts} also
hold in this new setting, as well as the reformulation of the
problem as a fixed point equation,  $\Kl{>} = \S{0}\circ
\FF(\Kl{>})$ (see \eqref{def:fpe:lambda}), with $\S{0}$ and $\FF$
defined in~\eqref{defS0param} and~\eqref{defFcalparam}. Therefore,
it is enough to prove that the fixed point equation has an analytic
solution.

The first thing we  need to control is the weak contraction of the
nonlinear terms in the analytic case. For that we first need to
prove an analogous result to Lemma~\ref{Rlemma} to decompose
$\Omega(\r,\gamma)$ properly. For that, for a given $\r>0$, we
consider $u>0$ and $a_0>0$ such that $a_0 u^{-\a}=\r$ and sequences
$a_k\in \RR$, $k\ge 0$, and $b_k\in \RR$, $k\ge 1$, satisfying
condition~\eqref{akbkcond}. Moreover, for any $\gamma\le 1$, we
introduce
\begin{equation}
\label{def:Omega_k}
\Omega_{k}= \left \{  x\in \Omega(\r,\gamma) : \Vert  x\Vert \in
I_{k}:=\left [ \frac{b_{k+1}}{(u+k+1)^{\a}} ,
\frac{a_{k}}{(u+k)^{\a}}\right ]\right \} = \{x \in
\Omega(\r,\gamma): \re x \in V_k\},
\end{equation}
where the sets~$V_k$ where introduced in~\eqref{defVk}.

\begin{lemma}
Let $ p$ be the homogeneous polynomial with respect to $(x,y)$
defined by~\eqref{defFparam}. Let $\Ra:\Omega(\r,\gamma) \to
\C^n$ be an analytic map such that $\Ra( x,\lambda) -  x-  p(
x,0,\lambda) =\OO(\Vert x \Vert^{N+1})$ uniformly in $\Lambda$.

Assume that there exists $\r_0>0$ such that $ p$ satisfies the
corresponding conditions in H$\lambda$ and, moreover, $\Ap>\bp$.

Then for any $\ka<\ap$ and $\kb>\bp$, there exist $\r_1,\gamma_1>0$
such that for any $\gamma\leq \gamma_1$ and $\r\leq \r_1$ the
following claims hold.
\begin{enumerate}
\item[(1)] If $ (x,\lambda) \in \Omega(\r,\gamma)\times \Lambda(\gamma)$,
\begin{equation*}
\Vert \Ra( x,\lambda) - x\Vert \leq \kb \Vert  x\Vert^N,\qquad
\Vert  \Ra ( x, \lambda) \Vert \leq \Vert  x\Vert \big ( 1-\ka \Vert
x\Vert^{N-1}\big ).
\end{equation*}
\item[(2)] The set $\Omega(\r,\gamma)$ is invariant by $\Ra$, that is,
$\Ra(\Omega(\r,\gamma))\subset \Omega(\r,\gamma)$.
\item[(3)] Let $\{a_{k}\},\{b_{k}\}$ be the two sequences defined in Lemma~\ref{Rlemma} and
$\Omega_k$ defined in~\eqref{def:Omega_k}. We have that
\begin{equation*}
\overline{\Omega(\r,\gamma)}\backslash \{0\} = \cup_{k=0}^{\infty}
\Omega_k \;\;\; \text{ and } \;\;\; \Ra(\Omega_k) \subset
\Omega_{k+1}.
\end{equation*}
Consequently, if $x\in \Omega_k$, then one has that
$$
\frac{\a}{b(u+k+1+j)}\big (1+ \OO(k^{-\beta})\big ) \leq \Vert
\Ra^j(x) \Vert^{N-1} \leq \frac{\a}{a(u+k+j)} \big (1+
\OO(k^{-\beta}) \big ).
$$
\end{enumerate}
\end{lemma}
\begin{proof}
We first note that, if $\chi(x,\lambda)$ is a real analytic
function,
\begin{align*}
\chi(x,\lambda) =& \chi(\re x ,\re \lambda) + \ii D \chi(\re x, \re
\lambda)[\im x, \im \lambda] \\&- \int_{0}^1 (1-\mu) D^2
\chi(x(\mu),\lambda(\mu)) [\im x, \im \lambda]^2 \dd \mu,
\end{align*}
with $x(\mu)=\re x+ \ii \mu \im x$ and $\lambda(\mu)=\re \lambda+
\ii \mu \im \lambda$.

 In addition, if $\chi,
\Dl \chi, \Dl^2 \chi =\OO(\Vert x \Vert^k)$, we have that, if
$\lambda \in \Lambda(\gamma)$:
\begin{equation}\label{exppolpanalytic}
\chi(x,\lambda)=\chi(\re x,\re \lambda) + \ii \Dx \chi(\re x,\re
\lambda) \im x + \gamma^2 \OO(\Vert x \Vert^k).
\end{equation}
The first item is a direct consequence of the above expression, for
$\chi(x,\lambda)=\Ra(x,\lambda)-x$, the
definition~\eqref{defconstantslambda} of $\ap, \bp$ and that
$\chi(x,\lambda)=p(x,0,\lambda)+ \OO(\Vert x \Vert^{N+1})$. The
second one is also a consequence of~\eqref{exppolpanalytic}. Indeed,
on the one hand, if $\gamma\leq \gamma_1$ and $\r \leq \r_1$,
writting $\Ra(x,\lambda)=x+\chi(x,\lambda)$,
\begin{align*}
\text{dist}(\re \Ra(x,\lambda), (\Vr)^{c}) &= \text{dist}(\re x +
p(\re x, 0, \re \lambda),(\Vr)^{c})- C \gamma^2 \Vert x \Vert^N -
C\Vert x \Vert^{N+1}\\ &\geq \Vert x \Vert^{N} (\CIn -
\OO(\gamma_1^2,\r_1)) \geq \frac{\CIn}{2} \Vert x \Vert^N,
\end{align*}
taking $\gamma_1, \r_1$ small enough. On the other hand,
\begin{align*}
\Vert \re \Ra(x,\lambda)\Vert \geq&  \Vert \re x\Vert (1- (\bp +
\OO(\gamma_1^2+\r_1^2)) \Vert x \Vert^{N-1}),
\\
\Vert \im \Ra (x,\lambda) \Vert  \leq& \Vert (\Id + \Dx p(\re x, 0, \re \lambda)) \im x \Vert + C\gamma^2 \Vert x \Vert^{N}\\
\leq & \gamma (1 - (\Ap+\OO(\gamma_1,\r_1)) \Vert x \Vert^{N-1} )
\Vert \re x\Vert
\end{align*}
and then if $\Ap>\bp$, taking $\r_1,\gamma_1$ small enough, $\Vert
\im \Ra(x,\lambda) \Vert \leq \gamma \Vert \re \Ra(x,\lambda) \Vert$
for any $\gamma \leq \gamma_1$.

Finally, the third item is a consequence of Lemma~\ref{Rlemma},
item~(2) and the fact that $x \in \Omega_k$ if and only if $\re x
\in V_k$ and $\Vert \im x \Vert \leq \gamma \Vert \re x \Vert =
\Vert x \Vert$.
\end{proof}

Let $U(\r,\gamma)=\Omega(\r,\gamma) \times \Lambda(\gamma)$. We
define the Banach space of analytic functions
$$
\mathcal{Z}_k = \{ h : U(\r,\gamma)\to \C^{n+m},\; \text{real
analytic, such that\;} \Vert h \Vert_k  <\infty\},
$$
where
$$
\Vert h \Vert_k = \sup_{(x,\lambda) \in U(\r,\gamma)} \frac{\Vert
h(x,\lambda)\Vert} {\Vert x\Vert^k}.
$$

From formula~\eqref{exppolpanalytic} applied to
$(D\FT)^{-1}(\Kl{\leq})(x)$ one can easily prove that
Lemma~\ref{lem:weakexpansion} holds true for $x\in
\Omega(\r,\gamma)$. As a consequence, if the scaling parameter is
small, bound~\eqref{bound:weakcontraction1} in
Lemma~\ref{lem:weakcontraction} is also true for $x\in \Omega_k$.

A proof analogous to the ones of Lemmas~\ref{lem:Ljonetoone}
and~\ref{lem:Skonlowregularity} for the continuous case proves that
a) the operator $\L{0}:\mathcal{Z}_{\k} \to \CC^{\omega}$, where
$\CC^{\omega}$ is the space of analytic functions on $U(\r,\gamma)$,
is continuous and one to one, and b) the linear operator
$\S{0}:\mathcal{Z}_{\k} \to \mathcal{Z}_{\k-N+1}$ is well defined
and bounded provided $\k-N+1-Ba^{-1}>0$. In addition, in the same
way as in Lemma~\ref{uniformlambda}, we obtain that there are bounds
of the norms of $\S{0}$ uniform in $\lambda\in \Lambda$.

Finally, one easily checks that the operator $\S{0} \circ \FF$ is
contractive on a suitable open ball of $\mathcal{Z}_{\ell-N+1}$. We
skip the details which are very similar to the ones
in~\cite{BFdLM2007}. This ends the proof in the analytic case.

It only remains to deal with the $\CC^{\Sigma_{s,\omega}}$ case. We
first note that, for any $\lambda\in \Lambda$ fixed, $K(\cdot,
\lambda)$ is analytic in $\Omega(\r,\gamma)$ for $\r,\gamma$ small
enough independent of $\lambda$. Moreover, since
$\CC^{\Sigma_{s,\omega}} \subset \CC^{\Sigma_{s,\infty}}$, given $F \in \CC^{\Sigma_{s,\omega}}$ we also
have that $K\in \CC^{\Sigma_{s,\infty}}$. Therefore, $K \in
\CC^{\Sigma_{s,\omega}}$.

\section{The Flow case}\label{sec:flowcase}
In this section we prove Theorem~\ref{maintheoremflow}, the
analogous result of Theorem~\ref{maintheoremparam} for flows.

The proof is performed in two steps in Sections~\ref{flowstomaps}
and~\ref{mapstoflows} below. The first step is to see that the
Poincar\'e map $F$ associated to the periodic vector field~$X$
in~\eqref{defX} has an invariant parametrization~$K$ and a reparametrization~$R$
satisfying the invariance equation $F\circ K = K\circ
R$. To do so we apply Theorem~\ref{maintheoremparam}. The second
step is to check that the invariance
condition~\eqref{invcondtheoremflow} for flows:
\begin{equation}
\label{invcondflowprova} \varphi(u;t,K(x,t,\lambda),\lambda)-
K(\psi(u;t,x,\lambda),u,\lambda) =0
\end{equation}
is satisfied for $K$, where $\varphi$ is the flow of $X$ and $\psi$ is the flow of a vector
field $Y$ on $\RR^n$ to be determined.

We assume that the vector field $ X\in \CC^{\DS}$ where in the
definition~\eqref{defCDS} of $\CC^{\DS}$ we take $z=(x,y)$ and
$\mu=(t,\lambda)$. We will denote by $D_z$ and $D_{\mu}$ the
derivatives with respect to these variables.

\subsection{From flows to maps}\label{flowstomaps}
Assume that $X\in \CC^{\DS}$ is a $T$-periodic vector field of the
form~\eqref{defX}
\begin{equation}\label{defXproof}
X(x,y,t,\lambda) = \left ( \begin{array}{c}  p(x,y,\lambda) +
f(x,y,t,\lambda) \\  q(x,y,\lambda)+
g(x,y,t,\lambda)\end{array}\right ),
\end{equation}
that $p$ satisfies H$\lambda$ and let $\Kl{\leq},Y\in
\CC^{\Sigma_{s^{\leq}, \rfm}}$ satisfying items (a), (b) and (c) in
Theorem~\ref{maintheoremflow}. In particular we have that condition
\eqref{HIKlflow} is satisfied, namely,
$$
X(\Kl{\leq}(x,t,\lambda),t,\lambda) - D \Kl{\leq}(x,t,\lambda)
Y(x,\lambda) - \partial_t \Kl{\leq}(x,t,\lambda) =\OO(\Vert x
\Vert^{\k})
$$
for a given $\k$ such that
$\ri<\k\leq r$.

We denote by $\varphi(u;t,x,y,\lambda)$ and $\psi(u;t,x,\lambda)$
the associated flows of $\dot{z}=X(z,t,\lambda)$, $z=(x,y)$, and
$\dot{x} =Y(x,\lambda)$ respectively. For $t\in \RR$ and $u\in
[t,t+T]$,
\begin{equation}\label{invcondordlflow}
\varphi(u;t,\Kl{\leq}(x,t,\lambda),\lambda)-
\Kl{\leq}(\psi(u;t,x,\lambda),u,\lambda) =\OO(\Vert x \Vert^{\k}),
\end{equation}
uniformly in $u,\lambda$. The proof is a
consequence of Gronwall's lemma, \eqref{invcondtheoremflow} and the $\CC^0$
dependence of $ \Kl{\leq}$ with respect to $t$.

We introduce the
Poincar\'e maps $F(x,y,t,\lambda) = \varphi(t+T;t,x,y,\lambda)$  and $R(x,\lambda) =
\psi(T;0,x,\lambda)=\psi(t+T;t,x,\lambda)$.
Applying~\eqref{invcondordlflow} to $u=t+T$, we obtain that
\begin{equation}\label{invaproxflowproof}
F(\Kl{\leq}(x,t,\lambda),t,\lambda) -
\Kl{\leq}(R(x,\lambda),t,\lambda) = \OO(\Vert x \Vert^{\k}).
\end{equation}

We want to apply Theorem~\ref{maintheoremparam}, so we have to check
the setting and hypotheses of that theorem for~$F$.

By Hypothesis HP and, since $X$ is of the form~\eqref{defXproof},
for any $(x,y) \in B_{\rho}$, we have $\Vert X(x,y,t,\lambda) \Vert
\leq C \rho^N$. Then, on the one hand, the flow
$\varphi(u;t,x,y,\lambda)$ is well defined for $u\in [t,t+T]$ if
$(x,y)\in B_{\varrho}$ and $\varrho$ is small enough. On the other
hand, by Gronwall's lemma,
\begin{equation}\label{boundvarphiflow}
\Vert \varphi(u;t,x,y,\lambda) \Vert \leq C \Vert (x,y) \Vert,\qquad
(u,x,y,\lambda)\in [t,t+T]\times B_{\varrho}\times \Lambda.
\end{equation}

Now we check that $F$ has the form~\eqref{defFparam}. Applying Taylor's theorem to
$\varphi(u;t,x,y,\lambda)$, with respect to $u$:
\begin{equation*}
\begin{aligned}
F(x,y,\lambda,t) =& \varphi(t+T;t,x,y,\lambda) = \left
(\begin{array}{c} x \\ y \end{array}\right ) + T \left (
\begin{array}{c}   p(x,y,\lambda) + f(x,y,t,\lambda) \\
q(x,y,\lambda)+g(x,y,t,\lambda)\end{array}\right ) \\ & +
 \int_{t}^{t+T} (t+T-u) D_z X(\varphi(u; t,x,y,\lambda),u,\lambda)  X(\varphi(u; t,x,y,\lambda),u,\lambda) \dd u\\
&+\int_{t}^{t+T}(t+T-u) D_t X(\varphi(u; t,x,y,\lambda),u,\lambda)
\dd u.
\end{aligned}
\end{equation*}
Using bound~\eqref{boundvarphiflow} in the above formula for the
Poincar\'{e} map $F$, we see that $F$ has the form \eqref{defFparam}
and satisfies H$\lambda$ for any fixed $t\in \RR$ since $p$ does
not depend on $t$. Moreover, using that $f$ and $g$ are periodic
with respect to $t$, $D_{(x,y)}^2 f,
D_{(x,y)}^2 g$ are bounded and they satisfy Hypothesis HP. We also have that the remainder
$(\tilde{f},\tilde{g})= F- \Id -(Tp,Tq)$ satisfies Hypothesis HP.

Concerning the items of Theorem~\ref{maintheoremparam}, (a) follows
from the hypotheses and general regularity results for flows, (b)
for $\Kl{\leq}$ also follows from hypothesis and (c) have already
been obtained in~\eqref{invaproxflowproof}.

It remains to check that $R(x,\lambda)= \psi(T;0,x,\lambda)$
satisfies (b) in Theorem~\ref{maintheoremparam}. Namely, defining
$\Delta R(x,\lambda):=R(x,\lambda) - x-Tp(x,0,\lambda)$ we have to
check that, uniformly in $\lambda\in \Lambda$,
$$
D_{\lambda}^j \Dx^i \Delta R(x,\lambda) = \OO(\Vert x
\Vert^{N+1-i}),\qquad (i,j)\in \CC^{\Sigma_{s^{\leq},\rfm}}.
$$
These bounds are consequence of the following elementary result, whose proof we omit.
\begin{lemma}
Let $Z:\Vro\times \Lambda \to \RR^n$ be a vector field of the form
$Z(x,\lambda) = Z_{0}(x,\lambda) + Z_1(x,\lambda)$. Let
$\chi(t;x,\lambda)$ be its flow.

Let $\sigma\geq 0$ and $\nu\geq 2$. Assume that  $Z_0, Z_1\in
\CC^{\Sigma_{\sigma,\nu}}$ and that there exist $l>k\geq 2$ such
that, for all $(i,j) \in \CC^{\Sigma_{\sigma,\nu}}$:
\begin{equation*}
\Dl^i \Dx^j Z_0(x,\lambda) = \OO(\Vert x
\Vert^{k-j}),\qquad \Dl^i \Dx^j Z_1(x,\lambda) =\OO(\Vert x
\Vert^{l-j})
\end{equation*}
uniformly in $\lambda\in \Lambda$.

Then for any $u_0>0$ there exists $\r$ small enough such that, if
$x\in V_{\r/2}$ and $u\in [0,u_0]$, the flow $\chi$ satisfies
$\chi(u;x,\lambda)= x + u Z_0(x,\lambda) +
\widetilde{Z}_1(u,x,\lambda)\in \Vr$ with
$$
\Dl^i \Dx^j \widetilde{Z}_1(u,x,\lambda) =\OO(\Vert x
\Vert^{k+1-j}),\qquad (i,j)\in \Sigma_{\sigma,\nu}
$$
uniformly in $(u,\lambda)\in [0,u_0]\times \Lambda$.
\end{lemma}

Summarizing, let $\max\{\ri,\ril\} <\k\leq r$, $\Kl{\leq}$ and $Y$
be such that~\eqref{invcondordlflow} holds true. Applying
Theorem~\ref{maintheoremparam} to the Poincar\'{e} map
$F(x,y,t,\lambda)= \varphi(t+T;t,x,y,\lambda)$ with
$R(x,\lambda)=\psi(t+T;t,x,\lambda)$, we obtain a solution
$K=\Kl{\leq}+\Kl{>}\in \CC^{\Sigma_{s^>,\rv}}$ of the invariance
condition
\begin{equation}
\label{invcondflowparam} F(K(x,t,\lambda),t,\lambda)
=K(\psi(t+T;t,x,\lambda),t,\lambda)
\end{equation}
with $\Kl{>}(x,t,\lambda) = \OO(\Vert x \Vert^{\k-N+1})$ uniformly
in $\lambda$. Moreover, by the uniqueness of the solution, $\Kl{>}$
(and consequently $K$) is periodic with respect to $t$.

\subsection{From maps to periodic flows}\label{mapstoflows}
In this section we prove that the parametrization $K$ found in the
previous Section~\ref{flowstomaps} satisfies the invariance
condition~\eqref{invcondflowprova} for flows. To avoid cumbersome
notations, in this section we will skip the dependence on $\lambda$.

Using the properties of general solutions of vector fields, the definitions
of~$F$ and~$R$ and~\eqref{invcondflowparam} we obtain
$$
K(x,s)=\varphi(s;s+T,K(R(x),s)),\qquad R(\psi(s;t,x)) =
\psi(s;t,R(x)).
$$
We define
\[
\mathcal{K}_s(x,t) = \varphi(t;s,K(\psi(s;t,x),s)).
\]
We have $\mathcal{K}_t(x,t)=K(x,t)$ and
\begin{align*}
F(\mathcal{K}_s(x,t),t) &= \varphi(t+T;s,K(\psi(s;t,x),s)) =
\varphi(t+T;s+T, K(\psi(s;t,R(x)),s))\\ &=
\varphi(t;s,K(\psi(s;t-T,x,s))), \\
\mathcal{K}_s(R(x),t))&=\varphi(t;s,K(\psi(s;t,R(x),s)))=
\varphi(t;s,K(\psi(s;t-T,x),s)).
\end{align*}
Consequently, $\mathcal{K}_s(x,t)$ satisfies the invariant
condition~\eqref{invcondflowprova} for any $s$.

Applying again Taylor's theorem,
\begin{align*}
\mathcal{K}_s(x,t) =&\varphi(t;s,K(\psi(s,t,x),s)) =
\varphi(t;s,K^{\leq}(\psi(s;t,x),s)) \\ &+ \int_{0}^1
D\varphi(t;s,K^{\leq}(\psi(s;t,x),s) +
 w K^>(\psi(s;t,x),s) )  K^>(\psi(s;t,x),s)\, d w
\end{align*}
and, applying equality \eqref{invcondordlflow} to $\psi(s;t,x)$,
\begin{align*}
\mathcal{K}_s(x,t) -& K^{\leq}(x,t) =\OO(\Vert x \Vert^{\ell})
+ \int_0^1 D K^{\leq}(x+w(\psi(s;t,x)-x),t)[\psi(s;t,x)-x]\,dw
\\&+ \int_{0}^1 D\varphi(t;s,K^{\leq}(\psi(s;t,x),s) +
 w K^>(\psi(s;t,x),s) )  K^>(\psi(s;t,x),s) \, dw  .
\end{align*}
Therefore, since $\psi(s;t,0) = 0$ and $\psi(s;t,x) = x + \OO(\|x\|^N)$, we have that $\mathcal{K}_s(x,t) -
K^{\leq}(x,t) =\OO(\Vert x\Vert^{\ell-N+1})$ and this implies, by the
uniqueness statement in Theorem~\ref{maintheoremparam} that
$\mathcal{K}_s(x,t) = K(x,t)$. Then
$$
K(\psi(s;t,x),s)= \varphi(s;t,\mathcal{K}_s(x,t)) =
\varphi(s;t,K(x,t))
$$
and the proof is complete.

\section{Acknowledgments}
I.B and P.M. have been partially supported by the Spanish Government
MINECO-FEDER grant MTM2015-65715-P and the Catalan Government grant 2014SGR504.
The work of E.F. has been partially supported by the Spanish
Government grant MTM2013-41168P and the Catalan Government grant
2014SGR-1145.

\section*{References}
\bibliography{references}
\bibliographystyle{alpha}
\end{document}